\documentclass[11pt,reqno]{amsart}
\usepackage[margin=3.6cm]{geometry}
\usepackage{graphicx}
\usepackage{amssymb}
\usepackage{bbm}
\usepackage{color}

\newcommand{\R}{{\mathbb R}}
\newcommand{\N}{{\mathbb N}}

\newcommand{\C}{{\mathbb C}}
\newcommand{\Z}{{\mathbb Z}}
\newcommand{\h}{h}
\newcommand{\eo}{\varepsilon_0}

\newcommand{\ep}{\varepsilon}
 \newcommand{\repi} {\Re}
 \newcommand{\impi} {\Im}
  \newcommand{\Exp} {\mathrm{Exp}}
\newcommand{\cs}{$\clubsuit$}

 \renewcommand{\Im} {\mathrm{Im\,}}
 \renewcommand{\Re} {\mathrm{Re\,}}
 \DeclareMathOperator{\inj}{inj_M}

\newtheorem{theorem}{Theorem}
\newtheorem{corollary}[theorem]{Corollary}
\newtheorem{proposition}[theorem]{Proposition}
\newtheorem{lemma}[theorem]{Lemma}

\theoremstyle{definition}
\newtheorem{definition}{Definition}
\newtheorem{remark}{Remark}

\def\reals{{\mathbb R}}

\def\p{\partial}
\def\half{\frac{1}{2}}
\def\tchi{\tilde{\chi}}
\def\Ci{{\mathcal C}^\infty}
\def\tpsi{\tilde{\psi}}
\def\supp{\mathrm{supp}\,}
\def\O{{\mathcal O}}
\def\tzeta{\tilde{\zeta}}
\def\quarter{\frac{1}{4}}

\begin{document}

\title[Restriction bounds for Neumann eigenfunctions of planar  domains]{Non-concentration  and restriction bounds for Neumann eigenfunctions of piecewise $C^{\infty}$  bounded planar  domains}

\author[H. Christianson]{Hans Christianson}
\address{Department of Mathematics, UNC Chapel Hill} \email{hans@math.unc.edu}

\author[J. Toth]{John  A. Toth}
\address{Department of Mathematics and
Statistics, McGill University, 805 Sherbrooke Str. West, Montr\'eal
QC H3A 2K6, Ca\-na\-da.} \email{jtoth@math.mcgill.ca}

\maketitle

\begin{abstract}
Let $(\Omega,g)$ be a piecewise-smooth, bounded convex domain in $\R^2$ and consider $L^2$-normalized Neumann eigenfunctions $\phi_{\lambda}$ with eigenvalue $\lambda^2$ and $u_{\lambda}:= \phi_{\lambda} |_{\partial \Omega}$  the associated Dirichlet data (ie. boundary restriction of $\phi_{\lambda}$). Our first main result (Theorem \ref{T:non-con}) is a small-scale {\em non-concentration} estimate: We prove that for {\em any} $x_0 \in \overline{\Omega},$ (including boundary corner points)  and any $\delta \in [0,1),$ 
$$ \| \phi_h \|_{B(x_0,\lambda^{-\delta})\cap \Omega} = O(\lambda^{-\delta/2}).$$ 
Our subsequent results involve applications of the nonconcentration estimate to upper bounds for $L^2$ restrictions of  boundary eigenfunctions that are valid up to boundary corners. In particular, in Theorem \ref{dirichlet} we prove that for any {\em flat} boundary edge $\Gamma$ (possibly including corner points), the boundary restrictions $u_h:= \phi_h |_{\partial \Omega}$ satisfy the bounds
$$ \|u_{\lambda} \|_{L^2(\Gamma)}  = O_{\epsilon}(\lambda^{1/4 + \epsilon}),$$
for any $\epsilon >0.$
The exponent $1/4$ is sharp and the result improves on  the $O(\lambda^{1/3})$ universal $L^2$-restriction bound for Neumann
eigenfunctions due to Tataru \cite{Ta}.
The $O(\lambda^{1/4})$ -bound is also an extension to the boundary
(including corner points) of well-known interior $L^2$ restriction
bounds of Burq-Gerard-Tzvetkov \cite{BGT} along totally-geodesic hypersurfaces. 
\end{abstract}\ \\
\section{introduction}

\subsection{Non-concentration estimates} Let $\Omega \subset \R^2$ be a bounded, convex planar domain with boundary $\partial \Omega.$ We say that $\Omega$ is {\em piecewise smooth} if the boundary  $\partial \Omega = \cup_{j=1}^N \Gamma_j$ such that there exist defining functions $f_j \in C^{\infty}(\R^2; \R)$ with
$$\Gamma_j \subset \{ x \in \R^2; f_j(x) = 0, \,\,\, df_j (x) \neq 0 \}.$$

 We refer to the $\Gamma_j$'s as the boundary edges. We say that a piecewise-smooth $\Omega$ is a {\em domain with corners} if the
$\Gamma_j$'s are diffeomorphic to closed intervals with $\Gamma_j \cap \Gamma_{j+1} = c_j  \in \R^2; j=1,...,,N,$ such that at   $c_j = \Gamma_j \cap \Gamma_{j+1},$ 
$$ \text{rank} \, ( df_j(c_j), df_{j+1}(c_j) ) = 2.$$
We refer to ${\mathcal C}:= \{ c_j \}_{j=1}^{N}$ as the set of {\em corner points} and the rank condition on the defining functions at the $c_j$'s ensures that the boundary edges $\Gamma_j; j=1,...,N$ intersect at non-zero angles. We denote the angle at a corner $c_j$ by $\alpha_j.$

A fundamental issue regarding eigenfunctions involves their concentration properties (or lack thereof) on small  balls with radius that depends on the semiclassical parameter $h$ as $h \to 0^+.$  

Let $(M,g)$ be a compact Riemannian manifold without boundary and $\phi_h$  a Laplace eigenfunction with eigenvalue $h^{-2}$. Then, as pointed  out in \cite{So}, using the explicit asymptotic formula for the half-wave operator $e^{it \sqrt{- \Delta}}: C^{\infty}(M) \to C^{\infty}(M)$ it is not hard to prove that there exists $C_M>0$ such that 

\begin{equation} \label{soggebound}
\| \phi_h \|_{L^2(B(r))}^2 = O(r) \| \phi_h \|_{L^2(M)}^2, \quad  \forall r \geq C_M h\end{equation}\

We refer to estimates of the form (\ref{soggebound}) as {\em non-concentration} bounds.
The example of highest weight spherical harmonics  on the round sphere (see Remark \ref{gaussian} below) shows that (\ref{soggebound}) is, in general, sharp. However, in certain cases, one expects improvements. For instance, in the case of surfaces with non-positive curvature, one can get logarithmic improvements  \cite{So} (see also \cite{Han, HR}).

Since the proof of (\ref{soggebound}) uses the wave parametrix in a crucial way, the extension to manifolds with boundary is  non-obvious since the behaviour of  the  wave operators near $\partial \Omega$ is much more complicated than in the boundaryless case.
The first  main result of this paper (Theorem \ref{T:non-con}) is an
extension of the bounds in (\ref{soggebound}) to Neumann
eigenfunctions of a bounded piecewise-smooth, convex planar
domain. Our basic method of proof here is entirely stationary and uses
a Rellich commutator argument rather than wave methods.
 This stationary approach allows us to deal with {\em both} boundaries and corners as well.  We note that our result below holds right up to the boundary, including corner points.



\begin{theorem}
  \label{T:non-con}
  Let $\Omega \subset \reals^2$ be a piecewise $C^\infty$ bounded, convex
  domain 
  and consider the semiclassical Neumann eigenfunction problem:
  \begin{equation*}
    \begin{cases}
      -h^2 \Delta \phi_h(x) = \phi_h(x),  \quad x \in \Omega, \\
      \p_\nu \phi_h  |_{\p \Omega} = 0, \\
      \| \phi_h \|_{L^2 ( \Omega) } = 1,
    \end{cases}
  \end{equation*}
  where $\p_\nu$ is the outward pointing normal derivative.
Let 
  $p_0 \in
  \overline{\Omega}$ be a point in $\Omega$ or on the boundary.  Then
  for any $0 \leq \delta <1$, 
  
  \begin{equation}
    \| \phi_h \|^2_{L^2 ( B(p_0 , h^{\delta}))} = O(h^{\delta}).
  \end{equation}

\end{theorem}\

\begin{remark}
The theorem is also true for Dirichlet eigenfunctions, but the proof
in that case is much easier.  We will point out the small
modifications necessary to the proof of Theorem \ref{T:non-con} in the proof.

\end{remark}


  






\begin{remark}
\label{R:h-delta}
  
  As the proof will indicate,
  the bound for eigenfunction $L^2$ mass in
  a ball of radius $h^\delta$, $0 \leq \delta \leq 1/2$, is relatively straightforward.
The cases where $0 \leq \delta <1/2$ follow immediately from the argument
proving the $\delta = 1/2$ case.  
  To
  improve to $1/2 < \delta < 1$, we use the estimate for $\delta =
  1/2$ to bootstrap to $\delta = 2/3$.  Then an induction step proves
  that for any integer $k>0$, the result is true for $\delta = 1-1/3k$.

\end{remark}





\begin{remark} \label{gaussian}  The estimate in Theorem \ref{T:non-con} is sharp. To see this,  let $\gamma \subset \Omega$ be a geodesic segment with $\gamma = \{ (x',x_n=0) \in \Omega; |x'| < \delta \}$ and $U = \{(x',x_n); |x_n| < \delta \}$ be a tubular neighbourhood, where $(x',x_n): U \to \R^n$ are Fermi coordinates. An $L^2$-normalized Gaussian beam localized on $\gamma$ is  of the form
$$ u_h(x) = (2\pi h)^{-1/4} e^{-x_n^2/h} \,e^{ i x'/h} \, ( \,  a(x',x_n;) + O(h) \,); \, a \in C^{\infty}(U), \,\, |a(x)| >0, \,\, x \in U.$$
It follows that
$$ \|u_h \|_{B(0,h^{1/2})}^2 \sim \int_{|x_n| \leq h^{1/2}} \int_{|x'| < h^{1/2}} |u_h(x)|^2 \, dx \sim h^{1/2}.$$

Consider the case where $\Omega = \{(x,y); \frac{x^2}{a^2} +
\frac{y^2}{b^2} = 1, y \geq 0 \}$  where $a>b>0$ is the half-ellipse
and let $\phi_h$ be an $L^2$-normalized Neumann eigenfunction. It is
well-known (see \cite{TZ1} section 2.2)  that there exists a subsequence of eigenfunctions that are Gaussian beams along the major axis $\{ (x,0); - a \leq x \leq a \}.$
Consequently, the estimate in Theorem \ref{T:non-con} is sharp in
general. In the special case where the $u_h$ satisfy polynomial
 small-scale quantum ergodicity (SSQE) on a scales $h^{1/2},$ since the volume of a ball of radius $h^{1/2}$ is
$h,$ one putatively expects a bound of $O(h)$ on the RHS in Theorem
\ref{T:non-con}. Unfortunately, to our knowledge, there are no
rigorous results  on polynomial SSQE known at present,
although logarithmic SSQE was proved by X. Han \cite{Han}.

\end{remark}

\subsection{Restriction bounds along geodesic boundary components}

Our second theorem deals with bounds for eigenfunction restrictions. This problem has been the focus of many papers over the past decade and has deep and interesting connection to the study of the asymptotics of eigenfunction nodal set and, in particular,  intersection bounds \cite{CTZ,DZ,ET, G, GRS, HZ, JJ, JJZ, TZ1, TZ2, TZ3}

\noindent Specifically, in the case of $L^2$-normalized Neumann eigenfunctions with $\|  \phi_h \|_{L^2(\Omega)} =1,$ the $h$-Sobolev estimates give
$$ \|u_h \|_{L^2(\partial \Omega)} = O( h^{-1/2} \| \phi_h \|_{L^2(\Omega)} ) = O(h^{-1/2}),$$
and so the bounds in Theorem \ref{dirichlet} are polynomial
improvements over the (automatic) $h$-Sobolev estimates.

It was proved by Tataru \cite{Ta} that for bounded
domains with smooth  boundary, the elementary Sobolev bounds can be improved and the Dirichlet traces $u_h$ of 
Neumann eigenfunctions satisfy 
\begin{equation} \label{Tataru}
\| u_h \|_{L^2(\partial \Omega)} = O(h^{-1/3}).
\end{equation}
For general smooth boundaries,  (\ref{Tataru}) is sharp and is saturated by whispering gallery modes on the disc \cite{HT}.

 \begin{remark} We emphasize that the $O(h^{-1/3})$ boundary estimate in (\ref{Tataru}) should {\em not} be confused with the restriction bound 
$\| u_h \|_{L^2(H)} = O(h^{-1/6})$ (see \cite{BGT}) in the case where $H$ is an {\em interior} curve segment with positive curvature. Roughly speaking, the latter is consistent with the decay of semiclassically rescaled Airy functions in classically {\em allowable} regions, whereas the former corresponds to Airy decay in the classically {\em forbidden} region. This is an interesting contrast that we hope to address in detail elsewhere.  \end{remark}

\noindent As an application of the non-concentration estimate (see Theorem \ref{T:non-con}),  we consider the problem of deriving sharp $L^2$ restriction bounds for Neumann eigenfunctions along geodesic boundary segments {\em up to corners}. When $\Gamma \subset \mathring{\Omega}$ is a {\em strictly interior} geodesic segment, the bound $\|u_h \|_{L^2(\Gamma)} = O(h^{-1/4})$  follow from the general $L^p$  restriction bounds of Burq-Gerard-Tzvetkov \cite{BGT}. The main novel feature of Theorem \ref{dirichlet} is the extension of the interior $O(h^{-1/4})$ geodesic restriction bound to the boundary (including corners) with a  loss of $ h^{-\epsilon}$ for any $\epsilon >0,$ thereby improving on the universal Tataru $O(h^{-1/3})$-bound along geodesic boundary edges. 


To state our second main result, we will need the following 
\begin{definition} \label{admissible}. Let $\Omega$ be a piecewise $C^{\infty}$ planar domain with corners and $\Gamma_j \subset \partial \Omega$ be a flat edge. We say that $\Omega$ is {\em admissible} if for any adjacent interior angle $\alpha_j$ to a corner $c_j \in \Gamma_j$ the following conditions are satisfied:
$$ (i) \,\,\,\big\{ L_j(t):= c_j +  e^{2 i (\pi - \alpha_j)} t, \,\, \, t \in \R \big\} \cap {\mathcal C} = c_j, \quad \text{when} \,\, \alpha_j \in (\pi/2,\pi),$$
$$ (ii) \,\,\,  \big\{ L_j(t):=  c_j +  e^{2 i \alpha_j} t,  \,\, \, t \in \R \big\} \cap {\mathcal C} = c_j, \quad \text{when} \,\, \alpha_j \in (0,\pi/2].$$
\end{definition}

Roughly speaking, admissibility amounts to the condition that  glancing rays to a flat boundary edge $\Gamma_j$ starting from a corner point $c_j \in \Gamma_j$ reflect off an adjacent edge and do not hit any other corner point. The condition in terms of angles is stated slightly differently depending on whether the corner angle at $c_j$ is obtuse as in (i), or acute as in (ii) (see Figures 4 and 5 where the lines $L_j$ are pictured).

\begin{theorem} \label{dirichlet}
Let $\Omega \subset \R^2$ be a convex, bounded domain with corners and boundary  $\partial \Omega$ that is admissible in the sense of Definition \ref{admissible}.  Let $\Gamma_j$ be a totally geodesic (i.e. flat) boundary segment and  $u_h:= \phi_h |_{\partial \Omega}$ be  the Dirichlet traces of the Neumann eigenfunctions, $\phi_h.$  Then, for any $\epsilon >0$ there exists $C_{\epsilon}>0$ and $h_0(\epsilon)>0$ such that for $h \in (0,h_0(\epsilon)],$

$$  \| u_h \|_{L^2(\Gamma_j)}  \leq C_{\epsilon} h^{-1/4 - \epsilon}.$$

\end{theorem} 


Let  $N(h): C^{0}(\partial \Omega) \to C^{0}(\partial \Omega)$ be the double layer potential corresponding to the free Green's function $G(q,q',h)$ of the Helmholtz equation $(-h^2 \Delta_q - 1) G(q,q',h) = \delta(q-q')$ in $\R^2$ (see section \ref{restriction}). 
Setting aside technicalities, the rough idea of the proof of Theorem \ref{dirichlet} involves combining a suitably microlocalized version of the boundary jumps equations $u_h = N(h) u_h$ with Sobolev restriction applied to the non-concentration results in Theorem \ref{T:non-con} to bound the $L^2$-mass of eigenfunction restrictions near corners. Specifically, in sections \ref{restriction} and \ref{defn} we show that when $\Gamma_j$ is a flat boundary edge, there exist certain boundary operators $N^{\mathcal D}_j(h), N^{\mathcal G}_j(h): C^{0}(\partial \Omega) \to C^{0}(\Gamma_j)$ so that with $u^{j}_h := {\bf 1}_{\Gamma_j} u_h,$

\begin{equation} \label{gutsofthm2}
 u^{j}_h = N_j^{\mathcal G}(h) u_h + N_j^{\mathcal D}(h) u_h  + O(1)_{L^2}. \end{equation}\
 
 In (\ref{gutsofthm2}), $N^{\mathcal G}_j(h)$ is an $h$-Fourier integral operator ($h$-FIO) with canonical relation graph $\beta,$ where $\beta: \mathring{B}^*\partial \Omega \to \mathring{B}^*\partial \Omega$ is the billiard map. We refer to it as the {\em geometric} term in (\ref{gutsofthm2}).  On the other hand, the operator $N^{\mathcal D}_j(h)$ corresponds to diffraction at the corners bounding the edge $\Gamma_j$ and so, we refer to the latter as the {\em diffractive} term.  Not suprisingly, the main contribution indeed comes from the geometric term and by successive reflections in the sides adjacent to a flat side $\Gamma_j$ (see also Figures 4 and 5),  in section \ref{defn} we show that under the admissibility assumption in Definition \ref{admissible}, for any $\epsilon >0,$
 $$ \|N^{\beta}_j(h) u_h \|_{L^2(\Gamma_j)} = O_{\epsilon}(h^{-1/4-\epsilon}).$$
 A more straightforward argument in subsection \ref{diffraction} also shows that  the diffractive term satisfies
 $$\| N^{\mathcal D}_j(h) u_h \|_{L^2} = O_{\epsilon}(h^{-1/4-\epsilon}).$$ 
Inserting these bounds in (\ref{gutsofthm2}) then proves Theorem \ref{dirichlet}. 

We note that Theorem \ref{dirichlet} holds up to corners and also
 that the $-1/4$-power in Theorem \ref{dirichlet} is sharp. Indeed, it is not hard to show that the Gaussian beam eigenfunctions associated with the major axis of the semi-ellipse  saturate the bound in (\ref{dirichlet}).
Using the $L^p$ bound results for polygonal domains in  \cite{BFM},  Matt Blair (personal communication) has recently proved  that for {\em polygonal domains} one can dispense with the $h^{-\epsilon}$-correction in Theorem \ref{dirichlet}. However, at the moment,  for general domains with corners, we cannot rule out additional  background diffractive effects in the restriction bounds near corners leading to a possible $h^{-\epsilon}$ loss. We hope to address this question in detail elsewhere.

 The results in both Theorems \ref{T:non-con}  and \ref{dirichlet} should extend to the general Riemannian setting of compact manifolds with boundary. However, the latter case presents additional complications that we hope to address elsewhere using more sophiscated 2-microlocal machinery. 

Throughout the paper, given a set $X$ and two non-negative functions $f,g: X \to \R^+$, the notation $f \lessapprox g$ means that there exists a constant $C>0$ such that $f(x) \leq C g(x)$ for all $x \in X.$ Similarily, the notation  $f \approx g$ means that both $f \lessapprox g$ and $g \lessapprox f.$ In addition, we will use the notation $ O(h^{-\alpha -0})$  as a convenient shorthand for $O_{\epsilon}(h^{-\alpha - \epsilon})$ for any $\epsilon >0.$

 \section{One point non-concentration in shrinking balls}

Before jumping into the details of the proof of Theorem
\ref{T:non-con}, let us sketch the main intuitive idea.  Suppose $p_0$
is a point on a flat side of $\p \Omega$, and assume for simplicity
that $\p \Omega = \{ y = 0 \}$ locally near $p_0$ and $p_0 = (0,0)$.
Let $\chi$ be a smooth monotone bounded function, $\chi(y) \sim h^{-1/2} y$
in an $h^{1/2}$ neighbourhood of $y = 0$, and constant outside a
neighbourhood of size $M h^{1/2}$ for large $M$.  Then $\chi'(y)$ is a
bump function supported on $-M h^{1/2} \leq y \leq M h^{1/2}$ with
$\chi'(y) \sim h^{-1/2}$ on $-h^{1/2} \leq y \leq h^{1/2}$.  We then
apply a Rellich commutator type argument:
\begin{align*}
  \int_\Omega & ([-h^2 \Delta -1, \chi \p_y ] \phi_h) \phi_h dV \\
  & = -2
  \int_\Omega \chi'(y) (h^2 \p_y^2 \phi_h ) \phi_h dV + \O(1) \\
  & = 2
  \int_\Omega \chi' | h \p_y \phi_h |^2 dV + \O(1) \\
  & \gtrapprox h^{-1/2} \int_{B((0,0),h^{1/2}) \cap \Omega} | h \p_y
  \phi_h |^2 dV - \O(1).
\end{align*}
Computing the commutator explicitly shows the left hand side is
bounded.  
Adding a similar computation with $\chi(x) \p_x$ and 
rearranging would prove the theorem (for $\delta = 1/2$).  
A suitable $h$-dependent cutoff allows us to integrate by parts to go
from estimating $\| h \nabla \phi_h \|_{L^2(B)}$ from below to
estimating $\| \phi_h \|_{L^2(B)}$ from below.  Here the 
$O(1)$ error term is from differentiating $\chi$ twice: $h \chi'' =
O(1)$.  This allows us to prove the theorem at the limiting  scale
$h^{1/2}$.  The tricky part is using the $\delta = 1/2$ result to 
prove the result for $\delta = 2/3$, and then apply an induction
argument to get the result for any $\delta<1$.   
Of course
in this little sketch, the $\O(1)$ terms from integrating by parts,
etc. are actually very subtle in the case of Neumann eigenfunctions,
and the bulk of the proof is dealing with these ``lower order terms''.

\subsection{Proof of Theorem \ref{T:non-con}}
\begin{proof}[Proof of Theorem \ref{T:non-con}]
For simplicity, we will assume throughout the proof that the
eigenfunctions $\phi_h$ are real-valued, however the general case can
be obtained by taking real/imaginary parts where necessary.  This is
not a problem since the quantity we eventually want to compute is
real-valued.  However, assuming $\phi_h$ is real-valued does save us
some notational headaches.

The proof will proceed by looking at boundary pieces away from corners
and at corners separately, although the proof for corners has much in
common with smooth sides.


 The proof has several steps.  First we establish the result for
$\delta = 1/2$.  The proof for $0 \leq \delta < 1/2$ is similar (and easier), so we
omit the details.  Then we use the $\delta = 1/2$ estimate to
bootstrap the $\delta = 2/3$ estimate.  Again, for $1/2 < \delta <
2/3$, the proof is the same as for $\delta = 2/3$ (but again easier).
Our final step is an induction to prove that for any integer $k>0$ the
result is true for $\delta = 1-1/3k$.

We will employ a number of convenient cutoff
functions.

  Let $\tchi(s) \in \Ci ( \reals)$ satisfy the following conditions:
  \begin{itemize}

  \item $\tchi$ is odd,

  \item $\tchi' \geq 0$,

    \item $\tchi(s) \equiv -1$ for $s \leq -3$ and $\tchi(s) \equiv 1$ for $s
      \geq 3$,

    \item
      $\tchi(-1) = -1/2$ and $\tchi(1) = 1/2$,

    \item $\tchi(s) = \frac{s}{2}$ for $-1 \leq s \leq 1$.

  \end{itemize}
See Figure \ref{F:tchi-2} for a picture.

    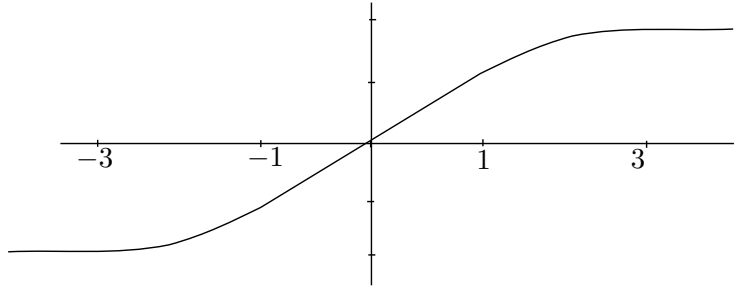
\begin{figure}
\hfill
\centerline{
\begingroup%
  \makeatletter%
  \providecommand\color[2][]{%
    \errmessage{(Inkscape) Color is used for the text in Inkscape, but the package 'color.sty' is not loaded}%
    \renewcommand\color[2][]{}%
  }%
  \providecommand\transparent[1]{%
    \errmessage{(Inkscape) Transparency is used (non-zero) for the text in Inkscape, but the package 'transparent.sty' is not loaded}%
    \renewcommand\transparent[1]{}%
  }%
  \providecommand\rotatebox[2]{#2}%
  \newcommand*\fsize{\dimexpr\f@size pt\relax}%
  \newcommand*\lineheight[1]{\fontsize{\fsize}{#1\fsize}\selectfont}%
  \ifx\svgwidth\undefined%
    \setlength{\unitlength}{273.47494125bp}%
    \ifx\svgscale\undefined%
      \relax%
    \else%
      \setlength{\unitlength}{\unitlength * \real{\svgscale}}%
    \fi%
  \else%
    \setlength{\unitlength}{\svgwidth}%
  \fi%
  \global\let\svgwidth\undefined%
  \global\let\svgscale\undefined%
  \makeatother%
  \begin{picture}(1,0.38988665)%
    \lineheight{1}%
    \setlength\tabcolsep{0pt}%
    \put(0,0){\includegraphics[width=\unitlength,page=1]{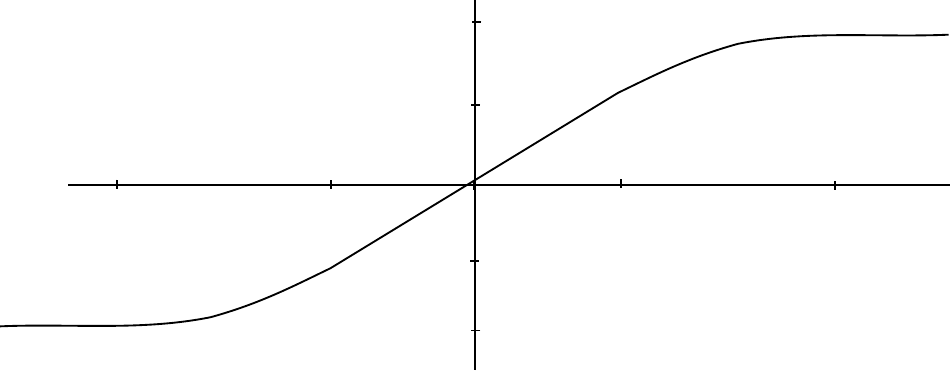}}%
    \put(0.09378875,0.16286937){\color[rgb]{0,0,0}\makebox(0,0)[lt]{\lineheight{1.25}\smash{\begin{tabular}[t]{l}$-3$\end{tabular}}}}%
    \put(0.32922772,0.16438834){\color[rgb]{0,0,0}\makebox(0,0)[lt]{\lineheight{1.25}\smash{\begin{tabular}[t]{l}$-1$\end{tabular}}}}%
    \put(0.64517173,0.15983148){\color[rgb]{0,0,0}\makebox(0,0)[lt]{\lineheight{1.25}\smash{\begin{tabular}[t]{l}$1$\end{tabular}}}}%
    \put(0.85782626,0.16135039){\color[rgb]{0,0,0}\makebox(0,0)[lt]{\lineheight{1.25}\smash{\begin{tabular}[t]{l}$3$\end{tabular}}}}%
  \end{picture}%
\endgroup%
}
\caption{\label{F:tchi-2} A  sketch of the function $\tchi$ used in
  the proof of Theorem \ref{T:non-con}.}
\hfill
\end{figure}
Let $\gamma(s) = \tchi'(s)$ so that $\gamma$ has support in $\{ -3
\leq s \leq 3 \}$, $\gamma(s) \geq 0$, and $\gamma(s) \equiv 1/2$ for
$| s | \leq 1$.

    Choose also a smooth bump function $\tpsi(s) \in \Ci (\reals)$
    satisfying
    \begin{itemize}
    \item $\tpsi(s)$ is even and $\tpsi' \leq 0$ for $s \geq 0$,

    \item $\tpsi(s) \equiv 1$ for $-1 \leq s \leq 1$,

    \item $\tpsi (s) \equiv 0$ for $| s | \geq 2$.

      \end{itemize}


 \subsubsection{  Analysis away from corner points } 
    
    We first consider a boundary point $p_0$ which is on a smooth
    component of the boundary 
  $\Gamma$ away from corners.  
Rotate, translate, and use graph coordinates  so that $\Gamma
\subset \{ y = \alpha(x) \}$ for  locally  smooth $\alpha$, $p_0 = (0,0)$, and $\Omega$
lies below the curve $y=\alpha(x)$.  We will eventually need to
invert $y = \alpha(x)$, so rotate further if necessary to assume that
$\alpha'(0) = 1$.  Let $\beta = \alpha^{-1}$ so that $y = \alpha(x)
\iff x = \beta(y)$ locally near $(0,0)$.  We assume as before that
$\Omega$ lies below the curve $y = \alpha(x)$; that is, $\Omega \subset \{ (x,y); y < \alpha(x) \}.$

Let $\kappa = (1 + (\alpha')^2)^\half$ be the arclength element with
respect to $x$.  Then the normal and tangential derivatives are respectively
\begin{equation}
  \label{E:nuxy-1}
\p_\nu = -\frac{\alpha'}{\kappa} \p_x + \frac{1}{\kappa} \p_y , \,\,\,
\p_\tau = \frac{1}{\kappa} \p_x + \frac{\alpha'}{\kappa} \p_y
\end{equation}
so that
\begin{equation}
  \label{E:nuxy-2}
\p_x = \frac{1}{\kappa} \p_\tau - \frac{\alpha'}{\kappa} \p_\nu ,
\,\,\, \p_y = \frac{\alpha'}{\kappa} \p_\tau + \frac{1}{\kappa} \p_\nu.
\end{equation}

For $\epsilon>0$ sufficiently small but independent of $h$, let

\begin{equation}
  \chi(x,y) = \tchi ( x/h^{1/2}) \tpsi (x/\epsilon) \tpsi(
  y/\epsilon)\label{E:chidef}.
\end{equation}\

If $\epsilon >0$
is sufficiently small, we may assume that $\supp ( \chi |_{\p \Omega}) \subset
\Gamma$.  We have $\chi(x,y) = x/2h^{1/2}$ for $-h^{1/2} \leq x \leq
h^{1/2}$ and $-\epsilon \leq y \leq \epsilon$.  We use  the short hand
notation $\chi_x  := \partial_{x} \chi,  \chi_y := \partial_y \chi$, so    $\supp \chi_x$  consists of three connected components, one
near zero, one near $-\epsilon$, and one near $\epsilon$.  Note:
since $\tchi(x/h^{1/2})$ is constant for $ x  \leq -3 h^{1/2}$ and $x \geq
3h^{1/2}$, we have that  
$\chi_x$ depends on $h$ for $-3h^{1/2} \leq x \leq 3 h^{1/2}$, but on
the set 
$\{ | x | \geq \epsilon \}$, $\chi_x = \epsilon^{-1} \tchi ( x/h^{1/2}
) \tpsi' ( x/\epsilon)\tpsi(y/\epsilon)$ is independent of $h$.  This means that
\begin{equation}
  \label{E:chixsupp}
\chi_x(x,y) \geq  h^{-1/2} \gamma(x/h^{1/2}) \gamma(y/h^{1/2}) - \O(1)
\end{equation}
so that, in particular, $\chi_x \geq  h^{-1/2}/4$ on $B((0,0), h^{1/2})$.

In order to ease notation, let $r>0$ be a small parameter  not depending on $h$  such that
$r \gg \epsilon$ but a $r$ neighbourhood of $(0,0)$ still
does not meet any  corners.  This is just so that integrating in
$[-r ,r]^2 \cap \Omega$ includes the
full support of $\chi$ inside $\Omega$.   See Figure
\ref{F:smooth-side-chi} for a picture.

    \begin{figure}
\hfill
\centerline{
\begingroup%
  \makeatletter%
  \providecommand\color[2][]{%
    \errmessage{(Inkscape) Color is used for the text in Inkscape, but the package 'color.sty' is not loaded}%
    \renewcommand\color[2][]{}%
  }%
  \providecommand\transparent[1]{%
    \errmessage{(Inkscape) Transparency is used (non-zero) for the text in Inkscape, but the package 'transparent.sty' is not loaded}%
    \renewcommand\transparent[1]{}%
  }%
  \providecommand\rotatebox[2]{#2}%
  \newcommand*\fsize{\dimexpr\f@size pt\relax}%
  \newcommand*\lineheight[1]{\fontsize{\fsize}{#1\fsize}\selectfont}%
  \ifx\svgwidth\undefined%
    \setlength{\unitlength}{232.93729591bp}%
    \ifx\svgscale\undefined%
      \relax%
    \else%
      \setlength{\unitlength}{\unitlength * \real{\svgscale}}%
    \fi%
  \else%
    \setlength{\unitlength}{\svgwidth}%
  \fi%
  \global\let\svgwidth\undefined%
  \global\let\svgscale\undefined%
  \makeatother%
  \begin{picture}(1,0.76970696)%
    \lineheight{1}%
    \setlength\tabcolsep{0pt}%
    \put(0,0){\includegraphics[width=\unitlength,page=1]{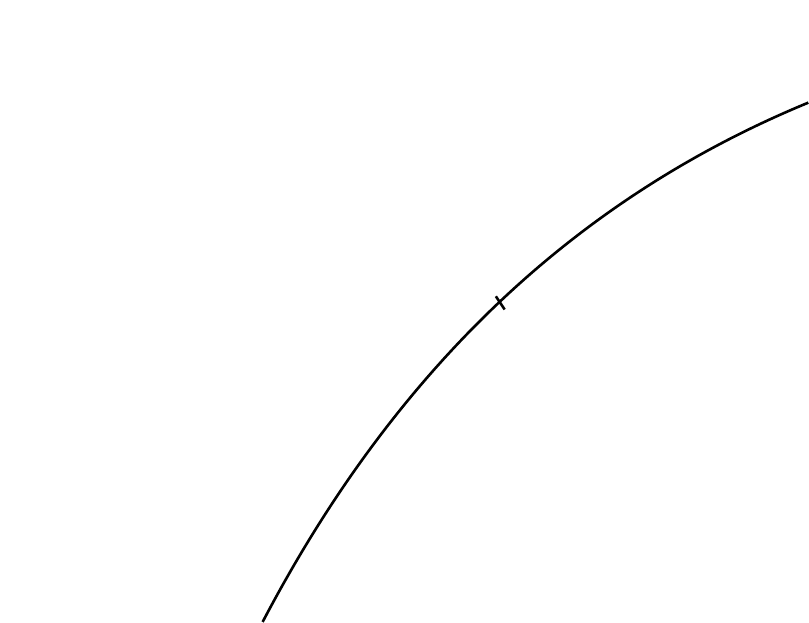}}%
    \put(0.53146402,0.41244644){\color[rgb]{0,0,0}\makebox(0,0)[lt]{\lineheight{1.25}\smash{\begin{tabular}[t]{l}$(0,0)$\end{tabular}}}}%
    \put(0.40044624,0.07339991){\color[rgb]{0,0,0}\makebox(0,0)[lt]{\lineheight{1.25}\smash{\begin{tabular}[t]{l}$y = \alpha(x)$\end{tabular}}}}%
    \put(0,0){\includegraphics[width=\unitlength,page=2]{smooth-side-chi.pdf}}%
    \put(0.1632925,0.73709022){\color[rgb]{0,0,0}\makebox(0,0)[lt]{\lineheight{1.25}\smash{\begin{tabular}[t]{l}$\chi \equiv \tchi(x/h^\half)$\end{tabular}}}}%
    \put(-0.00356352,0.17667156){\color[rgb]{0,0,0}\makebox(0,0)[lt]{\lineheight{1.25}\smash{\begin{tabular}[t]{l}$\supp(\chi)$\end{tabular}}}}%
  \end{picture}%
\endgroup%
}
\caption{\label{F:smooth-side-chi} $\Omega$ in a neighbourhood of a
  point on a smooth side and the function $\chi$.}
\hfill
\end{figure}

We will use a Rellich-type commutator argument, but terms that appear
``lower order'' have non-trivial dependence on $h$ and are not really
lower order.  We have
\[
[-h^2 \Delta , \chi \p_x ] = -2 \chi_x h^2 \p_x^2 - h \chi_{xx} h \p_x
- 2 \chi_y h \p_x h \p_y - h\chi_{yy} h \p_x.
\]
Since $\chi_{xx} = \O(h^{-1})$ and $\chi_y$ and $\chi_{yy}$ are
bounded independent of $h$, we have
\begin{align}
  \int_{\Omega} & ([-h^2 \Delta -1, \chi \p_x ] \phi_h ) \phi_h dV
  \notag \\
  & = \int_{\Omega} ((-2 \chi_x h^2 \p_x^2 - h \chi_{xx} h \p_x
  - 2 \chi_y h \p_x h \p_y - h\chi_{yy} h \p_x) \phi_h ) \phi_h dV.
  \label{E:first-comm}
\end{align}

We recall the standard estimate for first order derivatives:
\[
\int_\Omega | h \nabla \phi_h |^2 dV = \int_\Omega (-h^2 \Delta \phi_h
) \phi_h dV = 1.
\]
We further have
\begin{align*}
  \int_\Omega&  \chi_y (h \p_x h \p_y \phi_h ) \phi_h dV \\
  & = \int_{-r}^r \int_{r}^{\alpha(x)} \chi_y (h \p_y h \p_x \phi_h )
  \phi_h dy dx\\
  & = - \int_{-r}^r \int_r^{\alpha(x)}( h \p_x \phi_h) (h \chi_{yy}
  \phi_h + \chi_y h \p_y \phi_h ) dy dx \\
  & \quad + \int_{-r}^r h \chi_y ( h \p_x \phi_h) \phi_h
  |_{-r}^{\alpha(x)} dx.
\end{align*}
For the boundary term, the support properties of $\chi$ means
\[
\int_{-r}^r h \chi_y ( h \p_x \phi_h) \phi_h
  |_{-r}^{\alpha(x)} dx = \int_{-r}^r h \chi_y ( h \p_x \phi_h) \phi_h
  (x,{\alpha(x)}) dx.
  \]
  Along the face $y = \alpha(x)$, we have $h \p_x \phi_h = \kappa^{-1}
  h \p_\tau \phi_h$, so, in tangent coordinates,
  \begin{align}
  \int_{-r}^r & h \chi_y ( h \p_x \phi_h) \phi_h
  (x,{\alpha(x)}) dx \notag \\
&  = \int h \chi_y \kappa^{-1} (h \p_\tau \phi_h)
  \phi_h dS \notag \\
  & = \frac{h^2}{2} \int \chi_y \kappa^{-1} \p_\tau | \phi_h|^2 dS
  \notag \\
  & = - \frac{h^2}{2} \int (\p_\tau \chi_y \kappa^{-1}) | \phi_h |^2
  dS . \label{E:bdy-h-Sobolev-1}
  \end{align}
  The function $\p_\tau \chi_y \kappa^{-1} = \O(h^{-1/2})$, so, using
  the standard $h$-Sobolev estimates,
  \begin{align}
& \left| \frac{h^2}{2} \int (\p_\tau \chi_y \kappa^{-1}) | \phi_h |^2
    dS \right| \notag \\
    & \quad = \O(h^{1/2}) \int_\Omega (| h \nabla \phi_h |^2 + | \phi_h |^2 dV)
    \notag \\
    & \quad = \O(h^{1/2}). \label{E:bdy-h-Sobolev-2}
  \end{align}
  This implies
  \[
  \int_\Omega  \chi_y (h \p_x h \p_y \phi_h ) \phi_h dV = \O(1),
  \]
so that  \eqref{E:first-comm} becomes
  \begin{align*}
      \int_{\Omega} & ([-h^2 \Delta -1, \chi \p_x ] \phi_h ) \phi_h dV
  \\
   & = -2 \int_\Omega (\chi_x h^2 \p_x^2 \phi_h ) \phi_h dV + \O(1).
  \end{align*}
  

Since supp \, $\chi_x \subset \{ (x,y); \beta(y) < x < r,  |y| < r \},$  by an integration by parts, 
\begin{align}
  -2& \int_\Omega (\chi_x h^2 \p_x^2 \phi_h ) \phi_h dV \notag \\
  & = -2 \int_{-r}^r \int_{\beta(y)}^r (\chi_x h^2 \p_x^2
  \phi_h ) \phi_h dx dy \notag \\
  & = 2\int_{-r}^r \int_{\beta(y)}^r (\chi_x h \p_x
  \phi_h ) h \p_x \phi_h dx dy \notag \\
  & \quad +  2\int_{-r}^r \int_{\beta(y)}^r h (\chi_{xx} h \p_x
  \phi_h )  \phi_h dx dy \notag \\
  & \quad - 2 \int_{-r}^r \left( (h\chi_x h \p_x \phi_h)
  \phi_h\right) |^r_{\beta(y)} dy \notag \\
  & = 2\int_{-r}^r \int_{\beta(y)}^r \chi_x |h \p_x
  \phi_h|^2 dx dy\notag  \\
  & \quad - 2 \int_{-r}^r \left( h(\chi_x h \p_x \phi_h)
  \phi_h\right) |_{\beta(y)}^r dy + \O(1), \label{E:boundaryI1}
\end{align}
where we have again used that $\chi_{xx} = \O(h^{-1})$.
Unfortunately, as $\chi_x = \O(h^{-1/2})$, the boundary term is not
necessarily bounded  in the Neumann case. 

 However, we will see that the largest part miraculously cancels with
a similar boundary term when we run a similar argument for a vector field in
the $\p_y$ direction.  Let
\[
I_1 = - 2 \int_{-r}^r \left( (h\chi_x h \p_x \phi_h)
\phi_h\right) |_{(\beta(y),y)}^r dy
\]
be the boundary term from \eqref{E:boundaryI1}.  Using the support
properties of $\chi_x$, we have $\chi_x ( r , y) = 0$, so that
\[
I_1  = 2 \int_{-r}^r \left( (h\chi_x h \p_x \phi_h)
\phi_h\right) (\beta(y), y) dy.
\]
We now change variables $y = \alpha(x)$ so that
\begin{equation}
  \label{E:bdyx1}
I_1 = 2 \int_{-r}^r \left( (  h  \chi_x h \p_x \phi_h)
\phi_h\right) (x, \alpha(x)) \alpha' dx.
\end{equation}
We will return to this shortly.

Consider now the function
\begin{equation} \label{rho}
\rho(x,y) := \alpha'(x) \tchi ( \beta(y)/ h^{1/2}) \tpsi (x/\epsilon) \tpsi(
y/\epsilon). \end{equation}\

We have
\[
[-h^2 \Delta -1, \rho \p_y ] = -2 \rho_y h^2 \p_y^2 - h \rho_{yy} h
\p_y - 2 \rho_{x} h \p_y h \p_x - h\rho_{xx} h \p_y.
\]
Again, since $\rho_{yy} = \O(h^{-1})$ and $\rho_x$ and $\rho_{xx}$ are
bounded, we have
\[
\int_\Omega ([-h^2 \Delta -1 , \rho \p_y] \phi_h ) \phi_h dV = -2
\int_\Omega ( \rho_y h^2 \p_y^2 \phi_h) \phi_h dV + \O(1).
\]
We again integrate by parts, but  now  in the $y$ direction.  We have
\begin{align}
  -2 &
  \int_\Omega ( \rho_y h^2 \p_y^2 \phi_h) \phi_h dV \notag \\
  & = -2 \int_{-r}^r \int_{-r}^{\alpha(x)} (\rho_y h^2
  \p_y^2 \phi_h ) \phi_h dy dx \notag \\
  & = 2\int_{-r}^r \int_{-r}^{\alpha(x)} \rho_y| h
  \p_y \phi_h |^2 dy dx \notag \\
  & \quad + 2
\int_{-r}^r \int_{-r}^{\alpha(x)} (h \rho_{yy} h
\p_y \phi_h ) \phi_h dy dx
\notag \\
& \quad -2 \int_{-r}^r \left( (h \rho_y h \p_y \phi_h)
\phi_h \right) |_{r}^{\alpha(x)}dx \notag \\
& = 2\int_{-r}^r \int_{-r}^{\alpha(x)} (\rho_y h
  \p_y \phi_h ) h \p_y \phi_h dy dx \notag \\
& \quad -2 \int_{-r}^r \left( (h \rho_y h \p_y \phi_h)
  \phi_h \right) (x,\alpha(x)) dx + \O(1). \label{E:boundaryI2}
\end{align}
Here we have again used that $\rho_{yy} = \O(h^{-1})$ and that
$\rho_y(x , -r) = 0$.

Now let
\begin{equation}
  \label{E:I2-1}
I_2 = -2 \int_{-r}^r \left( (h \rho_y h \p_y \phi_h)
\phi_h \right) (x,\alpha(x)) dx
\end{equation}
be the boundary term from \eqref{E:boundaryI2}.  We observe that
\begin{align*}
  \rho_y & = \alpha'(x) \beta'(y) h^{-1/2} \tchi'(\beta(y)/h^{1/2})
  \tpsi(x/\epsilon) \tpsi(y/\epsilon) + \O(1).
\end{align*}
In \eqref{E:boundaryI2}, we are evaluating at $y = \alpha(x)$, so we
get 
\begin{align}
\rho_y ( x, \alpha(x)) & =   \alpha'(x) \beta'(\alpha(x)) h^{-1/2} \tchi'(x/h^{1/2})
\tpsi(x/\epsilon) \tpsi(\alpha(x)/\epsilon) + \O(1) \notag \\
& = \chi_x (x , \alpha(x)) + \O(1) \label{E:rhoy-chix}
\end{align}
with $\chi$ as in \eqref{E:chidef}.  Substituting into
\eqref{E:I2-1}, we have
\[
I_2 = 
-2 \int_{-r}^r \left( (h \chi_x h \p_y \phi_h)
\phi_h \right) (x,\alpha(x)) dx + \O(1).
\]

We now use the Neumann boundary conditions.  We have 
\begin{align}   0 & = \p_\nu \phi_h(x , \alpha(x)) &
  \notag \\
  & = -\frac{\alpha'}{\kappa} \p_x \phi_h ( x, \alpha(x)) +
  \frac{1}{\kappa} \p_y \phi_h(x, \alpha(x)) \label{E:px-py}
\end{align}
so that $\alpha' \p_x \phi_h (x, \alpha(x)) = \p_y \phi_h ( x ,
\alpha(x)).$  Substituting into \eqref{E:bdyx1}, we have

\begin{align} \label{display1}
  I_1 + I_2 & = \nonumber
2 \int_{-r}^r \left( (h\chi_x h \p_x \phi_h)
\phi_h\right) (x, \alpha(x)) \alpha' dx \\ \nonumber
& \quad 
-2 \int_{-r}^r \left( (h \rho_y h \p_y \phi_h)
\phi_h \right) (x,\alpha(x)) dx  \\ \nonumber
& = 
2 \int_{-r}^r \left( (  h \chi_x h \p_x \phi_h)
\phi_h\right) (x, \alpha(x)) \alpha' dx \\ \nonumber
& \quad 
-2 \int_{-r}^r \left( (h \chi_x h \p_y \phi_h)
\phi_h \right) (x,\alpha(x)) dx + \O(1) \\
& = \O(1).
\end{align}


Summing \eqref{E:boundaryI1} and \eqref{E:boundaryI2} we have
\begin{align*}
  \int_\Omega & ([-h^2 \Delta -1, \chi \p_x ] \phi_h) \phi_h dV 
  \\
  & \quad \int_\Omega  ([-h^2 \Delta -1, \rho \p_y ] \phi_h) \phi_h dV \\
  & = 2\int_\Omega \chi_x | h \p_x \phi_h |^2 dV + 2 \int_\Omega
   \rho_y | h \p_y \phi_h |^2 dV + \O(1).
  \end{align*}
From \eqref{E:chixsupp} we have
\[
\chi_x \geq  h^{-1/2} \gamma(x/h^{1/2}) \gamma(y/h^{1/2}) -
\O(1),
\]
and similarly
\[
\rho_y \geq   h^{-1/2} \gamma(x/h^{1/2}) \gamma(y/h^{1/2}) -
\O(1).
\]
Hence
\begin{align} 
    \int_\Omega & ([-h^2 \Delta -1, \chi \p_x ] \phi_h) \phi_h dV 
  \notag \\
  & \quad + \int_\Omega  ([-h^2 \Delta -1, \rho \p_y ] \phi_h) \phi_h
  dV \notag  \\
  & \geq \int_\Omega  h^{-1/2} \gamma(x/h^{1/2})
  \gamma(y/h^{1/2}) ( | h \p_x \phi_h|^2 + | h \p_y \phi_h |^2) dV -
  \O(1) \notag \\
  & = \int_\Omega  h^{-1/2} \gamma(x/h^{1/2})
  \gamma(y/h^{1/2}) ( - h^2 \p_x^2 \phi_h - h^2 \p_y^2 \phi_h )\phi_h
  dV \notag \\
  & \quad + h\int_{\p \Omega} h^{-1/2} \gamma(x/h^{1/2})
  \gamma(y/h^{1/2}) (h \p_\nu \phi_h) \phi_h dS
  -
  \O(1) 
 \notag  \\
  & = \int_\Omega  h^{-1/2} \gamma(x/h^{1/2})
  \gamma(y/h^{1/2}) | \phi_h|^2 
  dV - \O(1), \label{E:comm-sum-1}
\end{align}
where,  in the last line of  (\ref{E:comm-sum-1}),  we have used the eigenfunction equation and the Neumann boundary
conditions.  Since
\begin{equation*}
\int_\Omega  h^{-1/2} \gamma(x/h^{1/2})
  \gamma(y/h^{1/2}) | \phi_h|^2 
  dV \geq \quarter \int_{B(p_0 , h^{1/2})}  h^{-1/2}  | \phi_h|^2 
  dV,
  \end{equation*}
  we have
  \begin{align}
    \quarter \int_{B(p_0 , h^{1/2})}  h^{-1/2}  | \phi_h|^2 
    dV
    & \leq 
\int_\Omega  ([-h^2 \Delta -1, \chi \p_x ] \phi_h) \phi_h dV 
  \notag \\
  & \quad + \int_\Omega  ([-h^2 \Delta -1, \rho \p_y ] \phi_h) \phi_h dV
  + \O(1). \label{E:Oh-half}
  \end{align}

  Expanding the commutator, using the eigenfunction equation, and integrating by parts, we have
  \begin{align}
    \int_\Omega  & ([-h^2 \Delta -1, \chi \p_x ] \phi_h) \phi_h dV
    \notag \\
    & = \int_\Omega ((-h^2 \Delta -1) \chi \p_x \phi_h) \phi_h dV  -
    \int_\Omega (\chi \p_x (-h^2 \Delta -1)  \phi_h) \phi_h dV \notag \\
    & = \int_\Omega (\chi \p_x \phi_h ) ((-h^2 \Delta -1) \phi_h) dV
     - \int_{\p \Omega} (h \p_\nu \chi h \p_x \phi_h ) \phi_h dS
    \notag \\
    & \quad + \int_{\p \Omega} ( \chi h \p_x \phi_h) ( h \p_\nu
    \phi_h) dS \notag \\
    & = - \int_{\p \Omega} (h \p_\nu \chi h \p_x \phi_h ) \phi_h dS.\label{E:comm-exp-111}
  \end{align}
  Using \eqref{E:nuxy-1}, \eqref{E:nuxy-2}, and the Neumann boundary
  conditions, we have
  \begin{align}
   h \p_\nu \chi h \p_x \phi_h & = \left( -\frac{\alpha'}{\kappa} h\p_x
   + \frac{1}{\kappa} h\p_y \right) \chi h \p_x \phi_h \notag \\
   & = \left( - \frac{ \alpha'}{\kappa} h \chi_x + \frac{1}{\kappa} h
   \chi_y \right)h \p_x \phi_h  \notag \\
   & \quad + \chi h \p_\nu h \p_x \phi_h \notag \\
   & = \left( - \frac{ \alpha'}{\kappa} h \chi_x + \frac{1}{\kappa} h
   \chi_y \right)\left( \frac{1}{\kappa}\right) h \p_\tau \phi_h\notag
   \\
   & \quad - \chi \frac{ \alpha'}{\kappa} h^2 \p_\nu^2 \phi_h + \O(h)
   h \p_\tau \phi_h \notag \\
   & = - \frac{ \alpha'}{\kappa^2} h \chi_xh \p_\tau \phi_h- \chi \frac{
     \alpha'}{\kappa} h^2 \p_\nu^2 \phi_h + \O(h) h \p_\tau \phi_h. \label{E:boundary-102}
  \end{align}
  \begin{remark}
Here is where we see that dealing with the boundary terms for the
Neumann eigenfunctions is 
significantly more difficult than in the case of Dirichlet
eigenfunctions.  Indeed, in the easier case of Dirichlet eigenfunctions, since $\Omega$ is convex, $\int_{\p \Omega} |
h \p_\nu \phi_h |^2 dS$ is bounded and the integrand has a sign. 

    \end{remark}

  Plugging
\eqref{E:boundary-102} into \eqref{E:comm-exp-111}    and using
integration by parts and Sobolev embedding for the $O(h)$ terms
as we did in \eqref{E:bdy-h-Sobolev-1}-\eqref{E:bdy-h-Sobolev-2},
 we have
  \begin{align*}
       \int_\Omega  & ([-h^2 \Delta -1, \chi \p_x ] \phi_h) \phi_h dV
       \\
       & = \int_{\p \Omega} \left(\chi \frac{
         \alpha'}{\kappa} h^2 \p_\nu^2 \phi_h \right) \phi_h dS \\
       & \quad + \int_{\p \Omega} \left( \frac{ \alpha'}{\kappa^2} h
       \chi_xh \p_\tau \phi_h \right) \phi_h dS + \O(1).
  \end{align*}

  A similar computation gives
  \begin{align*}
    \int_\Omega & ([-h^2 \Delta -1, \rho \p_y ] \phi_h ) \phi_h dV \\
    & = - \int_{\p \Omega} \left( \rho \frac{1}{\kappa} h^2 \p_\nu^2
    \phi_h \right) \phi_h dS \\
    & \quad - \int_{\p \Omega} \left( \frac{\alpha'}{\kappa^2} h\rho_y h
    \p_\tau \phi_h \right)\phi_h dS  + \O(1).
  \end{align*}
  Recalling that
  \[
  \rho(x, \alpha(x)) = \alpha' \chi(x, \alpha(x))
  \]
  and
  \[
  \rho_y(x, \alpha(x)) = \chi_x (x , \alpha(x)) + \O(1),
  \]
  we sum:
  \begin{align}
    \int_\Omega  & ([-h^2 \Delta -1, \chi \p_x ] \phi_h) \phi_h dV
    \notag \\
    & \quad + \int_\Omega ([-h^2 \Delta -1, \rho \p_y ] \phi_h )
    \phi_h dV \notag \\
    & = \int_{\p \Omega} \left(\chi \frac{
         \alpha'}{\kappa} h^2 \p_\nu^2 \phi_h \right) \phi_h dS
    \notag \\
       & \quad + \int_{\p \Omega} \left( \frac{ \alpha'}{\kappa^2} h
    \chi_xh \p_\tau \phi_h \right) \phi_h dS \notag \\
    & \quad - \int_{\p \Omega} \left( \rho \frac{1}{\kappa} h^2 \p_\nu^2
    \phi_h \right) \phi_h dS \notag \\
    & \quad - \int_{\p \Omega} \left( \frac{\alpha'}{\kappa^2} h\rho_y h
    \p_\tau \phi_h \right)\phi_h dS  + \O(1) \notag \\
    & = \O(1), \label{E:O1}
  \end{align}
  since the displayed terms on the RHS of (\ref{E:O1}) all cancel.

  Finally, equating \eqref{E:Oh-half} with \eqref{E:O1}, we get
  \[
  \int_{B(p_0 , h^{1/2})}  h^{-1/2}  | \phi_h|^2 
    dV = \O(1)
    \]
    as asserted.

\subsubsection{   Analysis near corner points }  We now consider the case where $p_0$ is a corner.
Translate and
rotate so that $p_0 = (0,0)$, and $\p \Omega$ locally has two smooth
sections.  That is, after a rotation and translation, there exist  locally  smooth functions $\alpha_1$ and
$\alpha_2$ such that $\alpha_1$ is monotone increasing, $\alpha_2$ is
monotone decreasing, $\alpha_1'(0) >0$, and $ \alpha_2'(0) <0$,
and  near $(0,0)$ 
\[
\p \Omega = \{ y = \alpha_1(x) ; 0 \leq x \leq \eta \} \cup \{ y = 
\alpha_2(x) ; 0 \leq x \leq \eta \}
\]
for some $\eta>0$ independent of $h$.  We assume further that locally $\Omega$ lies to
the right of these sections (this is automatic due to convexity of
$\Omega$).
Then locally each $\alpha_j$ has an inverse, which we denote
$\beta_j$.  That is, near $(0,0)$, $y = \alpha_j(x) \iff x =
\beta_j (y)$.

We will need to know the tangential and normal derivatives in these
coordinates.  For the top section where $y = \alpha_1(x)$, we have
already computed in \eqref{E:nuxy-1} and \eqref{E:nuxy-2} with
$\alpha$ replaced by $\alpha_1$.  For the
bottom section where $y = \alpha_2(x)$, let $\kappa_2 = (1 +
(\alpha_2')^2)^\half$ so that the tangent is $\tau = \kappa_2^{-1} ( 1
, \alpha_2')$.  Recalling that $\alpha_2' < 0$ near $0$, the outward
unit normal then is $\nu = \kappa_2^{-1}(\alpha_2' , -1 )$.  Hence
\begin{equation}
  \label{E:nuxy-3}
\p_\nu = \frac{\alpha_2'}{\kappa_2} \p_x - \frac{1}{\kappa_2} \p_y , \,\,\,
\p_\tau = \frac{1}{\kappa_2} \p_x + \frac{\alpha_2'}{\kappa_2} \p_y
\end{equation}
so that
\begin{equation}
  \label{E:nuxy-4}
\p_x = \frac{1}{\kappa_2} \p_\tau + \frac{\alpha_2'}{\kappa_2} \p_\nu ,
\,\,\, \p_y = \frac{\alpha_2'}{\kappa_2} \p_\tau - \frac{1}{\kappa_2} \p_\nu.
\end{equation}

For $\epsilon>0$ sufficiently small, let $\chi(x,y)$ be the same as in
\eqref{E:chidef}.  We again use a parameter $r \gg \epsilon$ but
sufficiently small that $[-r, r]^2$ does not meet any other
corners.  Again, this is just to ease notation in our integral
expressions.  
Applying the same commutator argument as in the smooth boundary segment case, the
interior computations are the same, we just need to check what happens
on the boundary.  The key difference from the  case with no corners is
that boundary integrals have to be considered piecewise.  See Figure
\ref{F:corner-chi} for a picture of the setup.

   \begin{figure}
\hfill
\centerline{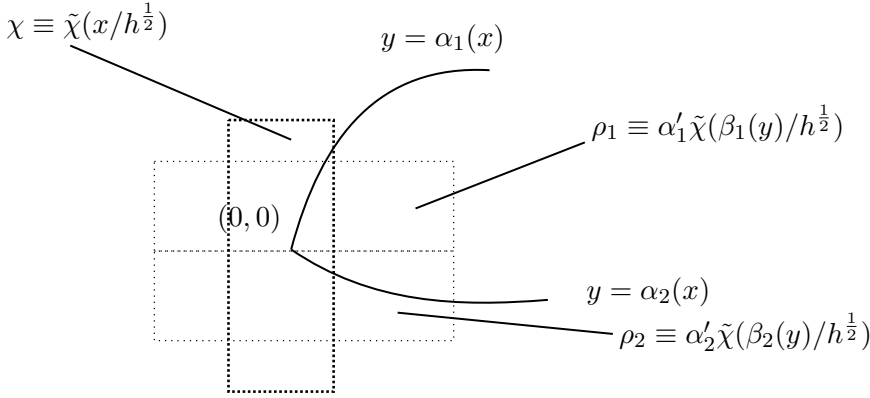}
\caption{\label{F:corner-chi} $\Omega$ in a neighbourhood of a corner
  and the functions $\chi$, $\rho_1$, and $\rho_2$.}
\hfill
\end{figure}

Integrating
by parts in the $x$ direction on the interior terms, we have:
\begin{align*}
 \int_\Omega  ([-h^2 \Delta -1, \chi \p_x ] \phi_h ) \phi_h dV & =  -2
 \int_\Omega  \chi_x (h^2 \p_x^2 \phi_h ) \phi_h dV + \O(1) \\
& = -2 \int_{-r}^0 \int_{x = \beta_2(y)}^r \chi_x (h^2 \p_x^2
  \phi_h ) \phi_h dx dy \\
  & \quad  -2 \int_0^r \int_{x = \beta_1(y)}^r \chi_x (h^2 \p_x^2
  \phi_h ) \phi_h dx dy  + \O(1) \\
  & = : I_1 + I_2 + \O(1).
  \end{align*}
Let us examine $I_1$ first:
\begin{align*}
  I_1 & = 2 \int_{-r}^0 \int_{\beta_2(y)}^r h \chi_{xx} (h
  \p_x \phi_h ) \phi_h dx dy + 2 \int_{-r}^0 \int_{\beta_2(y)}^r  \chi_{x} (|h
  \p_x \phi_h |^2 )dx dy \\
  & \quad -2 \int_{y = -r}^0 h \chi_x (h \p_x \phi_h ) \phi_h
  |_{x = \beta_2(y)}^{x = r} dy \\
  & = 
  2 \int_{-r}^0 \int_{\beta_2(y)}^r  \chi_{x} |h
  \p_x \phi_h |^2 dx dy \\
  & \quad +2 \int_{y = -r}^0 (h \chi_x (h \p_x \phi_h )
  \phi_h)(\beta_2(y), y) dy + \O(1),
\end{align*}
since $\chi$ has support in $x \leq 2 \epsilon \ll r$, and $h
\chi_{xx} = \O(1)$.

Similarly,
\begin{align*}
I_2 & = 
  2 \int_0^r \int_{ \beta_1(y)}^r  \chi_{x} (|h
  \p_x \phi_h |^2 dx dy \\
  & \quad +2 \int_0^r (h \chi_x (h \p_x \phi_h )
  \phi_h)(\beta_1(y), y) dy + \O(1).
\end{align*}
Summing, we have
\begin{align*}
  \int_{\Omega} & ([-h^2 \Delta -1, \chi \p_x ] \phi_h ) \phi_h dV \\
  & = I_1 + I_2 + \O(1)  \\
  & = 
  2  \int_\Omega \chi_{x} (|h
  \p_x \phi_h |^2 dx dy 
  \\
  & \quad +2 \int_{y = -r}^0 (h \chi_x (h \p_x \phi_h )
  \phi_h)(\beta_2(y), y) dy \\
  & \quad +2 \int_{y = 0}^r (h \chi_x (h \p_x \phi_h )
  \phi_h)(\beta_1(y), y) dy + \O(1).
  \end{align*}
For the two boundary terms, we change variables $y = \alpha_2(x)$ and $y
= \alpha_1(x)$ respectively to get
\begin{align}
   \int_{\Omega} & ([-h^2 \Delta -1, \chi \p_x ] \phi_h ) \phi_h dV
   \notag \\
& = 
  2  \int_\Omega \chi_{x} (|h
  \p_x \phi_h |^2 dx dy -2  \int_0^r  \alpha_2'(x)(h \chi_x (h \p_x \phi_h )
  \phi_h)(x, \alpha_2(x)) dx \notag \\
  & \quad +2  \int_0^r \alpha_1'(x) (h \chi_x (h \p_x \phi_h )
  \phi_h)(x , \alpha_1(x)) dx + \O(1). \label{E:corner-bdy-x-1}
  \end{align}
Note the sign change on the second integral to correct for reversed
orientation in the $x$ direction.

We now want to employ a  similar argument with $\p_y$.  However, our
function $\rho$ cannot be globally defined if we want to write $\rho$
in terms of $\chi$ on the boundary, since we are not assuming any
relation between $\alpha_1$ and $\alpha_2$.  Let
\[
\Omega_1 = \Omega \cap \{ x \leq r \} \cap \{ y \geq 0 \},
\]
and
\[
\Omega_2 = \Omega \cap \{ x \leq r \} \cap \{ y \leq 0 \}
\]
be the top and bottom parts of $\Omega$ near $(0,0)$.  For $j = 1,2$, let

\begin{equation} \label{rhoj}
\rho_j(x,y) := \alpha_j'(x) \tchi( \beta_j(y)/h^\half) \tpsi
(x/\epsilon) \tpsi (y/ \epsilon), \quad j=1,2.
\end{equation}\

See Figure \ref{F:corner-chi} for a picture of the setup.

Let us record some facts about the $\rho_j$'s.  First, along $y =
\alpha_j$, $j = 1,2$, we have
\begin{equation}
  \label{E:rho-y-j}
\rho_j ( x , \alpha_j(x)) = \alpha_j'(x) \tchi(x/h^\half)
\tpsi(x/\epsilon) \tpsi(\alpha_j(x)/\epsilon) = \alpha_j' \chi(x,
\alpha_j(x)).
\end{equation}
Along $y = 0$, $\rho_j = 0$, since $\tchi$ is an odd function and
$\beta_j(0) = 0$ for $j = 1,2$.  Along
$y = \alpha_j$,
\begin{align*}
\p_y \rho_j(x, \alpha_j(x)) & = h^{-\half} \alpha_j' (x)
\beta_j'(\alpha_j(x)) \tchi'(x/h^\half)
\tpsi(x/\epsilon)\tpsi(\alpha_j(x)/\epsilon) + \O(1) \\
& =  h^{-\half} \tchi'(x/h^\half)
\tpsi(x/\epsilon)\tpsi(\alpha_j(x)/\epsilon) + \O(1) \\
& = \p_x \chi(x, \alpha_j(x)) + \O(1).
\end{align*}
Here the implicit $\O(1)$ errors are from differentiating the $\tpsi$
functions so are supported away from the corners.  Finally, along $y =
0$, we have $\tchi'(0) = 1/2$ and $\tpsi'(0) = 0$, so that
\begin{align}
  \p_y \rho_j(x,0) &= h^{-\half} \tchi'(0) \tpsi(x/\epsilon) \tpsi(0) \\
  & = (h^{-1/2}/2) \tpsi(x/\epsilon).\label{E:rho-y-0}
  \end{align}

Now consider the vector field $\rho_1 \p_y$ on $\Omega_1$.   The same commutator computation and integrations by
parts in $y$ yields the following:
\begin{align}
  \int_{\Omega_1} & ([-h^2 \Delta -1, \rho_1 \p_y ] \phi_h ) \phi_h
  dV \notag \\
  & = 
  -2 \int_{\Omega_1} \rho_{1,y} (h^2 \p_y^2 \phi_h ) \phi_h dV  + \O(1) \notag \\
  & = -2 \int_{x = 0}^r \int_{y =0 }^{y = \alpha_1(x)}
  \rho_{1,y} (h^2 \p_y^2 \phi_h ) \phi_h dy dx + \O(1)\notag \\
  & = 
  2 \int_{x = 0}^r \int_{y = 0}^{y = \alpha_1(x)}
  \rho_{1,y} |h \p_y \phi_h |^2 dy dx \notag  - 2 \int_{x = 0}^r h \rho_{1,y} (h \p_y \phi_h ) \phi_h
  |_{y = 0}^{y = \alpha_1(x)} dx + \O(1) \notag \\
  & = 
  2 \int_{\Omega_1}
  \rho_{1,y} |h \p_y \phi_h |^2 dy dx  - 2 \int_{x = 0}^r (h \rho_{1,y} (h \p_y \phi_h ) \phi_h)(x,
  \alpha_1(x)) dx \notag
  \\
  & \quad + 2 \int_{x = 0}^r (h \rho_{1,y} (h \p_y \phi_h ) \phi_h)(x,
  0) dx + \O(1) \notag \\
  & = 2 \int_{\Omega_1}
  \rho_{1,y} |h \p_y \phi_h |^2 dy dx \notag \\
  & \quad - 2 \int_{x = 0}^r (h (\chi_x+\O(1)) (h \p_y \phi_h ) \phi_h)(x,
  \alpha_1(x)) dx \notag 
  \\
  & \quad + 2 \int_{x = 0}^r(h^{-1/2}/2) \tpsi(x/\epsilon) (h  (h \p_y \phi_h ) \phi_h)(x,
  0) dx + \O(1) \notag \\
& = 2 \int_{\Omega_1}
  \rho_{1,y} |h \p_y \phi_h |^2 dV  - 2 \int_{x = 0}^r (h \chi_x (h \p_y \phi_h ) \phi_h)(x,
  \alpha_1(x)) dx \notag \\
  & \quad + h^\half\int_{x = 0}^r \tpsi(x/\epsilon)   (h \p_y \phi_h ) \phi_h(x,
  0) dx + \O(1)
  . \label{E:corner-bdy-y-1}
\end{align}

Here we have used integration by parts along the boundary and Sobolev embedding on the implicit $\O(h)$ boundary
terms supported away from the corner, just as in \eqref{E:bdy-h-Sobolev-1}-\eqref{E:bdy-h-Sobolev-2}.

A similar computation on $\Omega_2$ using the vector field $\rho_2
\p_y$ gives
\begin{align}
   \int_{\Omega_2} & ([-h^2 \Delta -1, \rho_2 \p_y ] \phi_h ) \phi_h
  dV \notag \\
  & = 
  -2 \int_{\Omega_2} \rho_{2,y} (h^2 \p_y^2 \phi_h ) \phi_h dV  + \O(1) \notag \\
  & = -2 \int_{x = 0}^r \int_{y =\alpha_2 }^{y = 0}
  \rho_{2,y} (h^2 \p_y^2 \phi_h ) \phi_h dy dx + \O(1)\notag \\
  & = 
  2 \int_{x = 0}^r \int_{y = \alpha_2}^{y = 0}
  \rho_{2,y} |h \p_y \phi_h |^2 dy dx  - 2 \int_{x = 0}^r h \rho_{1,y} (h \p_y \phi_h ) \phi_h
  |_{y = \alpha_2}^{y = 0} dx + \O(1) \notag \\
  & =  2 \int_{\Omega_2}
  \rho_{2,y} |h \p_y \phi_h |^2 dV  - h^\half \int_{x = 0}^r \tpsi(x/\epsilon) (h \p_y \phi_h)
  \phi_h (x,0) dx \notag \\
  & \quad + 2 \int_{x = 0}^r h \chi_x ( h \p_y \phi_h) \phi_h (x
  , \alpha_2(x)) dx + \O(1).
  \label{E:corner-bdy-y-2}
  \end{align}

Summing \eqref{E:corner-bdy-y-1} and \eqref{E:corner-bdy-y-2} and
making the obvious cancellations, we have
\begin{align}
  \int_{\Omega_1} & ([-h^2 \Delta -1, \rho_1 \p_y ] \phi_h) \phi_h dV
  + \int_{\Omega_2}  ([-h^2 \Delta -1, \rho_2 \p_y ] \phi_h) \phi_h
  dV \notag \\
  & = 2 \int_{\Omega_1}
  \rho_{1,y} |h \p_y \phi_h |^2 dV + 2 \int_{\Omega_2}
  \rho_{2,y} |h \p_y \phi_h |^2 dV \notag  \\
  & \quad  - 2 \int_{x = 0}^r (h \chi_x (h \p_y \phi_h ) \phi_h)(x,
  \alpha_1(x)) dx \notag \\
  & \quad + 2 \int_{x = 0}^r h \chi_x ( h \p_y \phi_h) \phi_h (x
  , \alpha_2(x)) dx + \O(1). \label{E:corner-y-both}
  \end{align}

We now use the Neumann boundary conditions on $\phi_h$ and sum
\eqref{E:corner-bdy-x-1} and \eqref{E:corner-y-both}.
On the top
segment where $y \geq 0$, we have \eqref{E:nuxy-1} and
\eqref{E:nuxy-2} so 
\[
0 = \p_\nu \phi_h = -\frac{\alpha_1'}{\kappa_1} \p_x \phi_h +
\frac{1}{\kappa_1} \p_y \phi_h.
\]
Then $\p_y \phi_h = \alpha_1' \p_x \phi_h$ on the upper section.
Similarly, on the bottom section we have \eqref{E:nuxy-3} and
\eqref{E:nuxy-4} so that $\p_y \phi_h =  \alpha_2' \p_x
\phi_h$.  Consequently, \eqref{E:corner-y-both} becomes
\begin{align}
 \int_{\Omega_1} & ([-h^2 \Delta -1, \rho_1 \p_y ] \phi_h) \phi_h dV
  + \int_{\Omega_2}  ([-h^2 \Delta -1, \rho_2 \p_y ] \phi_h) \phi_h
  dV \notag \\
  & = 2 \int_{\Omega_1}
  \rho_{1,y} |h \p_y \phi_h |^2 dV + 2 \int_{\Omega_2}
  \rho_{2,y} |h \p_y \phi_h |^2 dV \notag  \\
  & \quad  - 2 \int_{x = 0}^r (h \chi_x (\alpha_1' h \p_x \phi_h ) \phi_h)(x,
  \alpha_1(x)) dx \notag \\
  & \quad + 2 \int_{x = 0}^r h \chi_x (\alpha_2' h \p_x \phi_h) \phi_h (x
  , \alpha_2(x)) dx + \O(1). \label{E:corner-y-both-2}
  \end{align}
Now summing \eqref{E:corner-bdy-x-1} and \eqref{E:corner-y-both-2} and
making the obvious cancellations, we
have
\begin{align}
   \int_{\Omega} & ([-h^2 \Delta -1, \chi \p_x ] \phi_h ) \phi_h dV
   \notag \\
   &\quad + \int_{\Omega_1}  ([-h^2 \Delta -1, \rho_1 \p_y ] \phi_h) \phi_h dV
  + \int_{\Omega_2}  ([-h^2 \Delta -1, \rho_2 \p_y ] \phi_h) \phi_h
  dV \notag \\
  & = 2  \int_\Omega \chi_{x} (|h
  \p_x \phi_h |^2 dV + 
  2 \int_{\Omega_1}
  \rho_{1,y} |h \p_y \phi_h |^2 dV + 2 \int_{\Omega_2}
  \rho_{2,y} |h \p_y \phi_h |^2 dV  + \O(1).  \label{E:comm-master-1}
  \end{align}

It remains to compute the commutators  on the LHS of (\ref{E:comm-master-1}). By Green's formula,
\[
\int_\Omega ([-h^2 \Delta -1, \chi(x,y) \p_x ] \phi_h) \phi_h dV = -
\int_{\p \Omega} (h \p_\nu \chi h \p_x \phi_h ) \phi_h dS
\]
and for $j = 1,2$
\begin{align*}
\int_{\Omega_j} &([-h^2 \Delta -1, \rho_j(x,y) \p_y ] \phi_h) \phi_h
dV \\
& = -
\int_{\p \Omega_j} (h \p_\nu \rho_j h \p_y \phi_h ) \phi_h dS +
\int_{\p \Omega_j} (\rho_j h \p_y \phi_h )(h \p_\nu \phi_h) dS \\
& = -\int_{\p \Omega_j} (h \p_\nu \rho_j h \p_y \phi_h ) \phi_h dS.
\end{align*}

Here the second integral in the second line is zero since $\phi_h$ has
Neumann boundary conditions along the boundary $y = \alpha_1$, and
$\rho_j = 0$ along the line $y = 0$.

On the upper segment, we use that
$
\p_\nu = -\frac{\alpha_1'}{\kappa_1} \p_x + \frac{1}{\kappa_1} \p_y , $ and 
$
\p_x = \frac{1}{\kappa_1} \p_\tau - \frac{\alpha_1'}{\kappa_1} \p_\nu$ to get
\begin{align*}
  h \p_\nu \chi h \p_x \phi_h & =  \chi h \p_x h \p_\nu \phi_h + [h
    \p_\nu , \chi h \p_x ] \phi_h \\
  & = -\frac{\alpha_1'}{\kappa_1}  \chi h^2
\p_\nu^2 \phi_h -
   \frac{\alpha_1'}{\kappa_1^2} h \chi_x
 h \p_\tau \phi_h + \O(h) h \p_\tau \phi_h .
\end{align*}
  
Similarly, on the lower segment, $
\p_\nu = \frac{\alpha_2'}{\kappa_2} \p_x - \frac{1}{\kappa_2} \p_y $ and 
$
\p_x = \frac{1}{\kappa_2} \p_\tau + \frac{\alpha_2'}{\kappa_2} \p_\nu $ so
that 
\begin{align*}
  h \p_\nu \chi h \p_x \phi_h & =\frac{\alpha_2'}{\kappa_2}  \chi h^2
\p_\nu^2 \phi_h + \frac{\alpha_2'}{\kappa_2^2}h \chi_x h \p_\tau \phi_h +
\O(h) h \p_\tau
\phi_h.
\end{align*}

Plugging in, we have
\begin{align}
  \int_\Omega & ([-h^2 \Delta -1, \chi(x,y) \p_x ] \phi_h) \phi_h dV
 \notag  \\
  &  =-
  \int_{\p \Omega} (h \p_\nu \chi h \p_x \phi_h ) \phi_h dS \notag \\
  & = - \int_{\p \Omega \cap \{ y \geq 0 \} } (-\frac{\alpha_1'}{\kappa_1}  \chi h^2
\p_\nu \phi_h -
   \frac{\alpha_1'}{\kappa_1^2} h \chi_x
   h \p_\tau \phi_h + \O(h) h \p_\tau \phi_h) \phi_h dS \notag \\
   & \quad - \int_{\p \Omega \cap \{ y \leq 0 \} } ( \frac{\alpha_2'}{\kappa_2}  \chi h^2
\p_\nu \phi_h + \frac{\alpha_2'}{\kappa_2^2} h \chi_xh \p_\tau \phi_h + \O(h)
h \p_\tau \phi_h) \phi_h dS \notag \\
& = - \int_{\p \Omega \cap \{ y \geq 0 \} } (-\frac{\alpha_1'}{\kappa_1}  \chi h^2
\p_\nu \phi_h -
   \frac{\alpha_1'}{\kappa_1^2} h \chi_x
   h \p_\tau \phi_h ) \phi_h dS \notag \\
   & \quad - \int_{\p \Omega \cap \{ y \leq 0 \} } ( \frac{\alpha_2'}{\kappa_2}  \chi h^2
\p_\nu \phi_h + \frac{\alpha_2'}{\kappa_2^2} h \chi_x h \p_\tau \phi_h )
\phi_h dS + \O(1), \label{E:chi-IBP-1}
\end{align}
where we have again used integration by parts along the boundary and Sobolev embedding on the implicit $\O(h)$
boundary terms supported away from the corner, just as we did in \eqref{E:bdy-h-Sobolev-1}-\eqref{E:bdy-h-Sobolev-2}.

For the computations involving the vector fields $\rho_j \p_y$, we
have  by Green's formula 
\begin{align*}
  \int_{\Omega_j}&  ([-h^2 \Delta -1, \rho_j \p_y ] \phi_h ) \phi_h dV
  \\
  & = - \int_{\p \Omega_j} (h \p_\nu \rho_jh \p_y \phi_h) \phi_h dS
  \\
  & = - \int_{\{ y = \alpha_1(x)\}} (h \p_\nu \rho_jh \p_y \phi_h)
  \phi_h dS - \int_{\{ y = 0 \}} (h \p_\nu \rho_jh \p_y \phi_h)
  \phi_h dS,
\end{align*}
since $\tpsi(x /\epsilon)$ has compact support in $\{ x \leq 2
\epsilon \ll r \}$.
Using the same computations  which led to \eqref{E:chi-IBP-1}, on $\{
y = \alpha_1 \}$, we have
\[
h \p_\nu \rho_1 h \p_y \phi_h = \frac{1}{\kappa_1} \rho_1 h^2 \p_\nu^2
\phi_h + \frac{\alpha_1'}{\kappa_1^2} h\rho_{1,y} h \p_\tau \phi_h +
\O(h) h \p_\tau \phi_h.
\]
On $\{ y = 0 \}$, from
$\Omega_1$, we have $\p_\nu = - \p_y$, so that
\begin{align*}
  h \p_\nu \rho_1 h \p_y \phi_h & = -h \p_y \rho_1 h \p_y \phi_h \\
  & = -h \rho_{1,y} h \p_y \phi_h - \rho_1 h^2 \p_y^2 \phi_h \\
  & = -h \rho_{1,y} h \p_y \phi_h \\
  & = -(h^\half/2) \tpsi(x/\epsilon) h \p_y \phi_h,
\end{align*}
since $\rho_1(x,0) = 0$.  In the last line we have used \eqref{E:rho-y-0}.
Putting this together, we have
\begin{align}
  \int_{\Omega_1}&  ([-h^2 \Delta -1, \rho_1 \p_y ] \phi_h ) \phi_h dV
 \notag  \\ 
  & = - \int_{\p \Omega_1} (h \p_\nu \rho_1h \p_y \phi_h) \phi_h dS
  \notag \\
  & = -\int_{\{ y = \alpha_1 \}} (h \p_\nu \rho_1 h \p_y \phi_h)
  \phi_h dS  - \int_{\{y = 0 \}} (h \p_\nu \rho_1 h \p_y \phi_h)
  \phi_h dS \notag  \\
  & = -\int_{\{ y = \alpha_1\}} ( \frac{1}{\kappa_1} \rho_1 h^2 \p_\nu^2
\phi_h + \frac{\alpha_1'}{\kappa_1^2} h\rho_{1,y} h \p_\tau \phi_h +
\O(h) h \p_\tau \phi_h) \phi_h dS \notag \\
& \quad - \int_{\{ y = 0 \}} (-h \rho_{1,y} h \p_y \phi_h) \phi_h dS
\notag \\
& = -\int_{\{ y = \alpha_1\}} ( \frac{1}{\kappa_1} \rho_1 h^2 \p_\nu^2
\phi_h + \frac{\alpha_1'}{\kappa_1^2} h\rho_{1,y} h \p_\tau \phi_h)
\phi_h dS \notag \\
& \quad - \int_{\{ y = 0 \}} (-(h^\half/2) \tpsi(x/\epsilon) h \p_y \phi_h) \phi_h dS
+ \O(1).\label{E:comm-om-1}
\end{align}
Here we have once again used the Sobolev embedding on the implicit
$\O(h)$ boundary terms just as we did in \eqref{E:bdy-h-Sobolev-1}-\eqref{E:bdy-h-Sobolev-2}.

In a similar fashion, we compute for $\Omega_2$:
\begin{align}
  \int_{\Omega_2} & ([-h^2 \Delta -1, \rho_2 \p_y ] \phi_h ) \phi_h
  dV \notag 
  \\
  & = - \int_{\p \Omega_2} (h \p_\nu \rho_2 h \p_y \phi_h) \phi_h dS
  \notag 
  \\
  & =  -\int_{\{ y = \alpha_2\}} ( -\frac{1}{\kappa_2} \rho_2 h^2 \p_\nu^2
\phi_h - \frac{\alpha_2'}{\kappa_1^2} h\rho_{2,y} h \p_\tau \phi_h)
\phi_h dS \notag \\
& \quad - \int_{\{ y = 0 \}} ((h^\half/2) \tpsi(x/\epsilon) h \p_y \phi_h) \phi_h dS
+ \O(1). \label{E:comm-om-2}
\end{align}
Here in the last line we have used that, from $\Omega_2$, $\p_\nu =
\p_y$ along $\{ y = 0 \}$.

Summing \eqref{E:comm-om-1} and \eqref{E:comm-om-2}, we have
\begin{align}
   \int_{\Omega_1} & ([-h^2 \Delta -1, \rho_1 \p_y ] \phi_h ) \phi_h
   dV
   +  \int_{\Omega_2}  ([-h^2 \Delta -1, \rho_2 \p_y ] \phi_h ) \phi_h
   dV\notag  \\
   & = 
-\int_{\{ y = \alpha_1\}} ( \frac{1}{\kappa_1} \rho_1 h^2 \p_\nu^2
\phi_h + \frac{\alpha_1'}{\kappa_1^2}h \rho_{1,y} h \p_\tau \phi_h)
\phi_h dS \notag \\
& \quad + \int_{\{ y = 0 \}} ((h^\half/2) \tpsi(x/\epsilon) h \p_y
\phi_h) \phi_h dS \notag \\
& \quad
 +\int_{\{ y = \alpha_2\}} ( \frac{1}{\kappa_2} \rho_2 h^2 \p_\nu^2
\phi_h + \frac{\alpha_2'}{\kappa_2^2}h \rho_{2,y} h \p_\tau \phi_h)
\phi_h dS \notag \\
& \quad - \int_{\{ y = 0 \}} ((h^\half/2) \tpsi(x/\epsilon) h \p_y \phi_h) \phi_h dS
+ \O(1) \notag \\
& = -\int_{\{ y = \alpha_1\}} ( \frac{1}{\kappa_1} \rho_1 h^2 \p_\nu^2
\phi_h + \frac{\alpha_1'}{\kappa_1^2} h\rho_{1,y} h \p_\tau \phi_h)
\phi_h dS \notag \\
& \quad 
 +\int_{\{ y = \alpha_2\}} ( \frac{1}{\kappa_2} \rho_2 h^2 \p_\nu^2
\phi_h + \frac{\alpha_2'}{\kappa_2^2} h\rho_{2,y} h \p_\tau \phi_h)
\phi_h dS 
+ \O(1) \label{E:comm-om-both}
\end{align}

From 
\eqref{E:rho-y-j}, we know that $$\rho_j(x, \alpha_j(x)) =\alpha_j'
\chi(x, \alpha_j(x))$$ and $$\rho_{j,y} (x, \alpha_j(x))= \chi_x( x ,
\alpha_j(x)) + \O(1),$$ so that
\eqref{E:comm-om-both} becomes
\begin{align}
     \int_{\Omega_1} & ([-h^2 \Delta -1, \rho_1 \p_y ] \phi_h ) \phi_h
   dV
   +  \int_{\Omega_2}  ([-h^2 \Delta -1, \rho_2 \p_y ] \phi_h ) \phi_h
   dV\notag  \\
&
= -\int_{\{ y = \alpha_1\}} ( \frac{1}{\kappa_1} \alpha_1' \chi h^2 \p_\nu^2
\phi_h + \frac{\alpha_1'}{\kappa_1^2} h\chi_x h \p_\tau \phi_h)
\phi_h dS \notag \\
& \quad 
 +\int_{\{ y = \alpha_2\}} ( \frac{1}{\kappa_2} \alpha_2' \chi  h^2 \p_\nu^2
\phi_h + \frac{\alpha_2'}{\kappa_2^2} h\chi_x h \p_\tau \phi_h)
\phi_h dS 
+ \O(1) \label{E:comm-om-both-2}
\end{align}

Summing \eqref{E:chi-IBP-1} and \eqref{E:comm-om-both-2}, we have
\begin{align}
  \int_\Omega & ([-h^2 \Delta -1, \chi \p_x] \phi_h) \phi_h dV \notag
  \\
  & \quad + 
      \int_{\Omega_1}  ([-h^2 \Delta -1, \rho_1 \p_y ] \phi_h ) \phi_h
   dV
   +  \int_{\Omega_2}  ([-h^2 \Delta -1, \rho_2 \p_y ] \phi_h ) \phi_h
   dV\notag \\
   & = 
 - \int_{\p \Omega \cap \{ y \geq 0 \} } (-\frac{\alpha_1'}{\kappa_1}  \chi h^2
\p_\nu \phi_h -
   \frac{\alpha_1'}{\kappa_1^2} h \chi_x
   h \p_\tau \phi_h ) \phi_h dS \notag \\
   & \quad - \int_{\p \Omega \cap \{ y \leq 0 \} } ( \frac{\alpha_2'}{\kappa_2}  \chi h^2
\p_\nu \phi_h + \frac{\alpha_2'}{\kappa_2^2} h \chi_xh \p_\tau \phi_h )
\phi_h dS \notag \\
& \quad -\int_{\{ y = \alpha_1\}} ( \frac{1}{\kappa_1} \alpha_1' \chi h^2 \p_\nu^2
\phi_h + \frac{\alpha_1'}{\kappa_1^2} h\chi_x h \p_\tau \phi_h)
\phi_h dS \notag \\
& \quad 
 +\int_{\{ y = \alpha_2\}} ( \frac{1}{\kappa_2} \alpha_2' \chi  h^2 \p_\nu^2
\phi_h + \frac{\alpha_2'}{\kappa_2^2} h\chi_x h \p_\tau \phi_h)
\phi_h dS 
+ \O(1). \label{E:master-comm-2}
\end{align}
All of the displayed boundary terms in \eqref{E:master-comm-2} cancel,
so that
\begin{align}
 \int_\Omega & ([-h^2 \Delta -1, \chi \p_x] \phi_h) \phi_h dV \notag
  \\
  & \quad + 
      \int_{\Omega_1}  ([-h^2 \Delta -1, \rho_1 \p_y ] \phi_h ) \phi_h
   dV
   +  \int_{\Omega_2}  ([-h^2 \Delta -1, \rho_2 \p_y ] \phi_h ) \phi_h
   dV\notag \\
   & = \O(1).
 \label{E:master-comm-21}
\end{align}

Finally, equating 
\eqref{E:comm-master-1} and \eqref{E:master-comm-2}, we have shown
\[
2  \int_\Omega \chi_{x} (|h
  \p_x \phi_h |^2 dV + 
  2 \int_{\Omega_1}
  \rho_{1,y} |h \p_y \phi_h |^2 dV + 2 \int_{\Omega_2}
  \rho_{2,y} |h \p_y \phi_h |^2 dV
= 
\O(1).
\]

Using the same estimates as in \eqref{E:chixsupp}, we have that
\[
\chi_x  \geq  h^{-1/2} \gamma(x/h^{1/2})\gamma(y/h^{1/2})
\gamma(y/h^{1/2}) - \O(1) \]
and on each of $\Omega_j$,
\[
\rho_{j,y} \geq  h^{-1/2} \gamma(x/h^{1/2})\gamma(y/h^{1/2}),
\]
so arguing as in \eqref{E:comm-sum-1}, we finally get
\[
\int_{B((0,0), h^\half)} h^{-\half} | \phi_h |^2 dV = \O(1).
\]

{\bf Step 2: $\delta = 2/3$.}

We are now ready to bootstrap the estimate for $\delta = 2/3$.  The
argument proceeds exactly as in the $\delta = 1/2$ case, but now some
of the error terms are no longer so easy to absorb.  
 We begin with the case where $p_0$ is not a corner, starting with
 defining the cutoff $\chi$ as in 
 \eqref{E:chidef}.  For $\epsilon>0$ small but independent of $h$, let

 \begin{equation}
  \chi(x,y) = \tchi ( x/h^{2/3}) \tpsi (x/\epsilon) \tpsi(
  y/\epsilon)\label{E:chidef-101}.
 \end{equation}\

Observe the only difference  in (\ref{E:chidef-101}) versus the cutoff in (\ref{E:chidef})  is the $h^{-2/3}$ appearing instead of $h^{-1/2}$.
This is good, since we will once again need some boundary terms to
cancel.  The argument is identical to the argument in the $\delta =
1/2$ case except for one piece: $\chi_{xx}$ is no longer $\O(h^{-1})$
but instead is $\O(h^{-4/3})$.  We will have to work harder to control
this.

Beginning with the commutator, 
since $\chi_y$ and $\chi_{yy}$ are both 
bounded, we have
\begin{align*}
  \int_{\Omega} & ([-h^2 \Delta -1, \chi \p_x ] \phi_h ) \phi_h dV \\
  & = \int_{\Omega} ((-2 \chi_x h^2 \p_x^2 - h \chi_{xx} h \p_x
  - 2 \chi_y h \p_x h \p_y - h\chi_{yy} h \p_x) \phi_h ) \phi_h dV \\
  & = -2 \int_\Omega (\chi_x h^2 \p_x^2 \phi_h ) \phi_h dV
-\int_{\Omega}  h \chi_{xx}( h \p_x \phi_h )
 \phi_h dV
  + \O(1).
  \end{align*}

Let
\[
I = \int_\Omega  h \chi_{xx}( h \p_x \phi_h )
\phi_h dV.
\]
Even though $\chi_{xx} = \O(h^{-4/3})$, we will nevertheless show $I$
is bounded.  Write
\[
I = \int_{-r}^r \int_{\beta(y)}^r h \chi_{xx} (h \p_x
\phi_h) \phi_h dx dy
\]
and integrate by parts:
\begin{align}
  I & = - \int_{-r}^r \int_{\beta(y)}^r 
  (\phi_h) h \p_x (h \chi_{xx} \phi_h) dxdy + \int_{-r}^r h^2
  \chi_{xx} | \phi_h|^2 |_{\beta(y)} ^r dy\notag  \\
  & = - {I} - h^2 \int_{-r}^r \int_{\beta(y)}^r 
  \chi_{xxx} |\phi_h|^2  dxdy  + \int_{-r}^r h^2
  \chi_{xx} | \phi_h|^2 |_{\beta(y)}^r dy. \label{E:I-101}
\end{align}
Let
\[
I_1 = h^2 \int_{-r}^r \int_{\beta(y)}^r 
\chi_{xxx} |\phi_h|^2  dxdy.
\]
We have $\chi_{xxx} = h^{-2}$, so $I_1 = \O(1)$.  We pause briefly
here to observe that the function $\chi_{xxx}$ still has large support
in the $y$ direction, so we cannot use the $\delta = 1/2$
non-concentration estimate here.  We use that for the next term: let
\[
I_2 = \int_{-r}^r h^2
\chi_{xx} | \phi_h|^2 |_{\beta(y)} ^r dy.
\]
As before, the support properties of $\chi$ and its derivatives tells
us
\[
I_2 = - \int_{-r}^r h^2
\chi_{xx} | \phi_h|^2 (\beta(y), y)  dy.
\]


 \begin{remark}  Away from corners, the bound $I_2 = O(1)$ follows from the universal eigenfunction boundary restriction upper bound in (\ref{Tataru}). Indeed, since $h^2 \chi_{xx} = O(h^{2/3}) \tilde{\chi}_{xx}$,  
$$ I_{2} = O(h^{2/3}) \int_{\partial \Omega} \tilde{\chi}_{xx}| \phi_h |^2  dS = O(1)$$
where the last estimate follows from the Tataru bound $ \int_{\partial \Omega} \tilde{\chi}_{xx} | \phi_h |^2 dV = O(h^{-2/3})$ since $\tilde{\chi}_{xx}$ is supported away from corners. However, since we will need our estimates to hold near corners as well, we give a more direct argument here to bound $I_2.$ \end{remark}

Note that $I_2$ \,  is a boundary integral with support in three different regions in
the $x$ direction.  We have $\chi_{xx} = \O(h^{-4/3})$ for $-3 h^{2/3}
\leq x \leq 3 h^{2/3}$, and $\chi_{xx} = \O(1)$ for $| x | \geq 3
h^{2/3}$.  In the latter region, the boundary integral then has $h^2$,
so Sobolev embedding gives $\O(h)$.  It is on the region $-3 h^{2/3}
\leq x \leq 3h^{2/3}$ where we may encounter a problem. 
Let $[a(h), b(h)]$ be the image in $y$ of $[-3 h^{2/3} , 3 h^{2/3}]$.
Using the   support properties of $\tchi$ and the Fundamental Theorem
of Calculus to relate the boundary integral to an interior integral 
 (similar to a Sobolev estimate), 
  \begin{align}\label{sobolev}
  |I_2| \leq C \int_{[a(h) , b(h)]} h^2 h^{-4/3} | \phi_h|^2 (
  \beta(y), y) dy  \nonumber \\
  \leq C  h^{2/3}\int_{B(p_0, M h^{2/3})} (\p_x | \phi_h |^2 ) dV .
  \end{align}
  

  Here $M>0$ is a constant large enough so that
  \[
  \{ (\beta(y), y) : a(h) \leq y \leq b(h) \} \subset B(p_0, M
  h^{2/3}),
  \]
  But for $h>0$ sufficiently small, $B(p_0, M h^{2/3}) \subset B(p_0 ,
  h^{1/2} )$, so that,  by an application of the non-concentration bound for $\delta =1/2$ 
  and Cauchy-Schwarz,
  
  \begin{align*}
    h^{2/3}& \int_{B(p_0, M h^{2/3})} (\p_x | \phi_h |^2 ) dV \\
    & \leq 2h^{2/3} \int_{B(p_0 , h^{1/2})}h^{-1} | h \p_x \phi_h | |
      \phi_h | dV \\
      & \leq C h^{2/3 -1 +1/2}
      \\
      & = \O(h^{1/6}).
  \end{align*}
  Combining this with the estimate on $I_1$ and plugging into \eqref{E:I-101}, we have
  \[
 2  I = \O(1).
 \]

 Now the computations \eqref{E:boundaryI1}-\eqref{E:comm-sum-1} are identical, including the
 boundary cancellations, leading to
 \begin{align}
    \int_\Omega & ([-h^2 \Delta -1, \chi \p_x ] \phi_h) \phi_h dV 
  \notag \\
  & \quad + \int_\Omega  ([-h^2 \Delta -1, \rho \p_y ] \phi_h) \phi_h
  dV \notag  \\
  & \geq 
   \int_\Omega  h^{-2/3} \gamma(x/h^{2/3})
  \gamma(y/h^{2/3}) | \phi_h|^2 
  dV - \O(1). \label{E:comm-sum-101}
  \end{align}

On the other hand,  expanding the commutator, using the Neumann
boundary conditions, and applying Sobolev embedding as in \eqref{E:boundary-102} yields the exact same identity:
  \begin{align*}
       \int_\Omega  & ([-h^2 \Delta -1, \chi \p_x ] \phi_h) \phi_h dV
       \\
       & = \int_{\p \Omega} \left(\chi \frac{
         \alpha'}{\kappa} h^2 \p_\nu^2 \phi_h \right) \phi_h dS \\
       & \quad + \int_{\p \Omega} \left( \frac{ \alpha'}{\kappa^2} h
       \chi_xh \p_\tau \phi_h \right) \phi_h dS + \O(1).
  \end{align*}
  And again, similar computations give
   \begin{align*}
    \int_\Omega & ([-h^2 \Delta -1, \rho \p_y ] \phi_h ) \phi_h dV \\
    & = - \int_{\p \Omega} \left( \rho \frac{1}{\kappa} h^2 \p_\nu^2
    \phi_h \right) \phi_h dS \\
    & \quad - \int_{\p \Omega} \left( \frac{\alpha'}{\kappa^2} h\rho_y h
    \p_\tau \phi_h \right)\phi_h dS  + \O(1).
   \end{align*}
   Again using the same miraculous cancellation on the boundary terms,
   we finally arrive at
   \[
   \int_\Omega   ([-h^2 \Delta -1, \chi \p_x ] \phi_h) \phi_h dV
   + 
 \int_\Omega  ([-h^2 \Delta -1, \rho \p_y ] \phi_h ) \phi_h dV =
 \O(1).
 \]
 Comparing to \eqref{E:comm-sum-101}, we have
 \[
   \int_\Omega  h^{-2/3} \gamma(x/h^{2/3})
  \gamma(y/h^{2/3}) | \phi_h|^2 
  dV = \O(1),
  \]
  which completes the proof in the case $p_0$ is not a corner.

  We finally remark that we can follow along line by line the proof in the case
  $p_0$ is a corner with similar modifications as   in the case
  $\delta = 1/2$
  to conclude
  the estimate with $\delta = 2/3$ holds at a corner as well.\\

  {\bf Step 3 (induction):} $2/3 < \delta <1.$

Our goal now is to prove that for any integer $k >0$, the theorem is
true for $\delta = 1-1/3k$.  The case $k = 1$ has already been shown,
so we are ready for the induction step.

We will need better control over some of the boundary terms than we
have had previously.  We will employ more or less the same cutoffs, so
the same important cancellation will occur, but it is the ``lower
order'' terms we need to estimate.  The issue is that lower order for
the induction means estimates for $\delta = 1-1/3k$ to prove the
estimates for $\delta = 1-1/3(k+1)$.  Since in these cases $\delta
>2/3$, this is more complicated.

In order to fix the ideas and notations, let $\tchi$ and $\tpsi$ be as
in the start of the proof.  We work initially away from a corner, but
the proof in the corner case follows line by line as the proof in the
$\delta = 1/2$ case, with one notable exception which we shall point
out as we proceed.

Fix $p_0 \in \p \Omega$ away from a corner and rotate and translate as
above so that $p_0 = (0,0)$, and locally $\p \Omega$ is a graph $y =
\alpha(x)$, $\alpha'(0) \neq 0$.  We also write $\beta = \alpha^{-1}$
so that the boundary can also be written $x = \beta(y)$.  
Let $r>0$ be as in the beginning of the proof, a number
independent of $h$  such that $B(p_0, r)$ does not meet any
corners.  Again, this is just to avoid messy numerology when writing
down our integral formulae.

Fix an integer $k>0$ and let
\[
\eta_k = 1-\frac{1}{3k}
\]
be the corresponding index.  Let 
\begin{equation}
  \label{E:chi-def-k}
  \chi = \tchi(x/h^{\eta_{k+1}})\tpsi^2(x/h^{\eta_{k}})
  \tpsi^2(y/h^{\eta_k}).
\end{equation}
We observe that this cutoff has derivative
$\sim h^{-\eta_{k+1}}$ for $x$ in an $h^{\eta_{k+1}}$ neighbourhood,
but is supported in a neighbourhood of size $h^{\eta_k}$.  In
particular, we record the following facts:
\begin{itemize}
  \item $\chi(x,y) = x/2h^{\eta_{k+1}}$ for $-h^{\eta_{k+1}} \leq x
    \leq h^{\eta_{k+1}}$ and $-h^{\eta_k} \leq y \leq h^{\eta_k}$.

  \item
    $\chi$ is supported in $[-2h^{\eta_k} , 2 h^{\eta_k}]^2$.

  \item
    The support of $\chi_x$ has three connected components in $x$:
    \[
    \chi_x = 1/2h^{\eta_{k+1}}, \,\, | x | \leq h^{\eta_{k+1}},
    \]
    and
    \[
    \chi_x = \O(h^{-\eta_{k+1}}), \,\, | x | \leq 3 h^{\eta_{k+1}};
    \]
    \[
    \chi_x = 0, \,\, 3 h^{\eta_{k+1}} \leq | x | \leq h^{\eta_k};
    \]
    and
    \[
    \chi_x = \O(h^{-\eta_k}), \,\, h^{\eta_k} \leq | x | \leq 2 h^{\eta_k}.
    \]
\end{itemize}
The purpose for replacing $\tpsi$ with $\tpsi^2$ will become apparent
shortly.

{\bf Claim:}  For $h>0$ sufficiently small, we have the estimate
\begin{equation}
  \label{E:chi-k-deriv}
\int_\Omega \chi( | h \p_x \phi|^2 + | h \p_y\phi|^2)dV = \O(h^{\eta_k}).
\end{equation}
To prove the claim, we will integrate by parts.  We first get rid of
the $\tchi$ part:
\[
| \chi | \leq \tpsi^2(x/h^{\eta_k}) \tpsi^2(y/h^{\eta_k}).
\]
In order to ease notation, let $\psi_k(x) = \tpsi(x/h^{\eta_k})$ and
similarly for $\psi_k(y)$.  
Then we integrate by parts.  Letting $I$ denote the integral (after
removing the $\tchi$):
\begin{align*}
  I & = \int_\Omega  \psi_k^2(x) \psi_k^2(y)( | h \p_x \phi|^2 + | h
  \p_y \phi|^2)dV \\
  & = \int_\Omega \psi_k^2(x) \psi_k^2(y) (-h^2 \Delta \phi) \phi dV \\
  & \quad - \int_{\Omega} 2
  h^{1-\eta_k} \tpsi'(x/h^{\eta_k})
  \psi_k(x)\psi_k^2(y) (h \p_x \phi) \phi dV \\
  & \quad -  \int_{\Omega} 2
  h^{1-\eta_k} \psi_k^2(x/h)
  \tpsi'(y/h^{\eta_k})\psi_k(y) (h \p_y \phi) \phi dV \\
  & \quad + \int_{\p \Omega} h \psi_k^2(x)
  \psi_k^2(y) ( h \p_\nu \phi) \phi dS.
\end{align*}
The last term is zero due to the Neumann boundary conditions.  For the
remaining terms, observe that $1 - \eta_k >0$ so we can estimate the
second and third terms using Cauchy's inequality:
\begin{align*}
\Bigg| \int_{\Omega} & 2
  h^{1-\eta_k} \tpsi'(x/h^{\eta_k})
  \psi_k(x)\psi_k^2(y) (h \p_x \phi) \phi dV \\
  & \quad +  \int_{\Omega} 2
  h^{1-\eta_k} \psi_k^2(x)
  \tpsi'(y/h^{\eta_k})\psi_k(y) (h \p_y \phi) \phi dV \Bigg|
  \\
  & \leq C h^{1 - \eta_k} \int_{[-2 h^{\eta_k}, 2 h^{\eta_k}]^2}
  (\psi_k^2(x)\psi_k^2(y)| h \p_x \phi|^2 +
\psi_k(y)^2  |\phi|^2) dV \\
  & \quad + C h^{1 - \eta_k} \int_{[-2 h^{\eta_k}, 2 h^{\eta_k}]^2}
  (\psi_k^2(x)\psi_k^2(y)| h \p_y \phi|^2 +
\psi_k(x)^2  |\phi|^2) dV.
\end{align*}
Recall we are assuming the theorem is true for $k$, so we have
\[
\int_{[-2 h^{\eta_k}, 2 h^{\eta_k}]^2}| \phi|^2 dV = \O(h^{\eta_k}).
\]
Collecting terms, we have 
\[
I \leq C h^{1 - \eta_k} I + \O(h^{\eta_k}).
\]
Rearranging proves the claim.

We now use this to control boundary terms.  This is really just a
cheap version of the usual Sobolev embedding, but we write out the
details as it is important for the corner case.

{\bf Claim:}  Let $\zeta(x)$ be a smooth function with support in
$\{ -3h^{\eta_k} \leq x \leq 3h^{\eta_k} \}$, $\zeta \equiv 1$ for $-2
h^{\eta_k} \leq x \leq 2h^{\eta_k}$, and $\p_x^m \zeta =
\O(h^{-m\eta_k})$.  
We have
\[
\int_{\p \Omega} \zeta | \phi |^2 dS =
\O(h^{\eta_k-1}).
\]

To prove this claim, 
 let
\[
I = \int_{\Omega } \zeta(x) \zeta(y/M)
 (h \p_x \phi) \phi dV.
\]
The number $M$ is simply chosen large enough, independent of $h$ so
that the function $\zeta(\beta(y)) \zeta(y/M) = \zeta(\beta(y))$, and $\supp
\zeta(x) \zeta(y/M)$ is in an $h^{\eta_k}$ neighbourhood of $p_0$.
From our first claim and Cauchy's inequality,
\[
| I | = \O(h^{\eta_k}).
\]
Integrating by parts:
\begin{align*}
  I & = \int_{-r}^r \int_{\beta(y)}^r \zeta(x) \zeta(y/M) (h \p_x \phi) \phi dx dy \\
  & = -I -\int_{-r}^r \int_{\beta(y)}^r h \p_x(\zeta(x)
  \zeta(y/M))| \phi|^2 dx dy \\
  & \quad - h\int_{-r}^r  \zeta(\beta(y)) \zeta(y/M)| \phi |^2 dy \\
  & = \O(h^{\eta_k}) + \O(h h^{-\eta_k} h^{\eta_k}) - h\int_{-r}^r  \zeta(\beta(y)| \phi |^2 dy.
\end{align*}
Rearranging proves the claim.  

\begin{remark}
  \label{R:ibp-on-bdy}
We pause now for an important observation which is the only place the
proof in the corner case deviates from the present case.  We will
eventually be estimating boundary integrals such as those with
$h^{-\eta_k}\tpsi'(x/h^{\eta_k}) \psi_k(x)\psi_k^2(y)$ replacing $\zeta(x) \zeta(y/M)$.  Observe that this is supported away from $x = 0$, so
that, if $(0,0)$ is a corner, this is supported away from the corner
so that we can integrate by parts {\it along the boundary}, even in
the corner case.
\end{remark}

We now follow the proof in the $\delta = 2/3$ case.  We compute the
commutator, being very careful for ``lower order terms''.  Recalling
the definition \eqref{E:chi-def-k} of $\chi$:
\begin{align}
  \int_\Omega & ([-h^2 \Delta -1, \chi \p_x ] \phi_h) \phi_h dV \notag
  \\
  & = \int_\Omega ( (-2 \chi_x h^2 \p_x^2 - h \chi_{xx} h \p_x - 2
  \chi_{y} h \p_y h \p_x - h \chi_{yy} h \p_x ) \phi_h ) \phi_h dV.
  \label{E:comm-1001}
\end{align}
Let us examine each term separately.  We have
\begin{align}
  \int_\Omega & (-2 \chi_x h^2 \p_x^2\phi_h ) \phi_h dV \notag \\
  & = \int_{-r}^r \int_{\beta(y)}^r (-2 \chi_x h^2
  \p_x^2\phi_h ) \phi_h dx dy \notag \\
  & = 
\int_{-r}^r \int_{\beta(y)}^r (2 \chi_x |h
\p_x\phi_h|^2  dx dy \notag \\
& \quad + 
\int_{-r}^r \int_{\beta(y)}^r (2 h\chi_{xx} h
  \p_x \phi_h ) \phi_h dx dy
 \notag  \\
  & \quad - 
\int_{-r}^r  2 h\chi_x( h
\p_x\phi_h ) \phi_h |_{\beta(y)}^r  dy.  \label{E:IBP-1001}
\end{align}
The term in \eqref{E:IBP-1001} with $\chi_{xx}$ also shows up in
\eqref{E:comm-1001}.  We know that $\chi_{xx} = \O(h^{-2 \eta_{k+1}})$
and is supported on a set of radius $\sim h^{\eta_k}$, so our first
claim gives
\begin{align*}
  \int_\Omega h \chi_{xx} (h \p_x \phi_h ) \phi_h dV  & = \O( h h^{-2
    \eta_{k+1}} h^{\eta_k})\\
  & = \O(1),
\end{align*}
since
\[
1 - 2 \eta_{k+1} + \eta_k = 1 - 2\left(1- \frac{1}{3(k+1)}\right) + 1
- \frac{1}{3k} = \frac{k-1}{3k(k+1)} \geq 0.
\]

For the two remaining terms in \eqref{E:IBP-1001}, we need to use the
support properties of $\chi_x$.  We have
\begin{align*}
  \chi_x & = h^{-\eta_{k+1}} \tchi'(x/h^{\eta_{k+1}})
  \tpsi^2(x/h^{\eta_k}) \tpsi^2 (y/h^{\eta_k}) \\
  & \quad + 2h^{-\eta_k} \tchi(x/h^{\eta_{k+1}})
  \tpsi'(x/h^{\eta_k})\tpsi(x/h^{\eta_k}) \tpsi^2 (y/h^{\eta_k}) .
\end{align*}
Recalling our function $\gamma(s) = \tchi'(s)$, we have
\[
\chi_x \geq h^{-\eta_{k+1}} \gamma(x/h^{\eta_{k+1}}) -
\O(h^{-\eta_k}),
\]
and let us stress again that the $\O(h^{-\eta_k})$ error term is
supported on scale $h^{\eta_k}$.  Hence we have
\[
\int_\Omega 2 \chi_x |h
\p_x\phi_h|^2  dV \geq h^{-\eta_{k+1}} \int_\Omega
\gamma(x/h^{\eta_{k+1}}) \gamma(y/h^{\eta_{k+1}}) | h \p_x \phi_h |^2
dV - \O(1).
\]

We now examine the boundary term in \eqref{E:IBP-1001}.  This is again
where we must be mindful of any differences between the case with or
without corners.  As in the previous steps in the proof, we will also be
using a commutant with the vector field $\rho \p_y$, where
\begin{equation}
\rho = \alpha'(x)\tchi(\beta(y)/h^{\eta_{k+1}}) \tpsi^2(x/h^{\eta_k})
\tpsi^2(y/h^{\eta_k}).
\label{E:rho-def-1001}
\end{equation}
The same cancellations of boundary terms will happen on the set where
$\rho_y = \chi_x$, which is for $-3h^{\eta_{k+1}} \leq x \leq 3
h^{\eta_{k+1}}$.  For $|x| \geq 3 h^{\eta_{k+1}}$, these functions do
not necessarily agree, but in this region both $\chi_x$ and $\rho_y$
are $\O(h^{-\eta_k})$ rather than $\O(h^{-\eta_{k+1}})$.  Further,
they are supported away from $x = 0$ so that we may further integrate
by parts on the boundary.  That is,
\begin{align*}
 & \int_{-r}^r   (2 h\chi_x h
\p_x\phi_h ) \phi_h |_{\beta(y)}^r  dy
\\
& = -\int_{-r}^r  2 h h^{-\eta_{k+1}}
\tchi'(x/h^{\eta_{k+1}}) \tpsi^2(\beta(y)/h^{\eta_k}) \tpsi^2 (y/h^{\eta_k}) h
\p_x\phi_h ) \phi_h (\beta(y),y) dy \\
&
\quad -\int_{-r}^r  4 h h^{-\eta_{k}}
\tchi(x/h^{\eta_{k+1}}) \tpsi'(\beta(y)/h^{\eta_k}) \tpsi(\beta(y)/h^{\eta_k})\tpsi^2 (y/h^{\eta_k}) h
\p_x\phi_h ) \phi_h (\beta(y),y)  dy.
\end{align*}
The cutoffs in the second term are supported away from $x = 0$, where
$\tchi = \pm 1$.  Let $\tau$ denote the tangent variable so that, as
above,
\[
\p_y \phi_h|_{\p \Omega} = \frac{\alpha'}{\kappa} \p_\tau \phi_h|_{\p
  \Omega}.
\]
Let
\[
\tzeta(y) = \tchi(\beta(y)/h^{\eta_{k+1}})  \tpsi'(\beta(y)/h^{\eta_k}) \tpsi^2
(y/h^{\eta_k}),
\]
and let $\zeta(\tau)$ denote $\tzeta$ in tangent coordinates,
so that $\p_\tau^m \zeta = \O(h^{-m\eta_k})$.  
Then
\begin{align*}
  \int_{-r}^r & 2 h h^{-\eta_{k}}
\tchi(x/h^{\eta_{k+1}}) \tpsi'(\beta(y)/h^{\eta_k})  \tpsi(\beta(y)/h^{\eta_k})\tpsi^2 (y/h^{\eta_k}) h
\p_x\phi_h ) \phi_h (\beta(y),y)  dy \\
& = \int_{\p \Omega}  h^{2 - \eta_k} \zeta(\tau)
\frac{\alpha'}{\kappa} \p_\tau ( | \phi_h|^2) d \tau \\
& = - \int_{\p \Omega}  h^{2 - \eta_k} \p_\tau (\zeta(\tau)
\frac{\alpha'}{\kappa})  | \phi_h|^2 d \tau \\
& = \O(h^{2 - 2 \eta_k} h^{\eta_k-1})
\\
& = \O(1),
\end{align*}
where we have used the second claim and that $\eta_k < 1$ for every
$k$.  Collecting terms, we have
\begin{align*}
  & \int_{-r}^r   (2 h\chi_x h
\p_x\phi_h ) \phi_h |_{\beta(y)}^r  dy
\\
& = -\int_{-r}^r  2 h h^{-\eta_{k+1}}
\tchi'(\beta(y)/h^{\eta_{k+1}}) \tpsi^2(\beta(y)/h^{\eta_k}) \tpsi^2 (y/h^{\eta_k}) h
\p_x\phi_h ) \phi_h (\beta(y),y) dy + \O(1).
\end{align*}

\begin{remark}
We stress again here that this
part of the proof is where we have to be careful if $p_0=(0,0)$ is a
corner.  The above 
integrations by parts would not be
possible near a corner without the observation that the integrand is
supported away from $x = 0$.

\label{R:support}
\end{remark}

We continue with the other two terms in \eqref{E:comm-1001}.  We have
$\chi_y = \O(h^{-\eta_k})$ and $h \chi_{yy} = \O(h^{1-2 \eta_k}) =
\O(h^{-\eta_k})$, and we are integrating over a region of radius
$h^{\eta_k}$, so using our claim,
\[
\int_\Omega ( (- 2
  \chi_{y} h \p_y h \p_x - h \chi_{yy} h \p_x ) \phi_h ) \phi_h dV = \O(1).
  \]

  We now use the vector field $\rho \p_y$ as in
  \eqref{E:rho-def-1001}.  All of the computations are similar, once
  again singling out the boundary terms which are supported near $x =
  0$ but where $\chi_x = \rho_y$ and summing as in the $\delta = 2/3$
  case, we get
  \begin{align*}
    \int_\Omega & ([-h^2 \Delta -1, \chi \p_x] \phi_h ) \phi_h dV +
    \int_\Omega ([-h^2 \Delta -1, \rho \p_y] \phi_h ) \phi_h dV \\
    & = 2 \int_\Omega \chi_x | h \p_x \phi_h |^2 dV + 2 \int_\Omega
    \rho_y | h \p_y \phi_h |^2 dV \\
    & \quad 
-\int_{-r}^r  2 h h^{-\eta_{k+1}}
\tchi'(\beta(y)/h^{\eta_{k+1}}) \tpsi(\beta(y)/h^{\eta_k}) \tpsi (y/h^{\eta_k}) h
\p_x\phi_h ) \phi_h (\beta(y),y) dy \\
& \quad   
+\int_{-r}^r  2 h h^{-\eta_{k+1}}
\tchi'(x/h^{\eta_{k+1}}) \tpsi(x/h^{\eta_k}) \tpsi (\alpha(x)/h^{\eta_k}) h
\p_y\phi_h ) \phi_h (x , \alpha(x)) dx + \O(1) \\
& \geq h^{-\eta_{k+1}} \int_\Omega \gamma(x/h^{\eta_{k+1}})
\gamma(y/h^{\eta_{k+1}}) | \phi_h |^2 dV - \O(1) \\
& \geq \quarter h^{-\eta_{k+1}} \int_{\Omega\cap B(p_0, h^{\eta_{k+1}})}  | \phi_h |^2 dV - \O(1)
\end{align*}

  Finally, we unpack the commutator as in the $\delta = 1/2$ case and
  use the claims and observations above to conclude that
  \[
  \int_\Omega ([-h^2 \Delta -1, \chi \p_x] \phi_h ) \phi_h dV +
    \int_\Omega ([-h^2 \Delta -1, \rho \p_y] \phi_h ) \phi_h dV =
    \O(1).
    \]
    This completes the proof in the case $p_0$ is not a corner.  In
    the case $p_0$ is a corner, we use Remark \ref{R:support} and the rest of
    the proof is identical.

  \end{proof}


\section{Restriction bounds for Dirichlet data: Proof of Theorem \ref{dirichlet} } \label{restriction}

As an application of  Theorem \ref{T:non-con}, we now prove the restriction bounds along totally geodesic boundary components up to corners in Theorem \ref{dirichlet}. 
A key technical component in the proof of Theorem \ref{dirichlet} involves estimating near-glancing mass  based on potential layer theory for
 the boundary data outside of $h^{\delta}$ neignbourhoods of the
 corners. To estimate restriction in $h^{\delta}$ neighbourhoods of
 the corners, we then use  the non-concentration result in Theorem
 \ref{T:non-con} combined with Sobolev estimates. Before carrying out
 the details, we briefly review some of the salient facts needed here
 and refer the reader to \cite{HZ} for further details.
 
\subsection{Potential layers and the boundary jumps equation} 

Let $\Omega \subset \R^2$ be a piecewise-smooth, bounded convex planar domain.
 The free Green's function 
 for  the Helmholtz equation 
 $$ (- \Delta - h^{-2}) G(x,y,h) = \delta_{x}(y), \quad (x,y) \in \R^2 \times \R^2$$
 is given in terms of Hankel functions:
 $$G(x,y,h) = \frac{i}{4} \text{Ha}^{(1)}_{0}( h^{-1}|x-y| ).$$

The corresponding double layer operator $N(h) : C^{0}(\partial \Omega) \to C^{0}(\partial \Omega)$ is given by
\begin{eqnarray} \label{double}
N(h) f(q) &=& \int_{\partial \Omega} N(q,q',h) \, f(q') d\sigma(q') ,\nonumber \\
N(q,q',h) = 2 \partial_{\nu(q)} G(q,q',h) &=& \frac{i}{4} h^{-1} \big\langle \nu(q'), \frac{q-q'}{|q-q'|} \big\rangle \,\cdot \,   \text{Ha}^{(1)}_{1}(h^{-1}|q-q'|), \end{eqnarray}\
where, 
\begin{equation} \label{hankel}
 \text{Ha}^{(1)}_1(z) = \Big( \frac{2}{\pi z} \Big)^{1/2}  \frac{e^{i (z - 3\pi/4) } }{\Gamma(3/2)} \int_{0}^{\infty} e^{-s}  s^{1/2} (1 - \frac{s}{2iz} )^{1/2} \, ds. \end{equation}

Here, and throughout the paper, $\nu(q)$ denotes the unit boundary external
normal at   $ q \in \mathring{ \partial \Omega} = \partial \Omega \setminus {\mathcal C}$.

We recall that the boundary jumps equation says that
\begin{equation}
\label{e:jumps}
u_h(q) = N(h) u_h(q); \quad q \in \partial \Omega
\end{equation}\
where  $N(h)$ is the double layer operator in (\ref{double}).

Let
$$ {\mathcal S} := {\mathcal C} \cup S^* \mathring{\partial \Omega},$$
where the ${\mathcal C} = \cup_{m=1}^N \{ c_{m} \}$ is the set of
corner points and $S^* \mathring{\partial \Omega}$ is the {\em glancing set} of the interior of the boundary faces. 
We will sometimes refer to ${\mathcal S}$ simply as the {\em singular
  set}  and  let  $U$ denote a neighbourhood  of 
  $$ \Xi:=  ( \beta^{-1}(\mathring{B}^*\partial \Omega) )^{c} \times {\mathcal S} \cup {\mathcal S} \times ( {\beta}_{-}^{-1}(\mathring{B}^*\partial \Omega) )^c,$$
where $\beta_{-}:\mathring{B}^*\partial \Omega \to B^*\partial \Omega$ is the backwards billiard map.

We recall  (\cite{HZ} Prop. 4.2)  the following $h$-microlocal decomposition of the double layer operator:
\begin{align} \label{ndecomp}
N(h)&=N_\beta(h)+N_\Delta(h)+N_{{\mathcal S}}(h),
\end{align}
where   $N_{\Delta}(h) \in \Psi^{-1}_h(\partial \Omega )$, a boundary $h$-pseudodifferential operator of order $-1$ (see  section \ref{bdypsdos} for a precise definition) and $ WF_h' N_{\mathcal S}(h) \subset U.$ For our purposes, the most important  part of the double layer is $N_\beta (h) \in I_h^0(\mathring{\partial \Omega}; \Lambda_{\beta}),$  a zeroth-order $h$-Fourier integral operator ($h$-FIO) with canonical relation
$$ \Lambda_{\beta} = \{ (q,\xi,q',\xi') \in B^* \mathring{\partial \Omega}^+ \times B^*\mathring{\partial \Omega}^+; (q',\xi') = \beta(q,\xi) \},$$
where $\beta: B^*\mathring{\partial \Omega}^+ \to B^* \mathring{\partial \Omega}^+$ is the standard billiard map.

With $(q,\xi) \in B^* \mathring{\partial \Omega}^+$ and $(q',\xi') =
\beta(q,\xi),$ the operator $N_{\beta}(h)$ has principal symbol (see \cite{HZ} Prop. 6.1) 
\begin{equation} \label{symbol} 
\sigma(N_\beta)(\zeta,\beta(\zeta))=-i\frac{(1-|\xi|_q^2)^{1/4}}{(1-|\xi'|_{q'}^{2})^{1/4}}\, |dq d\xi|^{1/2}, \quad \zeta = (q,\xi) \in B^*\mathring{\partial \Omega}^+.
\end{equation}

Since $\Lambda_{\beta} \subset T^*\partial \Omega \times T^*\partial \Omega$ is a canonical graph, it follows by the $h$-Egorov theorem   (\cite{Zw} section 11.1)  that

\begin{equation}
N_{\beta}^* \, N_{\beta} \in \Psi_h^0(\partial \Omega), \quad \sigma(N_{\beta}^* \, N_{\beta})(q,\xi) = \frac{(1-|\xi|_q^2)^{1/2}}{(1-|\xi'|_{q'}^{2})^{1/2}}. \end{equation}

In the following, we make the additional assumption that $\Omega$ has a boundary decomposition
$$ \partial \Omega = \Gamma_j  \cup_{k \neq j} \Gamma_k  ,$$
where $\Gamma_j$ is a flat boundary edge and the $\{ \Gamma_k \}_{ k \neq j}$ are the remaining (possibly curved) boundary edges. We will follow the convention that $\Gamma_{j-1}$ and $\Gamma_{j+1}$ are the edges adjacent to $\Gamma_j$ sharing corner points $c_{j}$ and $c_{j+1}$ respectively with $\Gamma_j.$

In this case, the analysis of the operator $N(h)$ simplifies substantially due to the fact that along the flat edge $\Gamma_j,$ the Schwartz kernel 

\begin{equation} \label{vanish}
N(h)(q,q') \equiv 0; \quad (q',q) \in \mathring{\Gamma_j} \times \mathring{\Gamma_j}, \,\, j=1,...,N. \end{equation}
Indeed,  (\ref{vanish})  follows immediately from (\ref{double}) and the fact that
$$ \big\langle \nu(q'), \frac{q-q'}{|q-q'|} \big\rangle \equiv 0, \quad  (q,q') \in \mathring{\Gamma_j} \times \mathring{\Gamma_j}.  $$

 \subsubsection{h-FIO part of the potential layer} \label{hFIO}

It follows  from the integral formula (\ref{hankel}) for the
Hankel function that  the $h$-FIO part of the potential layer $N(h)$ has Schwartz kernel of the form

\begin{equation} \label{wkb}
N_{\beta}(q,q') = (2\pi h)^{-1/2} e^{i |q-q'|/h} c(q,q',h),\end{equation}
where $c(q,q',h) \sim \sum_{j=0}^{\infty} c_j(q,q') h^j, \,\, c_j \in C^{\infty}(\mathring{\partial \Omega} \times \mathring{\partial \Omega})$  when $ |q-q'| \gtrapprox h^{\delta}, \,\, \delta \in [0,1).$  To derive (\ref{wkb}),  one observes that 
the function $b \in C^{\infty}(\R)$ given by 
\begin{equation} \label{intb}
b(x):= \int_{0}^{\infty} e^{-\tau} \tau^{1/2}  \, ( 1 -
\frac{\tau}{2i} x^{-1} )^{1/2} \, d\tau. \end{equation}
has standard conormal asymptotic expansion as $x \to \infty.$ Indeed,
by Taylor  expansion of the integrand  in (\ref{intb}), it follows that
$$ b(x) \sim  \sum_{j=0}^{\infty} b_j x^{-j} \quad \text{as} \quad x \to \infty.$$

Consider the  piecewise-smooth function 

\begin{equation} \label{inta}
a(q,q') := |q-q'|^{-1/2} \, \langle \nu(q'), \frac{q-q'}{|q-q'|} \rangle  = O(|q-q'|)^{-1/2} \end{equation} 
uniformly for $(q,q') \in \partial \Omega.$ In this case, the factor $\langle \nu(q'), \frac{q-q'}{|q-q'|} \rangle  = O(1)$ and no better since  the  boundary normal  jumps at corners.
Then, the WKB expansion (\ref{wkb}) for the Schwartz kernel $N_{\beta}(h)(q,q')$ follows from (\ref{intb}) and (\ref{hankel}). Moreover, it follows that one can write $N_{\beta}(h)(q,q')$ somewhat more succinctly in the form

\begin{equation} \label{inteqn2}
N_{\beta}(h)(q,q')= (2\pi h)^{-1/2} e^{i |q-q'|/h} \, a(q,q') \, b( h^{-1} |q-q'|). 
\end{equation}\

Since in (\ref{intb}), $|1- \frac{\tau}{2i} x^{-1}| \geq 1$ where $x = |q-q'|/h,$ by differentiation under the integral sign  and using the estimates
$$ |\partial_{q,q'}^{\beta} x | = O_{\beta}( x \, |q-q'|^{-|\beta|} ),$$
  it follows that 
 for $(q,q') \in \mathring{\partial \Omega} \times \mathring{\partial \Omega}$ \\
 
\begin{align} \label{symbolbounds0}
\partial_{q,q'}^{\alpha} a(q,q') & = O_{\alpha} ( |q-q'|^{-1/2 - |\alpha| } );  \quad  |q-q'| \lessapprox 1, \\ \nonumber
 \partial_{q,q'}^{\beta} b( h^{-1} |q-q'| ) &=  O_{\beta} (1) \partial_{q,q'}^{\beta}  (x^{-1})= O_{\beta} (|q-q'|^{-|\beta|}); \quad  1 \lessapprox  h^{-1} |q-q'|  \lessapprox h^{-1}. 
\end{align} 
 
Moreover, the derivative estimates in   (\ref{symbolbounds0})  are {\em uniform} for $(q,q') \in \mathring{\partial \Omega} \times \mathring{\partial \Omega}.$ 

We note for future reference that from the Leibniz formula and the derivative estimates (\ref{symbolbounds0}) it follows that
the symbol $c(q,q',h)$ in (\ref{wkb}) can be written in product form 
 $$c(q,q',h):= a(q,q') \cdot b(h^{-1}(|q-q')),$$
where

\begin{equation} \label{symbolbounds}
  |q-q'|^{1/2}  \,\, \partial_{q,q'}^{\alpha} c (q,q',h) = O_{\alpha}
  (|q-q'|^{-|\alpha|}), \quad 1 \lessapprox h^{-1} |
q-q'| \lessapprox h^{-1}. \end{equation} \

From (\ref{symbolbounds}), it follows that
$$ h^{\delta/2} \partial_{q,q'}^{\alpha} c (q,q',h) = O_{\alpha} (h^{-\delta |\alpha|}); \quad |q-q'| \gtrapprox h^{\delta}, \,\, 0 \leq \delta <1, $$
and so, $$  h^{\delta/2} c \in S^{0}_{\delta}(1).$$
The formula in (\ref{inteqn2}) together with the symbolic estimates in (\ref{symbolbounds}) will be used in the next section.

\section{$h$-microlocalized jumps formula: estimates near glancing} \label{defn}

We introduce several cutoff functions at this
point.  As above, let $\Gamma_j$ be a flat boundary edge with corner endpoints $c_j$ and $c_{j+1}.$ More generally, we order all egdes $\Gamma_k: k=1,..,N-1$ in a counterclockwise fashion and let $c_{k}$ (resp. $c_{k+1}$) be the corner endpoints of $\Gamma_k$ adjacent to the edges $\Gamma_{k-1}$ (resp. $\Gamma_{k+1}).$ Throughout the paper, $q:[0,L] \to \partial \Omega$ will denote the  piecewise $C^\infty$ arclength parametrization of the boundary.

Let $\chi \in C^{\infty}_0(\R^2), \,\,  0 \leq \chi \leq 1, $ be a radial cutoff  with $\chi(x) =1$ for $|x| \leq 1$ and $\chi(x) = 0$ for $|x| \geq 2.$ Fix a constant $C_0 = \frac{1}{2} \min_j
|\Gamma_j|,$  and consider the corresponding boundary corner cutoffs $\psi_k: \partial \Omega \to [0,1]$ with

\begin{equation} \label{cornercutoff}
\psi_k(y):= \chi ( C_0^{-1}  ( q(y) - c_k ) ). \end{equation}\

Similarily, the corresponding small-scale boundary corner cutoffs on scales $h^{\delta}$ will  be denoted by $\psi_k^{\delta}: \partial \Omega \to [0,1],$ where

\begin{equation} \label{cornercutoffs}
\psi_k^{\delta}(y;h):=  \chi( C_0  \, h^{-\delta} (q(y) - c_k) ). \end{equation}\

It will also be useful to introduce notation for the sum of all corner cutoffs and so,  we introduce the cutoffs
\begin{equation} \label{cutoffsum}
\psi(y):= \sum_{k=1}^{N} \psi_k(y), \quad \psi^{\delta}(y,h):= \sum_{k=1}^{N} \psi_k^{\delta}(y,h). \end{equation}

   Thus,  $\psi_k^{\delta}: \partial \Omega \to [0,1]$ is a standard
   cutoff supported in an $h^{\delta}$-neighbourhood of the corner point $c_k$ and so, $(1-\psi_k^{\delta})$ is supported outside an $h^{\delta}$-neighbourhood of the corner point $c_k.$

 We continue to assume in the following that $0 \leq 2 \delta <1.$  Then, by Taylor expansion of the integral formula for the  Green's function
       in (\ref{intb}) it follows that

       \begin{equation} \label{asymp1}
       \sup_{\{(q,q'); |q-q'| \gtrapprox h^{2\delta} \}} \big| N(q,q',h) - e^{i |q-q'|/h} a(q,q') b(h^{-1}|q-q'|) \big| = O(h^{\infty}). \end{equation}\

In (\ref{asymp1}),  $b$ and $a$ are defined in (\ref{intb}) and (\ref{inta}) respectively and are, in particular, piecewise-smooth on the off-diagonal set $\{ (q,q') \in \partial \Omega \times \partial \Omega, |q-q'| \gtrapprox h^{2\delta} \}$ up to corner points in $q$ and $q'.$ 

As a special case, it follows from (\ref{asymp1}) that with the corner cutoffs in (\ref{cornercutoffs}) and for all edge indices $k, \ell \in \{ 1,...,N \},$ 
\begin{align} 
\sup_{(q,q') \in \partial \Omega \times \partial \Omega} & \big|
(1-\psi_k^{\delta})(q,h) \big( N(q,q',h) -  e^{i |q-q'|/h} a(q,q')
b(h^{-1}|q-q'|) \big) \psi_{\ell}^{2\delta}(q',h) \big| \notag \\
& = O(h^{\infty}), \label{ASYMP1} \end{align}
and similarily, when $\ell \neq k,$
\begin{align} 
\sup_{(q,q') \in \Gamma_{\ell} \times \Gamma_k} & \big|
(1-\psi_k^{\delta})(q,h) \big( N(q,q',h) -  e^{i |q-q'|/h} a(q,q')
b(h^{-1}|q-q'|) \big) (1- \psi_{\ell}^{2\delta}(q',h) \big| \notag \\
& = O(h^{\infty}). \label{ASYMP2} \end{align}

We will use (\ref{ASYMP1}) and (\ref{ASYMP2}) in  the next section where we  $h$-microlocalize the  jumps equation (\ref{e:jumps}) near the glancing set $S^* \Gamma_j.$ Before doing this, we  introduce some frequency cutoffs to the glancing set $S^* \mathring{\partial \Omega}:$ Specifically, for arbitarily small but fixed $\epsilon_0 >0$, let $\chi_{j} \in C^{\infty}_0(T^* \mathring{\Gamma_j} ) $ with $ 0 \leq  \chi_{j} \leq 1$ and such that $\chi_{j}(\xi) = 1$ when $ 1- \epsilon_0 \leq |\xi| \leq 1 + \epsilon_0$ and $\chi_{j}(\xi) = 0$ provided $ | 1- |\xi || \geq 2 \epsilon_0.$ 
The corresponsding $h$-pseudodifferential cutoffs are  $\chi_j(h):= Op_h(\chi_j) \in \Psi_h^0(\mathring{\Gamma_j}).$ 

Note that since $\Omega$ is convex with non-trivial corners, given
$(q,\xi) \in  \text{supp} \, \chi_j \subset B^*\Gamma_j$ the ray with
basepoint $q \in \Gamma_j$ and (co)-vector $\xi$ intersects  the
adjacent side $\Gamma_k$ {\em transversally}. In addition, the ray
intersects $\Gamma_k$ at a distance $\lessapprox \epsilon_0 $ to the
corner  $c_k$, $k=j-1, j+1.$   More precisely, there exists a constant $C_2=C_2(\alpha_j)>0$ depending only on the angle $\alpha_j$ such that with $$(q',\xi') = \beta(q,\xi), \quad (q,\xi) \in \text{supp} \, \chi_j,$$
and for $\epsilon_0 >0$ sufficiently small,

\begin{equation} \label{transversal}
 | |\xi'|_{q'} - 1 | \geq C_2 \,\, \text{and} \,\, |q' - c_{j} | \leq C_1 \epsilon_0 \,  |q-c_j| \,\,  \text{provided} \,\,\,\, | |\xi|_q - 1| \leq \epsilon_0.
\end{equation}\

Note that in (\ref{transversal}), the basepoint $q \in \Gamma_j$ and so, $q' \in \Gamma_k$ where $k = j-1$ or $k = j+1$ provided $\epsilon_0$ is chosen sufficiently small.

\subsection{Proof of Theorem \ref{dirichlet}: The obtuse case} \label{obtuse}

\begin{proof}

In the following,  when convenient, we will freely use the notation $ u_h^j:= u_h
    {\bf 1}_{\Gamma_j},$ for eigenfunction boundary traces along
    $\Gamma_j.$ To simplify the analysis slightly, we assume in this section that
    the flat edge $\Gamma_j$ intersects  adjacent sides at obtuse
    angles $\alpha_j > \pi/2.$ We observe that in such a case, near
    glancing rays to the flat edge $\Gamma_j$ intersect an adjacent
    side $\Gamma_k, \,\, k = j-1, j+1$, transversally and, after an additional reflection,
    under the admissiblity assumption in Definition \ref{admissible}
    on the interior angles, intersects the boundary $\partial \Omega$
    transversally and away from corners (see Figure \ref{F:Fig-1}).  This is a key observation in our analysis
    below.    We first give the proof of Theorem \ref{dirichlet} in
    the case where all angles at corners  $c_{j}$ and $c_{j+1}$ adjacent to the flat side $\Gamma_j$  are obtuse (this assumption
    allows for a technically somewhat simpler argument). Finally, we
    indicate the fairly  minor changes necessary for the proof in the
    general case in subsection \ref{acute}.

\begin{figure}
\hfill
\centerline{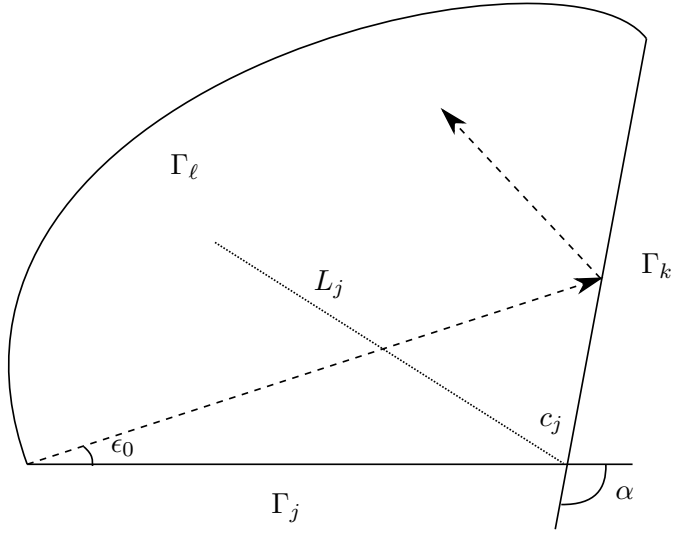}
\caption{\label{F:Fig-1} The setup for obtuse angles.}
\hfill
\end{figure}

In the following, we fix $\delta \in (0, 1/2).$ 
We will eventually estimate the near corner mass   $\| \psi_j^{\delta} u_{h}
\|_{L^2(\Gamma_j)}$   separately using the non-concentration estimate in
Theorem \ref{T:non-con}.   To estimate   $\| (1 - \psi_j^{\delta} ) u_h \|_{L^2(\Gamma_j)},$ in view of (\ref{e:jumps}) and (\ref{vanish}), one can write with $u^{j,k}_h = {\bf 1}_{\Gamma_{j,k}} u_h,$

\begin{equation} \label{inteqnrev}
(1-\psi_j^{\delta}(h) )u_h^j = \sum_{k \neq j} (1 - \psi_j^{\delta}) {\bf 1}_{\Gamma_j} N(h) \,  u_h^k.
\end{equation}\

Next, an application of the frequency cutoff operator  $\chi_{j}(hD): C^{\infty}(\mathring{\Gamma_j}) \to C^{\infty}(\mathring{\Gamma_j})$  to both sides of (\ref{inteqnrev}) gives
\begin{equation} \label{inteqn0.5}
\chi_j(h D)  (1- \psi_{j}^{\delta}(h)) u_h^j = \sum_{k \neq j} \chi_j(hD)  \, (1-\psi_j^{\delta}(h)) {\bf 1}_{\Gamma_j} N(h) \, u_h^k + O(h^{\infty}).
\end{equation}
 
 As for the non-glancing mass, the small-scale  Rellich commutator result in Lemma \ref{rellich} shows that

\begin{equation} \label{nonglancingmass}
\| [1-\chi_j (hD) ]\, (1-\psi_j^{\delta}(h)) u_h^{j} \|_{L^2(\Gamma_j)} = O(h^{-\delta/2}) = O(h^{-1/4 -0}). \end{equation}\

Since $\Gamma_j$ is assumed to be flat, we note that here $\chi_j(hD)$ is a {\em tangential} $h$-psdo (see Definition \ref{bdypsdos}) acting on the boundary components and consequently, the symbol $\chi_j(\xi)$ depends only on frequency coordinates in $T^* \mathring{\Gamma_j}.$ 

 The next step is to insert the small-scale corner cutoffs $\psi_k^{2\delta}(h)$ in (\ref{cornercutoffs}) on the RHS of (\ref{inteqn0.5}).  This gives 

\begin{eqnarray} \label{geo/diff}
\chi_j(h D)  (1- \psi_{j}^{\delta}(h)) u_h^j =  \, N_j^{\mathcal G}(h)  u_h 
+ \,  N_j^{\mathcal D} u_h  + O(h^{\infty}).
\end{eqnarray}\

where,
\begin{eqnarray} \label{geo/diff defn}
N_j^{\mathcal G}(h):= \sum_{k \neq j} \chi_j(hD)  \, (1-\psi_j^{\delta}(h)) {\bf 1}_{\Gamma_j} N(h) \, (1-\psi_k^{2\delta}(h))  {\bf 1}_{\Gamma_k}, \nonumber \\
 N_j^{\mathcal D}(h):= \sum_{k \neq j} \chi_j(hD)  \, (1-\psi_j^{\delta}(h)) {\bf 1}_{\Gamma_j} N(h) \, \psi_k^{2\delta}(h)  {\bf 1}_{\Gamma_k}.  \hspace{1cm}\end{eqnarray}\\

 In the following,  refer to $N_j^{\mathcal G}(h)$ as the {\em geometric} part of the potential layer and $N_{j}^{\mathcal D}(h)$ as the {\em diffractive} part.
Consistent with this terminology, we can write 
$$N_j^{\mathcal G,\mathcal D}(h) := \sum_{k \neq j} N_{jk}^{\mathcal G,\mathcal D}(h),$$
where we will refer to $N_{jk}^{\mathcal G}$ (resp. $N_{jk}^{\mathcal D}$) as the {\em geometric} (resp. {\em diffractive}) transfer operators.

\begin{remark} \label{FIO}
We note that from (\ref{ASYMP1}) and (\ref{ASYMP2}) it follows that, modulo $O(h^\infty)$-error, one can replace $N(q,q',h)$ with its WKB  expansion in (\ref{inteqn2}) in {\em both} the geometric (resp. diffractive) operators $N_{j}^{\mathcal G}$ (resp. $N_j^{\mathcal D}$) in (\ref{geo/diff}). It then follows that the geometric part $N_{j}^{\mathcal G}(h): C^{\infty}(\mathring{\partial \Omega}) \to C^{\infty}(\mathring{\Gamma_j})$ is an h-FIO of order zero with  a small-scale symbol in $h.$\end{remark}\

\subsubsection{Bounds for the geometric operators}
In this section, we estimate the geometric term $\| N_j^{\mathcal G} u_h \|_{L^2(\Gamma_j)}$ by analyzing the geometric transfer operators $N_{jk}^{\mathcal G}, \,\, k \neq j$ appearing the in the decomposition (\ref{geo/diff}) in more detail.  We begin with the following



\begin{lemma} \label{MCL} For sufficiently small choice of
  $\epsilon_0>0$ in the frequency cutoff $\chi_j(\xi)$ along
  $\Gamma_j,$  there exist uniform constants $C_1>0$ and $C_2>0$
  depending only on the domain $\Omega$ such that
  when $k  \in \{ j-1,j+1 \},$
 \begin{align}
 WF_h'&(N_{jk}^{\mathcal G}(h)) \subset \big\{ (y,\xi;y'\xi') \in B^*\mathring{\Gamma_j} \times B^*\mathring{\Gamma_k}; \, |q(y)-c_j| \geq C_1 h^{\delta},  \nonumber \\
&  |q(y') - c_j| \leq C_2 \epsilon_0 \,  |q(y)-c_j|, \,\, | |\xi'|_{y'}
 - \cos (\pi - \alpha_k) | \leq C_3 \epsilon_0,\,\, |y-c_k| \leq C_2
 \epsilon_0 \big\}.
  \label{WF}
 \end{align}

 Moreover, when $ k \notin \{ j-1, j+1 \}$,
$$   WF'_h N_{jk}^{\mathcal G}(h)  = \emptyset.$$

\end{lemma}

\begin{proof}
 First, note that for the geometric transfer operators
 $N_{jk}^{\mathcal G},$ in view of the corner cutoffs
 $(1-\psi^{\delta})$ and $(1-\psi^{2\delta})$  appearing  in the
 definition in (\ref{geo/diff defn}), it follows that  the
 Schwartz kernel 

 $$  \text{supp} \, S.K. N_{jk}^{\mathcal G}  \subset \{(q,q') \in \mathring{\Gamma_j} \times \mathring{\Gamma_k} ; |q-q'| \gtrapprox h^{\delta} \}; \quad k \neq j, $$ 
so that, modulo $O_{C^\infty}(h^{\infty})$-errors, one can use the WKB
type formula in (\ref{inteqn2}) for $N(q,q',h).$ Let $[0, \ell_j] \ni s
\to q(s)$ with $\ell_j = | \Gamma_j |$ be arclength parametrization of $\Gamma_j$ and $[0,\ell_k] \ni t \to q(t)$ be arclength parametrization of a boundary edge $\Gamma_k$ with $k \neq j.$ Then, in terms of these parametrizations the geometric transfer operator $N_{jk}^{\mathcal G}(h) = \chi_j(hD) (1-\psi_j^{\delta}) 1_{\Gamma_j} N(h) (1-\psi_k^{2\delta}) 1_{\Gamma_k}$ has Schwartz kernel of the form
 
  \begin{equation} \label{ibp1}
 (2\pi h)^{-3/2} \int \int e^{i  [ (s-s') \xi + |q(s') - q(t)| ] /h} \chi_j(\xi) (1-\psi_j^{\delta})(s,h) c(s',t,h) (1-\psi_k^{2\delta})(t,h) \, ds'  \, d\xi. \end{equation} \
 
 An application of stationary phase in (\ref{ibp1}) in the $(s',\xi)$-variable gives
 \begin{equation} \label{ibp2}
 N_{jk}^{\mathcal G}(s,t,h)= (2\pi h)^{-1/2}  e^{i |q(s)-q(t)|/h} (1-\psi_j^{\delta})(s,h)  \tilde{c}(s,t,h)  (1-\psi_k^{2\delta})(t,h)+ O(h^{\infty}), \end{equation}\
 where (see (\ref{symbolbounds})),  
  $$| q(s) - q(t) |^{1/2} \tilde{c}(s,t,h) \in S^{0}_{\delta}(1),$$
 $$ | q(s) - q(t) |^{1/2}  \tilde{c}(s,t,h) \sim \sum_{m=0}^{\infty}
 c_m \Big( \frac{h}{|q(s)-q(t)|} \Big)^{m},$$ 
 and with $\rho(s,t) = \frac{q(s)-q(t)}{|q(s)-q(t)|},$
 
 $$\tilde{c}(s,t,h) - c(s,t,h) \, \chi_j ( \rho(s,t) ) \in h S^{0}(1),$$
 
  $$\text{supp} \, \tilde{c} \subset \{ (s,t);  \big\langle \rho(s,t), d_s q(s) \big\rangle = 1 + O(\epsilon_0) \}.$$\
 
Note that stationary phase in $(s',\xi)$ does not involve differentiation of the small-scale cutoffs $(1-\psi_j^{\delta})$ and $(1-\psi_k^{2\delta})$, so there are no additional powers of $h^{-\delta}$ appearing in (\ref{ibp2}). We note that the support condition $\big\langle \rho(s,t), q'(s) \big\rangle = 1 + O(\epsilon_0)$ above on $\tilde{c}$ implies  by convexity of $\Omega$ that
$$(t,s) \in \text{supp} \, \tilde{c} \implies  |q(t)-c_k|  = O(\epsilon_0) |q(s)-c_k|, \quad k=j-1,j+1.$$
In particular, modulo $O(h^\infty)$ error it is enough to take $k=j-1,j+1$ in the sum (\ref{geo/diff defn}).

 Next, let $\chi_k^{tr} \in C^{\infty}_0 (\mathring{B}^*\Gamma_k)$ along an adjacent side. Then,  in view of (\ref{ibp2}),

\begin{align}
N^{\mathcal G}_{jk}(h)&  \chi_k^{tr}(t,hD_t) (s,t') \hspace{3.7in}
\notag \\ 
& = (2\pi h)^{-3/2} \int \int e^{i [ |q(s)-q(t)|  + (t-t') \eta ] /h}
\tilde{c}(s,t,h) (1-\psi_k^{2\delta})(t,h) \chi_k^{tr}(t,\eta) \, dt
d\eta \notag \\ & + O(h^{\infty}). \label{ibp3} \end{align}

We apply stationary phase in (\ref{ibp3}) in the $(t,\eta)$-variables. Since there is a 2-microlocal cutoff $(1-\psi_k^{2\delta})(t,h)$ that gets differentiated in the process, one must take some additional care at this point. The formal expansion of the RHS in (\ref{ibp3}) is then of the form
\begin{equation} \label{compare1}
(2\pi h)^{-1/2} e^{i |q(s) - q(t')|/h} \Big( \sum_{m=0}^{\infty}
  \frac{ (hD_t D_\eta)^{m}}{m ! }  [ \tilde{c}(s,t,h) (1-\psi_k^{2\delta})(t,h) \chi_k^{tr}(t,\eta) ] |_{t=t',\eta = \langle \rho(t',s), d_{t'} \rangle} \Big) \end{equation}\

We note that at most $m$ derivatives in $t$ hit the cutoff $(1-\psi_k^{2\delta})$ and each $D_t$ derivative  creates  a factor of $h^{-2\delta},$ so that the error term in (\ref{compare1}) at level $m$ is $O(h^{(1-2\delta)m})$ and it is then standard to show that  the asymptotic expansion is indeed legitimate (see subsection \ref{statphase} for a closely-related argument). It follows that

\begin{equation} \label{compare2}
\eta = \langle \rho(t',s), d_{t'}q(t') \rangle, \quad \langle \rho(t',s), d_s q(s) \rangle = 1 + O(\epsilon_0). \end{equation}\

Since the edges $\Gamma_j$ and $\Gamma_k$ intersect at angle $\alpha_k$, (\ref{compare2}) implies that
$$\eta = \cos(\pi -\alpha_k) + O(\epsilon_0),$$ 
and that completes the proof of the lemma.

\end{proof}

As a consequence of Lemma \ref{MCL},  by choosing $\epsilon_0 >0$ sufficiently small in the support of  the frequency cutoff $\chi_j(\xi'),$ it is natural to introduce a corresponding frequency cutoff $\chi_k^{tr} \in C^{\infty} (B^*\Gamma_k)$ supported transversally along the edge $\Gamma_k$. Specifically we choose $\chi_k^{tr} \in C^{\infty}(B^* \Gamma_k)$ such that

 $$  \text{supp}  \, \chi_{k}^{tr} \subset  \{ (q',\eta) \in   B^* \mathring{\Gamma_k} ;   |\eta| = \cos(\pi -\alpha_k) + O(\epsilon_0) \}.$$ \
  
 In addition, let $\chi_{\epsilon_0} \in C^{\infty}_{0}(\R)$ be a cutoff with $0 \leq \chi_{\epsilon_0} \leq 1$ and
 $\chi_{\epsilon_0}(u) = 1 $ for $|u| \leq \epsilon_0$ with
 $\chi_{\epsilon_0}(u) = 0$ for $|u|  \geq 2\epsilon_0.$  Next, setting
\begin{equation} \label{transcutoff} 
 \chi^{tr}_{k,q}(q',\eta):= \chi_k^{tr}(q',\eta) \, \, \chi_{\epsilon_0} \Big( \frac{|q'-c_k|}{|q-c_k|} \Big), 
\end{equation}
we have, in view of Lemma \ref{MCL},

\begin{eqnarray} \label{inteqn0.7}
N_{j}^{\mathcal G}(h) u_h
  =  \sum_{k=j-1}^{j+1} \chi_j(q,hD)  \, (1-\psi_j^{\delta}(q,h)) {\bf 1}_{\Gamma_j} N(h) \, \chi^{tr}_{k,q} (q',\eta) \, (1- \psi_k^{2\delta}(q',h) {\bf 1}_{\Gamma_k}   u_h \nonumber \\
 + O(h^{\infty}). \,\,\,\,
\end{eqnarray}

At this point, it is useful to introduce some additional transfer operators that are closely related to $N_{jk}^{\mathcal G}$ and $N_{jk}^{\mathcal D}$ (see (\ref{geo/diff defn}). Let $N_{jk}(h) : C^0(\Gamma_j) \to C^0(\Gamma_k)$ be the {\em transfer operator} with Schwartz kernel

\begin{eqnarray} \label{TRANSFER}
N_{jk}(q,q',h) := \chi_j(q,hD)  (1-\psi_j^{\delta})(q,h) N(q,q',h) ; \, \, (q,q') \in \Gamma_j \times \Gamma_k. \end{eqnarray}\

   We note that since the cutoff $\psi_{k}^{2\delta}(t)$ in the incoming $t$-variables is unaffected by differentiation in the $(s',\xi)$-variables in the stationary phase argument in (\ref{ibp1}), it follows  that  when $ (q,q') \in \Gamma_j \times \Gamma_k,$
 \begin{eqnarray} \label{TRANSFERcut}
N_{jk}(q,q',h) 
= \chi_j(q,hD)  (1-\psi_j^{\delta})(q,h) N(q,q',h) \, \chi_{\epsilon_0} \Big( \frac{|q'-c_k|}{|q-c_k|} \Big)  + O(h^{\infty}).\end{eqnarray}\

The point behind (\ref{TRANSFERcut}) is that near-glancing rays to
$\Gamma_j$ intersect adjacent sides $\Gamma_k$ {\em only} and do so
transversally with $|q'-c_i| \lessapprox \epsilon_0 |q-c_k|.$ However,
since $N_{jk}(h)$ incorporates {\em both} diffractive and geometric
terms,  
the transversal cutoff $\chi_k^{tr}(q',hD); \, k=j-1,j+1$ cannot be added in (\ref{TRANSFERcut}) in contrast with the geometric transfer operators $N_{jk}^{\mathcal G}(h)$ in Lemma \ref{MCL}.

In summary, we collect here for future reference the simple relation between the various transfer operators:

\begin{eqnarray} \label{TRANSFERrelate}
N_{jk}^{\mathcal G}(h) = N_{jk}(h)  \chi^{tr}_{k}(q',hD) (1-\psi_{k}^{2\delta})(h), \hspace{.1in} \\ \nonumber \\ \nonumber
N_{jk}^{\mathcal D}(h) = N_{jk}(h) \psi_{k}^{2\delta}(h). \hspace{1.2in}
  \end{eqnarray}\

\begin{remark} Note that in (\ref{inteqn0.7}), the sum on the RHS is only over the sides $\Gamma_{k}; k=j-1,j+1$  adjacent to the flat side $\Gamma_j$ and $\chi^{tr}_{k}(q',\eta)$ in (\ref{transcutoff})  a uniformly  {\em transversal} frequency cutoffs along an adjacent side $\Gamma_k$   with supp \, $\chi_k^{tr} \subset \{(q',\eta) \in B_0^* (\Gamma_k); | \eta - \cos(\pi -\alpha_k) | = O(\epsilon_0) \}$ and
 $ \cos (\pi -\alpha_k) + O(\epsilon_0) < 1$ for $\epsilon_0 >0$ small enough.
The heuristics here are quite simple: the  formula in  (\ref{inteqn0.7}) is a consequence of the fact that, due to the convexity of $\Omega,$ the ray corresponding to  $\xi \in \text{supp} \,\chi_j $ sufficiently close to glancing along $\Gamma_j$ (i.e. with $\epsilon_ 0 >0$ sufficiently small), necessarily hits only the sides $\Gamma_k; k=j-1,j+1$ adjacent to the flat side $\Gamma_j$. Moreover, all such {\em near-glancing} rays to $\Gamma_j$ hit the adjacent sides $\Gamma_{j-1}$ and $\Gamma_{j+1}$ at a distance $\lessapprox \epsilon_{0} $ to a common corner and reflect in a  (uniformly in $\epsilon_0$) {\em transversal} direction to the adjacent side roughly at angle $\pi - \alpha_k$ when $\epsilon_0 >0$ is small (see Figure 4). \end{remark}

To summarize, from (\ref{geo/diff}) and (\ref{inteqn0.7}) we have shown that

\begin{eqnarray} \label{decompupshot}
(1-\psi_j^{\delta}) \chi_j(hD) u_j  \hspace{4in} \nonumber \\
= \sum_{k=j-1}^{j+1} \chi_j(q,hD)  \, (1-\psi_j^{\delta}(q,h)) {\bf 1}_{\Gamma_j} N(h) \,  \chi_{k}^{tr}(q',hD)) \,  (1- \psi_k^{2\delta}(q',h) {\bf 1}_{\Gamma_k}   u_h \nonumber \\
+ N_{j}^{\mathcal D}(h) u_j + O(h^{\infty}). \hspace{3in}\end{eqnarray}\

Then, by choosing another transveral cutoff $\zeta_{k} \in C^{\infty}_0(\mathring{B}^*\Gamma_k)$ with $\zeta_{k} \Supset \chi_{k}^{tr},$ and so that supp  $ \zeta_k  \subset \{ (q,\eta) \in B_0^* \Gamma_k;   | \eta - \cos(\pi-\alpha_k) | = O(\epsilon_0) \},$   the microlocal decomposition in (\ref{decompupshot}) can be written in terms of the transfer operators $N_{jk}$ in (\ref{TRANSFERcut})  as follows:

\begin{eqnarray} \label{decompupshot2}
(1-\psi_j^{\delta}) \chi_j(hD) u_j   \hspace{3in}\nonumber \\
= \sum_{k=j-1}^{j+1} N_{jk}^{\mathcal G}(h) \, {\bf 1}_{\Gamma_k}   u_h 
+ N_{j}^{\mathcal D}(h) u_j + O(h^{\infty}) \hspace{1.7in} \nonumber \\
= \sum_{k=j-1}^{j+1} N_{jk}(h) \, \zeta_{k}(q',hD)  (1- \psi_k^{2\delta}(q',h)) {\bf 1}_{\Gamma_k}   u_h 
+ N_{j}^{\mathcal D}(h) u_j + O(h^{\infty}). \end{eqnarray}\

\begin{remark}
Setting $\delta = 1/2-0$ in (\ref{decompupshot2}), our aim here is to show that
\begin{equation} \label{keybound}
 \| (1-\psi_{j}^{\delta}(h)) \chi_j(hD) u_h^j \|_{L^2(\Gamma_j)} = O(h^{-\delta/2}) = O(h^{-1/4-0}). \end{equation}\
 
When $\rho \in C^{\infty}(\mathring{\Gamma_k})$ is supported away from corners, a standard Rellich commutator argument (see Lemma \ref{rellich}) shows that with a transversal frequency cutoff $\zeta_k \in C^{\infty}_{0}(\mathring{B}^*\Gamma_k),$ one has $\| \rho(y) \zeta_k(hD_y) \|_{L^2(\Gamma_k)} = O(1).$ Unfortunately, 
 since $(1-\tilde{\psi}_{k}^{2\delta}(h)$ is only supported outside an $h^{2\delta}$-neighbourhood of  a corner, at present, we cannot  rule out blowup  in $h$; indeed, from Lemma \ref{rellich}, at present  the best bound we can get near corners is of the form
 $ \| (1-\psi_k^{2\delta}) \zeta_k(hD_y) u_h \|_{L^2(\Gamma_k)} = O(h^{-\delta/2}) = O(h^{-1/4-0}).$ Unfortunately, as we show in section \ref{transfer} the transfer operators $N_{jk}(h)$ are {\em singular} $h$-FIO's associated with one-sided folds and with small-scale (in $h$) symbols. As a result, they are {\em not} bounded in $L^2.$ We show in section \ref{transfer} that
 $ \| N_{jk}(h) \|_{L^2 \to L^2} = O(h^{-1/4-0}).$ Consequently, the naive estimate  for the geometric term in (\ref{decompupshot2}) is
 $$ \| N_{j}^{\mathcal G}(h) u_h \|_{L^2(\Gamma_j}) = O(1) \| N_{jk}(h) \|_{L^2 \to L^2} \,\,  \| \zeta_{k}(q',hD)  (1- \psi_k^{2\delta}(q',h)) {\bf 1}_{\Gamma_k}   u_h  \|_{L^2(\Gamma_k)}  $$
 $$= O(h^{-1/2-0}).$$
 This is just the Sobolev bound and is too crude to be useful.
 
  To deal with this problem, we use the jumps equation $u_h = N(h) u_h$ in the geometric term on the RHS of  (\ref{decompupshot2}) yet again to reflect  near-glancing rays to $\Gamma_j$  hitting $\Gamma_k$  away from the corners  along the adjacent edge $\Gamma_k$ (see Figure 4).
 The point is that by choosing $\epsilon_0>0$ sufficiently small, under the admissibility assumption on the corner angles, these reflected rays have the property  that they next intersect the boundary $\partial \Omega$ transversally in the interior of the boundary  away from a {\em fixed} (in $h$) neighbourhood of the corners. The latter (transversal) $L^2$ mass is then shown to be $O(1)$ by Lemma \ref{rellich}. We now carry out the details of this additional step.
 
 \end{remark}

Inserting the jumps equation $u_h^{\partial \Omega} = N(h) u_h^{\partial \Omega}$ yet again in the first (geometric) term on the RHS of (\ref{decompupshot2})  gives

 \begin{eqnarray} \label{inteqn2.0}
\chi_j(h D)  (1-\psi_{j}^{\delta}(h)) u_h^j  = \sum_{k=j-1,j+1} N_{jk}^{\mathcal G}(h) u_h^k  + N_{j}^{\mathcal D}(h)  u^{\partial \Omega}_h   +  O(h^{\infty})  \quad \nonumber \\
=  \sum_{k=j-1}^{j+1} N^{\mathcal G}_{jk}(h)   \, N(h) \, \, u_h^{\partial \Omega}
+ O( \| N_{j}^{\mathcal D}(h)  u_h^{\partial \Omega}\|_{L^2} )   + O(h^{\infty}). \end{eqnarray}\

 The diffractive term $\| N_j^{\mathcal D}u_j \|$ is easier to estimate, so we defer this to section \ref{diffraction}. We first address the problem of bounding the geometric term. As we have already indicated earlier, the main point behind using the jumps equation $N(h) u_h = u_h$ yet again in (\ref{inteqn2.0}) is that near-glancing rays to the flat edge $\Gamma_j$ intersect near corners (and Rellich estimates in $h^{\delta}$-nbds of corners are quite delicate). To avoid this issue, by inserting the jumps equation yet again (and using admissibility assumption), one essentially reflects away these rays to the interior of the boundary $\partial \Omega$ far from any corners. The latter can then estimated by the  standard Rellich argument in Lemma \ref{rellich} (see also Figure 4). 

Next, we split up   the RHS of (\ref{inteqn2.0})  by inserting additional  cutoffs $\tilde{\psi}^{2\delta}(h)$ in (\ref{inteqn2.0}) where $\tilde{\psi}^{2\delta} = \sum_{k} \tilde{\psi}_k^{2\delta} $ with $\tilde{\psi}_k^{2\delta} \Subset \psi_k^{2\delta}$ is supported in the union of $h^{2\delta}/2$-radius balls centered at each of the corners and $\tilde{\zeta} \in C^{\infty}(\mathring{B}^* \Gamma_k)$ with $\tilde{\zeta}_k \Supset \zeta_k.$  We continue to choose this frequency cutoff so that supp \, $\tilde{\zeta}_k  \subset \{ (q,\eta) \in B_0^* \Gamma_k;   | \eta - \cos(\pi-\alpha_k) | = O(\epsilon_0) \}.$ At this point,  the admissibility assumption in Defintion \ref{admissible} will play a crucial role to ensure that rays reflected in the adjacent edge $\Gamma_k$ intersect $\partial \Omega$ away from corners.

From (\ref{inteqn2.0}), we can write 
\begin{align} 
\| \chi_j & (h D)  (1-\psi_{j}^{\delta}(h)) u_h  \|_{L^2(\Gamma_j)}  \notag \\
 & \leq \sum_{k=j-1}^{j+1} \| N_{jk}^{\mathcal G}(h) \|_{L^2 \to L^2} \,  \Big(  \| \tilde{\zeta}_k(h D) \, (1-\tilde{\psi}_{k}^{2\delta}(h)) \, {\bf 1}_{\Gamma_k} \, N(h) \, (1 - \tilde{\psi}^{2\delta}(h)) u_h \|_{L^2(\Gamma_k)}   \notag \\
 & \quad +  \| \tilde{\zeta}_k(h D) \, (1-\tilde{\psi}_{k}^{2\delta}(h)) \, {\bf 1}_{\Gamma_k} \, N(h) \,  \tilde{\psi}^{2\delta}(h) u_h \|_{L^2(\Gamma_k)} \, \Big) \notag \\
 & \quad + O(  \| N_j^{\mathcal D}(h) u_h \|_{L^2}) +
 O(h^{\infty}). \hspace{2.4in}
\label{CRUCIAL}
\end{align}

 In (\ref{CRUCIAL}) and below, we write $\| N_{jk}(h) \|:= \| N_{jk}(h) \|_{L^2 \to L^2}$ We now estimate each of the two geometric terms on the RHS of (\ref{CRUCIAL}) separately and then bound $\| N_j^{\mathcal D}(h)u_h \|$ separately in section \ref{diffraction}.

To bound the terms on the RHS of (\ref{CRUCIAL}) it is convenient to introduce some notation at this point. We set

\begin{eqnarray} \label{qops}
Q_1(h):=  \zeta_k(hD) \, (1- \psi_k^{2\delta}(h)) N(h) (1-\tilde{\psi}^{2\delta}(h)), \nonumber \\ \nonumber \\
Q_2(h):=  \zeta_k(hD) \, (1- \psi_k^{2\delta}(h)) N(h) \tilde{\psi}^{2\delta}(h)). \hspace{1cm}\end{eqnarray} \

Next, we further decompose the operators $Q_{1,2}(h)$ into {\em
  near-diagonal} and {\em off-diagonal} terms as follows: Let $\chi_M
\in C^{\infty}_{0}(\R^2), \,\, 0 \leq \chi \leq  1$  with $\chi_M(x) = 1$ when $|x| < \frac{1}{M}$ and
$\chi_M(x) = 0$ for $|x| \geq \frac{2}{M}.$  Here, we choose $M>0$
large enough so that $|q(y)-q(y')| \leq \frac{1}{M} h^{2 \delta}$ and
$(y,y') \notin \text{supp}  \,  \psi^{2\delta}_k \,    \times  \, \text{supp} \, \tilde{\psi}^{2\delta}
 $ implies that $(q,q') \in \Gamma_k \times \Gamma_k$ (i.e. both
points lie along the same edge, $\Gamma_k.$)  We then decompose the operators $Q_{1,2}(h)$ by writing  
$$Q_j(h) = Q_{j}^{(1)}(h) + Q_{j}^{(2)}(h); \quad j=1,2,$$
such that
$$ Q_1^{(j)}(h) = \zeta_k(hD) N_{1}^{(j)}(h), \quad j=1,2,$$
where
\begin{align} \label{n1}
 N_1^{(1)}(z,y',h) & =  \Big[ (1- \psi_k^{2\delta}(h)) N(h) (1-\tilde{\psi}^{2\delta}(h)) \Big](z,y')  \cdot \chi_M( h^{-2\delta}(q(z)-q(y')) ), \\
N_1^{(2)}(z,y',h) & = \Big[ (1- \psi_k^{2\delta}(h)) N(h)
  (1-\tilde{\psi}^{2\delta}(h) )\Big](z,y') \cdot (1 - \chi_M)(
h^{-2\delta}(q(z)-q(y')) ).\notag  \end{align}

Similarly, 
$$Q_2^{(j)}(h) = \zeta_k(hD) N_2^{(j)}(h): \quad j=1,2,$$
where
\begin{eqnarray} \label{n2}
N_2^{(1)}(z,y',h) = \Big[ (1- \psi_k^{2\delta}(h)) N(h) \tilde{\psi}^{2\delta}(h) \Big] (z,y')\cdot \chi_M( h^{-2\delta}(q(z)-q(y')) ), \nonumber \\
 N_2^{(2)}(z,y',h) = \Big[ (1- \psi_k^{2\delta}(h)) N(h) \tilde{\psi}^{2\delta}(h) \Big](z,y') \cdot (1- \chi_M)( h^{-2\delta}(q(z)-q(y')) ). \end{eqnarray} \

We note here that by choosing $M \gg 1,$

\begin{equation} \label{no kernel}
N_2^{(1)}(z,y',h) = 0 \end{equation}
and so, without loss of generality it suffices to consider only the $N_2^{(2)}$-term when considering $Q_2(h)$.

\subsubsection{ Estimating  $\| Q_1(h) u_h \|$} \label{statphase}

\noindent  {\bf $Q_1^{(2)}$-term:}    We start with analysis of the
$Q_1^{(2)}$-term.   Since in this case, $|q(z)-q(y')| \gtrapprox h^{2\delta}$  for $(z,y') \in \text{supp} \, Q_1^{(2)}(\cdot,\cdot),$ modulo $O_{C^{\infty}}(h^{\infty})$-error, it follows from (\ref{inteqn2})  and Lemma \ref{MCL} that in terms of parametrizing coordinates with $q=q(z) \in \Gamma_k,$  $q'=q(y') \in \partial \Omega, \rho(z,y') = \frac {q(z)-q(y')}{|q(z)-q(y')|},$ and with

\begin{equation} \label{localize}
\Theta:= \big\{ (z,y'); \,\, \big\langle d_z q(z), \rho(z,y') \big\rangle = \cos (\pi - \alpha_k) + O(\epsilon_0), \,\, |q(z)-c_j|  = O(\epsilon_0) |q(y')-c_k| \,\, \big\},\end{equation}

We claim that for $\epsilon_0 >0$ sufficiently small,

\begin{equation} \label{reflect away}
\inf_{ \{ (q(z),q(y')) \in \Gamma_k \times \partial \Omega; \, (z,y') \in  \Theta \} } \, |q(z) - q(y')| \geq C(\epsilon_0) >0.\end{equation} \

To prove (\ref{reflect away}), we note that when we fix $q(z) = c_k,$ a corner point adjacent to the flat side $\Gamma_j,$ and $q(y') \in \partial \Omega$ is the boundary intersection of a formally reflected tangential ray along $\Gamma_j,$ 
the estimate in (\ref{reflect away}) follows by convexity of $\Omega$ and the admissiblity assumption (see also Figure 4). For $\epsilon_0 >0$ small, (\ref{reflect away}) then follows  for general $(z,y') \in \Theta$ by continuity of the billiard map since we reflect near-glancing rays along $\Gamma_j$ in the adjacent side $\Gamma_k$ near the corner.

\noindent Thus, with $q(y) \in \Gamma_k, \,\, q(z) \in \Gamma_k$ and $q(y') \in \mathring{\partial \Omega},$ we have

\begin{align} \label{q1.1}
  Q_1^{(2)}& (y,y',h) \\
  & = (2\pi h)^{-1/2-1} \int_{\R} \int_{\partial \Omega} e^{i [ (y-z) \xi'  + |q(z) - q(y')|]/h} \, \zeta_k(y,\xi') \, \chi_{\Theta}(z,y') \,  c(z,y';h)  \notag \\
& \times  \, \, (1-\psi_k^{2\delta})(z,h) \,
\  (1-\tilde{\psi}^{2\delta})(y',h) \, \, \big(1-\chi_M( h^{-2\delta}
( |q(z)-q(y'| ) \big) \, dz \, d\xi' + O(h^{\infty}).\notag \end{align}

In (\ref{q1.1}), the cutoff $\chi_{\Theta} \in C^{\infty}(\Gamma_k \times \partial \Omega)$ with $0 \leq \chi_{\Theta} \leq 1$ such that $\chi_{\Theta}(z,y') = 1$ in a $C \epsilon_0$-width tubular neighbourhood of the manifold $\Theta$ in (\ref{localize}) with $C>0$ sufficiently large and $\chi_{\Theta}(z,y') = 0$ outside a $2 C \epsilon_0$-width tubular neighbourhood. 

From now on, we choose the arclength parametrization of the boundary so that $|d_y q(y)| =1$ for all $ y \in \mathring{\partial \Omega}.$
Here, we recall from  (\ref{symbolbounds}) that the symbol $c$ in (\ref{q1.1}) satisfies the estimates  $|q(z)-q(y')|^{1/2} \partial_{z,y'}^{\alpha} c (z,y',h) = O_{\alpha} (|q(z)-q(y')|^{-|\alpha|})$ and note that for the phase function
\begin{equation}
  \label{E:stationary}
 \phi(z;y,y',\xi):=  (y-z) \xi'  + |q(z) - q(y')|, \quad d_{z} \phi = \Big\langle d_z q(z), \frac{q(z)-q(y')}{|q(z)-q(y')|} \Big\rangle - \xi'.\end{equation}
Consequently, $d_z \phi =0$ if and only if $\xi' = \langle d_z q(z), \frac{q(z)-q(y')}{|q(z)-q(y')|} \rangle =  \langle d_z q(z), \rho(y',z) \rangle .$

It follows from (\ref{reflect away}) that with $\epsilon_0 >0$ sufficiently small,

\begin{equation} \label{SIMP1}
 \min_{(y',z) \in \text{supp} \, \chi_{\Theta}}   |q(y') - q(z)| \geq C_1 >0. \end{equation} \

We  also note that for $\epsilon_0 >0$ sufficiently small and with ${\mathcal C}$ denoting the corner set, under the admissibility assumption in Definitiion \ref{admissible} and for $\epsilon_0 >0$ sufficiently small, we claim that there exist a constant $C >0$ (uniform in $\epsilon_0$) such that for the cutoff $\chi_{\Theta}$ in (\ref{q1.1}), one also has

\begin{equation} \label{SIMP2}
\min_{ (z,y') \in \text{supp} \, \chi_\Theta }  \text{dist} \, ( q(y'), {\mathcal C} ) \geq C_2 >0.\end{equation}\

To prove (\ref{SIMP2}), we note that since the billiard map $\beta: B^*\partial \Omega \to B^*\partial \Omega$ is piecewise $C^{\infty},$ in view of the admissibility assumption,  it follows that  $\pi ( \beta (z,\xi') ) \subset \mathring{\partial \Omega}$ provided $|q(z) -c_j| = O(\epsilon_0)$ and $|\xi' - \cos (\pi-\alpha_k)| = O(\epsilon_0)$ with $\epsilon_0 >0$ sufficiently small. Thus, (\ref{SIMP2}) follows from the definition of $\Theta$ in (\ref{localize})  again, by choosing $\epsilon_0>0$ sufficiently small, since
$\text{dist} ( \, \text{supp}\, \chi_{\Theta},  \, \Theta \, ) \leq C \epsilon_0.$
Since $d_z \phi = \xi' -  \langle d_z q(z), \rho(y',z) \rangle,$ in view of (\ref{SIMP1}) and (\ref{SIMP2}) it then follows by repeated integrations by parts in $z$ that,

\begin{align} \label{q1.11}
  Q_1^{(2)}& (y,y',h) \\
  & = (2\pi h)^{-1/2-1} \int_{\R} \int_{\partial \Omega} e^{i [ (y-z)
      \xi'  + |q(z) - q(y')|]/h} \, \zeta_k(y,\xi') \,
  \chi_{\Theta}(z,y') \,  c(z,y';h) \notag \\
& \times  \, \, (1-\psi_k^{2\delta})(z,h) \, \  (1-\tilde{\psi})(y',h)
  \, \, \big(1-\chi_M((|q(z)-q(y')|) \big) \, dz \, d\xi' +
  O(h^{\infty}).  \notag \end{align} 

The point here is that  in view of (\ref{SIMP1}) and (\ref{SIMP2}),  $q(y') \in \text{int} \, \partial \Omega$ with dist \, $(q(y'), {\mathcal C}) \gtrapprox 1$ and also $|q(y') - q(z)| \gtrapprox 1.$  Thus, the small-scale cutoff $$(1-\tilde{\psi}^{2\delta})(y',h) \, \, \big(1-\chi_M( h^{-2\delta} ( |q(z)-q(y')|)  \, )  \big)$$ gets replaced with $$(1-\tilde{\psi})(y',h) \, \, \big(1-\chi_M( |q(z)-q(y')|) ) \big)$$ which is clearly in $S^{0}(1)$ (in fact, it is independent of $h$.).

It will be useful in the following to separate-out the standard $S^0(1)$-part of the amplitude in  (\ref{q1.11}) and define

\begin{equation} \label{S0}
c_{reg}(z,y',h):=  \chi_{\Theta}(z,y') \,  c(z,y';h)  \,  (1-\tilde{\psi})(y',h) \, \big(1-\chi_M((|q(z)-q(y')|) \big) \end{equation} 
where clearly $c_{reg} \in S^0(1).$ Thus, we simply rewrite (\ref{q1.11}) in the form

\begin{eqnarray} \label{Q1.1}
Q_1^{(2)}(y,y',h) = (2\pi h)^{-1/2-1} \int_{\R} \int_{\partial \Omega} e^{i [ (y-z) \xi'  + |q(z) - q(y')|]/h} \, \zeta_k(y,\xi') \,   c_{reg}(z,y';h)  \nonumber \\
\times  \, \, (1-\psi_k^{2\delta})(z,h) \,  dz \, d\xi' + O(h^{\infty}).\end{eqnarray} \

 In (\ref{Q1.1}) we  note that $ (hD_{z})^{\beta} ( 1-
 \psi^{2\delta})(z,h) = O(h^{|\beta| \epsilon'})$  since $2\delta =
 1-\epsilon'$ and also $(h D_{z})^{\beta} c(z,y',h) = |q(z) - q(y')|^{-1/2} \, O( h^{|\beta|} |q(z)-q(y')|^{-|\beta|} ) = O(h^{-1+\epsilon'}) O(h^{|\beta| \epsilon' })$ since $|q(z)-q(y')| \gtrapprox h^{1-\epsilon'}$ for $(z,y')$ in the support of the amplitude in the integral (\ref{Q1.1}).  Thus, by Leibniz rule,
 
 $$ (hD_z)^\beta \big( \, (1-\psi^{2\delta})(z,h) \, c_{reg}(z,y',h) \, \big) = O_{\beta} (h^{-1+\epsilon' + |\beta| \epsilon'}).$$\

 We also note that by convexity, $\cos \alpha_k <1,$ and so, for $\epsilon_0 >0$ sufficiently small (but independent of $h$), the   transversality conditions 

$$ \max ( \, \langle \rho(z,y'), d_{y'} q(y') \rangle,  \langle \rho(z,y'), d_{z} q(z) \rangle\leq \frac{1}{C_3(\epsilon_0)} < 1$$
also follows from (\ref{localize}) and convexity of $\Omega,$ where we recall that  $\rho(z,y')= \frac{ q(z) - q(y')}{|q(z)-q(y')|}.$

   To summarize,  it follows that for $\epsilon_0 >0$ sufficiently small, there exist constants $C_j >0; j=1,2,3,4$ uniform in $\epsilon_0$ such that the cutoff $\chi_{\Theta}$ in (\ref{q1.1}) satisfies
 
 \begin{eqnarray} \label{SUPPORT}
\text{supp}  \,\, \chi_{\Theta} 
 \subset \{ (z,y'); \,\,\max \big( \, \langle \rho(z,y'), d_{y'} q(y') \rangle, \langle \rho(z,y'), d_{z} q(z) \rangle \, \big) \leq \frac{1}{C_1} < 1,  \nonumber \\
|q(z)-c_j| \leq C_2 \epsilon_0, \,\,  |q(z)-q(y')| \geq C_3 > 0, \,\, \text{dist}(q(y'), {\mathcal C}) \geq C_4 >0 \}, \end{eqnarray} 
 where in (\ref{SUPPORT}) we recall that $(q(z), q(y') \in \Gamma_k \times \mathring{\partial \Omega}. $  \\

The next step is to apply stationary phase in  (\ref{Q1.1}) in $(z,\xi')$ taking into account the support properties of $\chi_{\Theta}$ (and consequently $c_{reg}$) in (\ref{SUPPORT}).  Given the phase function
 $$\phi(z,\xi';y,y'):=  (y-z) \xi' + |q(z)-q(y')|,$$
   the critical point equations are

$$ d_{z} \phi = -\xi' + \langle \rho(z,y'), d_z q(z) \rangle = 0 \iff  \xi' = \langle \rho(z,y'), d_z q(z) \rangle,$$
$$ d_{\xi'} \phi = y-z = 0 \iff z = y.$$\

 The only slight subtlety here is the presence of the corner cutoff $ 1- \psi_k^{2 \delta} \in S^{0}_{2\delta}(1)$ which is supported outside an $h^{2\delta}$-neighbourhood of the corner $c_k.$ This cutoff is $2$-microlocal since $2 \delta = 1-0 > 1/2.$  However, this term only depends on the $z$-variables and so, in particular,
 
\begin{equation} \label{singsymbol}
(h D_{z} D_{\xi'})^{\alpha}   \Big(  (1-\psi^{2\delta})(z;h)   \, c_{reg}(z,y',h)\Big) = O(h^{(1-2\delta) |\alpha|}). \end{equation}\

Since $2 \delta <1,$ one can legitimately apply stationary phase in (\ref{Q1.1}); indeed,   setting $\tilde{c}(z,\xi';y,y',h): = \chi(y,\xi') \, \chi_{\Theta}(z,y') \,  c_{reg}(z,y';h) \, (1-\psi_k^{2\delta})(z,h) \, \psi^{2\delta}(y',h),$  the remainder term of order $N$ is
$$ R_N(y,y',h) \leq h^N \int_{0}^1 \frac{(1-t)^N}{N\!} \| \widehat{ (D_z D_{\xi'})^N \tilde{c} } \|_{L^1} \, dt \leq  C_N h^N  \| \widehat{ (D_z D_{\xi'})^N \tilde{c} } \|_{L^1} $$
Thus, since $\tilde{c} \in C^{\infty}_{0},$ it follows that
\begin{equation} \label{rembound}
|R_N(y,y',h)| \leq C_N h^{N} \max_{|\alpha| \leq 2} \| D_{z,\xi'}^{\alpha} \, (D_z D_{\xi'})^N \tilde{c} \|_{L^\infty} = O_N(h^{-2\delta N - 4\delta}) \end{equation}
The last estimate in (\ref{rembound}) follows from Leibniz rule,
(\ref{singsymbol}) and the fact that at most $(N+2)$  derivatives in
$z$ hit the singular symbol $(1-\psi_{2\delta})(z,h).$ Thus, if
follows from (\ref{rembound}) that by choosing $N \gg 1$ sufficiently
large, one can apply stationary phase to the $O(h^{\infty})$-error in (\ref{Q1.1}). The result is that  the Schwartz kernel of $Q_1^{(2)}(h)$  can be written in the form:

\begin{eqnarray}\label{q1.2}
Q_1^{(2)}(y,y',h) = (2\pi h)^{-1/2} e^{i |q(y)-q(y')|h}  \, d_{sing}(y;h) \, d_{reg}(y, y',h) \, (1-\psi)(y')  \nonumber \\
 + O(h^{\infty})   \end{eqnarray}
where, $(q(y), q(y')) \in \mathring{\Gamma_k} \times \mathring {\partial \Omega}, \,\,\, d_{reg} \in S^0(1).$  The symbol $d_{sing}$ is 2-microlocal with $ \partial_{y,y'}^{\alpha} d_{sing} (y,h) = O(h^{-2 \delta |\alpha|}).$

Moreover, again  in view of (\ref{SUPPORT}), the ray $\rho(y,y') = \frac{q(y)-q(y')}{|q(y)-q(y')|}$ is then transversal to the boundary at both endpoints $q(y) \in \mathring{\Gamma_k}$ and $q(y') \in \mathring{\partial \Omega}.$  Thus, there exists $C_0 >1$ with
\begin{eqnarray} \label{q1.21}
\text{supp} \, d_{reg} \subset \Big\{ (y,y'); \, (q(y),q(y')) \in \mathring{\Gamma_k} \times \mathring{\partial \Omega}, \hspace{1in} \nonumber \\
\max \big( \,  | \langle d_{y'}q(y'), \rho(y,y') \rangle |, \,  | \langle d_{y}q(y), \rho(y,y') \rangle | \, \big) \leq \frac{1}{C_0}, \, |q(y)-q(y')| \geq C_1  \Big\}. \end{eqnarray}\

Setting $S(y,y'):= |q(y)-q(y')|,$  it follows by  direct computation that

\begin{eqnarray} \label{mixedhessian}
 \partial_{y} \partial_{y'} S(y,y') = \frac{1}{|q(y)-q(y')|}  \, \big[ \langle d_{y}q(y),  \,\, d_{y'}q(y') - \langle d_{y'}q(y'), \rho(y,y') \rangle \, \rho(y,y') \rangle \big] \nonumber \\
= \frac{1}{|q(y)-q(y')|}  \, \langle d_{y}q(y), \rho^{\perp}(y,y') \rangle   \cdot \langle d_{y'}q(y'), \rho^{\perp}(y,y') \rangle, \end{eqnarray}
where $\rho^{\perp}$ is unit vector orthogonal to $\rho.$
Thus, from  (\ref{q1.21}) and using the fact that $|\rho(y,y')|=1$ and $\Omega$ is convex, 

\begin{equation} \label{q1.3}
| \partial_{y} \partial_{y'} S(y,y')| \geq C_3 >0, \quad (y,y') \in \text{supp} \,  d_{reg}. \end{equation}\

 Thus, $Q_1^{(2)}(h)$ has a canonical relation that is a graph and so, by $h$-Egorov,  $P_1(h):= Q_{1}^{(2)}(h)^* Q_1^{(2)}(h): C^{\infty}(\mathring{\partial \Omega}) \to C^{\infty}(\mathring{\partial \Omega})$ is an $h$-psdo. Indeed, 
 from (\ref{q1.2}) and (\ref{q1.3}),  by a standard Kuranishi change
 of variables, it follows that the Schwartz kernel of the $h$-psdo $P_1(h):= Q_1^{(2)}(h)^* Q_1^{(2)}(h)$  is  modulo $O(h^\infty)$ of the form

\begin{eqnarray} \label{p1.1}
P_1(y',y'',h) = (2\pi h)^{-1} \int_{\R} e^{i(y'-y'') \eta/h} \, p_1(y',\eta; h) \,\, (1-\psi(y'))^2 \, \, f(z(y',\eta),h) \, d\eta, \end{eqnarray}\

In (\ref{p1.1}),  $p_1 \in S^{0}(\mathring{\partial \Omega})$ with supp $\, p_1 \subset \{(y',\eta) \in B^* \mathring{\partial \Omega}; |\eta|_{y'} \leq 1/C_1 <1, \,\, |q(y')- {\mathcal C}| \geq C_4 >0 \}$ due the transversal support properties of $d$ in (\ref{q1.2}). Also, we note that again by taking $\epsilon_0 >0$ small, the near glancing rays in the support of $\chi_j$ reflect off $\Gamma_k$ and intersect the interior of a {\em single} edge, say $\Gamma_{\ell}.$ Thus, in (\ref{p1.1}) we can assume that $ (q(y), q(y')) \in \mathring{\Gamma_{\ell}}$ for some fixed $\ell \in \{1,...,N \}.$

Since
\begin{equation} \label{matrixelements}
\| Q_1^{(2)}(h) u_h \|_{\partial \Omega}^{2}  = \langle P_1(h) u_h, u_{h} \rangle_{\Gamma_{\ell}} + O(h^{\infty}), \end{equation}
one is reduced to estimating the $h$-psdo matrix elements on the RHS of (\ref{matrixelements}).

The second part of the symbol, $ f(z(y',\eta),h),$ in (\ref{p1.1}) is somewhat more subtle since it is in a suitable semiclassical 2-microlocal class. Here, $z(y,\eta') \in \Gamma_k$ where
$$z(y,\eta') = \pi \beta(y,\eta'), \quad (y,\eta') \in B^*_0(\mathring{\Gamma_{\ell}}).$$ 
To describe this symbol in more detail, consider the curve $$H_{\ell} := \{ (y',\eta) \in B^*\Gamma_{\ell};   z(y',\eta) = 0, \,\, (y',\eta) \in \, \text{supp} \, p_1  \},$$ consisting of covectors in $B^*\Gamma_{\ell}$ which result from near-glancing rays to $\Gamma_j$ reflecting near the corner $c_k$ in the adjacent edge $\Gamma_k$ and  then hitting the interior of $\Gamma_{\ell}.$ The fact that $H_{\ell}$ is $C^{\infty}$ follows by the implicit function theorem since $\partial_{\eta} z(y',\eta) \neq 0$ for $(y',\eta) \in H_{\ell}$ from (\ref{q1.3}).

 Let $\Psi_{H_{\ell},2\delta}^{0}$ denotes the space of zeroth-order 2-microlocal h-psdos associated with the hypersurface (i.e. curve) $H_{\ell}  \subset B^*\Gamma_{\ell}'$ as in (\cite{CHT} section 2).
In the formula in (\ref{p1.1}), one readily verifies  that  $f \in S^{0}_{H_{\ell},2\delta} (\Gamma_{\ell}),$ so that

$$P_1(h) \in \Psi_{H_{\ell},2\delta}^0 (\Gamma_{\ell}).$$

Moreover, it is readily checked that
$$ f(z(y',\eta),h) = (1-\psi_{k}^{2\delta})^2 (z(y',\eta),h) + O_{S}(h^{1-2\delta}).$$

To summarize, setting 

\begin{equation} \label{product}
 f_{sing}(y',\eta,h):= p_1(y',\eta;h) \,  (1-\psi(y'))^2 \, \,f( z(y',\eta),h) \in S^{0}_{H_{\ell},2\delta}(\Gamma_{\ell}), \end{equation} 
we have shown  that

\begin{equation} \label{p1 2micro}
P_1(h) = Op_{h}(f_{sing}) + O(h^{\infty})_{L^2 \to L^2}. \end{equation}\

 Moreover, recall that the symbol $p_1(y',\eta)$ has transversal   support {\em away} from corners with  supp $\, p_1 \subset \{(y',\eta) \in B^* \mathring{\partial \Omega}; |\eta|_{y'} \leq 1/C_3(\epsilon_0) <1, \,\, |q(y')- {\mathcal C}| \geq C \epsilon_0 \}.$ So, from (\ref{product}),

$$ \text{supp}  \, f_{sing} \subset \{(y',\eta) \in B^* \mathring{\partial \Omega}; |\eta|_{y'} \leq 1/C_3(\epsilon_0) <1, \,\, |q(y')- {\mathcal C}| \geq C \epsilon_0 \}.$$\

 Then, we introduce an additional cutoff $\chi_{reg} \in C^{\infty}_{0}(B^*\Gamma_{\ell}), \,\, 0 \leq \chi_{reg} \leq 1,$ satisfying

 \begin{description}
   \item[(i)]
 \begin{align} \label{regcutoff}
  \text{supp}  \, \chi_{reg} \subset \{(y',\eta) \in B^* \mathring{\partial \Omega}; |\eta|_{y'} \leq 1/\tilde{C_3}(\epsilon_0) <1, \,\, |q(y')- {\mathcal C}| \geq  2 C \epsilon_0 \}; \, \tilde{C_{3}} < C_{3},\end{align}
and
   \item[(ii)] \begin{eqnarray}  \chi_{reg} (y,\eta') = 1, \, \, \, \, (y,\eta') \in \text{supp} \,f_{sing}.\end{eqnarray} 
 \end{description}
 
 Thus, from (ii), $\chi_{reg}  \, f_{sing} = f_{sing} $ with $\chi_{reg} \in S^{0}(1)$ in a standard symbol class satisfying the transversal support conditions in (\ref{regcutoff}).   Then, by $h$-psdo calculus (see \cite{CHT} subsection 2.2.2), since $P_1(h) = Op_h(f_{sing}),$ we have
 
$$ \langle P_1(h) u_h, u_h \rangle_{\Gamma_{\ell}} = \langle P_{1}(h) \chi_{reg}(h)  \tilde{u_h}, \chi_{reg}(h) \tilde{u_h}\rangle_{\Gamma_{\ell}'} + O(h^{\infty})$$
$$ = O(1) \| \chi_{reg}(h) u_h \|_{L^2}^2,$$\\
by $L^2$-boundedness of $P_1(h) = Op_h(b_{sing}).$ Finally, since the symbol $\chi_{reg}$ is supported transversally to the boundary edge $\Gamma_{\ell}$ and outside an $h$-independent neighbourhood of the  corners, by the Rellich result in  Lemma \ref{rellich},
 $$ \| \chi_{reg}(h) u_h \|_{\partial \Omega} = O(1).$$

Consequently,
 
 \begin{equation} \label{matrixbound}
  \langle P_1(h) u_h, u_h \rangle_{L^2(\Gamma_{\ell})} = O(1),
  \end{equation}   
 and  in view of (\ref{matrixelements}),  

 \begin{equation} \label{q2.1upshot}
  \| Q_1^{(2)}(h) u_h \|_{\partial \Omega} = O(1). \end{equation}\

\noindent  {\bf $Q_1^{(1)}$-term.} As for the near-diagonal
$Q_1^{(1)}(h)$-term, since $|q(y)-q(y') |\lessapprox h^{2\delta}$ for
$(y,y') \in \text{supp } Q_{1}^{(1)}(\cdot,\cdot),$ one cannot simply
use the asymptotic formula for the kernel in (\ref{inteqn2}).
Instead, we use the exact Hankel function formula (\ref{double}) together with a Schur lemma argument to control this term.   By $L^2$-boundedness, 

\begin{equation} \label{reduce}
 \| Q_1^{(1)}(h) u_h \|_{L^2} = \| \zeta_k(hD) N_1^{(1)}(h) u_h \|_{L^2} = O(1) \|N_1^{(1)}(h) u_h\|_{L^2} \end{equation}
and so, it suffices to bound $\|N_1^{(1)}(h) u_h\|_{L^2}$ from above.

From the explicit formulas (\ref{double}) and (\ref{hankel}) it follows that for $(q,q') \in \mathring{\Gamma_k} \times \mathring{\Gamma_k},$ 

\begin{equation} \label{lucky}
h^{-1} |\langle \nu_q, \rho(q,q')\rangle | = O(|q-q'| h^{-1}). \end{equation}\

Note that in (\ref{lucky}) {\em both} $q$ and $q'$ are constrained to the {\em same} boundary edge $\Gamma_k$, so there is no jump in the boundary normal resulting in the improved estimate in (\ref{lucky}). From the explicit formulas (\ref{hankel}), (\ref{double}) and  (\ref{lucky}), by  setting $z = |q-q'|/h,$ we have (see (\ref{n1}) for definition of $N_1^{(1)}$),
$$|N_1^{(1)}(q,q';h)| \leq C z^{1/2}\, \Big| \, \int_{0}^{\infty} e^{-s} s^{1/2} ( 1- \frac{s}{2iz} )^{1/2} ds \, \Big| \, \cdot \, \chi_M(h^{-2\delta}(q-q')) $$
$$ \leq C  \, \Big(  \int_{0}^{\infty} e^{-s} s^{1/2} ( s + 2 z )^{1/2} \, ds \Big) \,  \chi_M(h^{-2\delta}(q-q')) $$
\begin{equation} \label{schur}
 \leq C' \big( 1 + h^{-1/2} |q-q'|^{1/2} \big) \chi_M(h^{-2\delta}(q-q')) \lessapprox h^{-0} \chi_M(h^{-2\delta}(q-q')).\end{equation}\

Here, the last line follows since $2\delta = 1-0.$ From (\ref{schur}) we get that by the Schur lemma,

\begin{align} 
\| N_{1}^{(1)}(h) \|_{L^2 \to L^2} & \leq \max  \Big( \int_{|q-q'| \lessapprox h^{1-0}} |N_1^{(1)}(q,q',h)| dq,  \int_{|q-q'| \lessapprox h^{1-0}} |N_1^{(1)}(q,q',h)| dq' \Big)\notag \\
 = O(h^{1-0}). \label{q2.2}\end{align} 

Since by Sobolev restriction,  $\| u_h \|_{\partial \Omega} = O(h^{-1/2}),$ it follows from (\ref{q2.2})  and (\ref{reduce}) that

\begin{equation} \label{q2.3}
\| Q_1^{(1)}(h) u_h \|_{\partial \Omega} = O(h^{1/2-0}).\end{equation}\

Thus, from (\ref{q2.1upshot}) and  (\ref{q2.3}), 

\begin{equation} \label{q2upshot}
\|Q_1(h) u_h \|_{L^2(\partial \Omega)} = O(1).\end{equation}\

\subsubsection{ Estimating  $ \| Q_2(h) u_h \|$ } \label{fardiagonal} 
We begin by recalling from (\ref{no kernel}) that the Schwartz kernel  
$N_2^{(1)}(q,q',h) = 0$ and so, there is  no  $Q_2^{(1)}(h)$ term.

\noindent{\bf $Q_2^{(2)}$-term:}    We claim this term is $O(h^{\infty})$ and is therefore residual.
The argument is a rather standard wave front computation using the fact that under the admissibility assumption and  for $\epsilon_0>0$ small, all reflected rays leaving $\Gamma_k$ hit a boundary edge $\Gamma_{\ell}$ in the {\em interior} and far from corners.

Let $\tilde{\psi}_{\epsilon_0} = \sum_{k} \tilde{\psi}_{k,\epsilon_0},$ where the $\tilde{\psi}_{k,\epsilon_0}$ is a corner cutoff (independent of $h$)  supported in an $\epsilon_0$-neighbourhood of $c_k.$
Since $|q-q'| \gtrapprox h^{2\delta},$ one can replace $N_2^{(2)}(h)$ with the $h$-FIO  piece, $N_{\beta}(h).$  The result is that

\begin{eqnarray}\label{wf}
 Q_2^{(2)}(h) =  \zeta_k(hD) \, (1-\psi_k^{2\delta}(h)) {\bf 1}_{\Gamma_k} N_{\beta}(h) \tilde{\psi}^{2\delta}(h) + O_{C^\infty}(h^{\infty}) \nonumber \\
 =  \zeta_k(hD) \, (1-\psi_k^{2\delta}(h)) {\bf 1}_{\Gamma_k} N_{\beta}(h) \tilde{\psi}^{\epsilon_0} \tilde{\psi}^{2\delta}(h)  + O_{C^\infty}(h^{\infty}), \end{eqnarray}\
 
\noindent where in the last line we have used that since \ $\tilde{\psi}_{\epsilon_0} \Supset \tilde{\psi}^{2\delta}(h),$ clearly  $(1-\tilde{\psi}_{\epsilon_0}) \tilde{\psi}^{2\delta}(h) \equiv 0.$ 

By h-psdo calculus, it then follows that with $\tilde{\zeta}_k \Supset \zeta_k,$

$$ \zeta_k(hD) (1-\psi_k^{2\delta})(h) )  {\bf 1}_{\Gamma_k} N_{\beta}(h) \tilde{\psi}_{\epsilon_0} \tilde{\psi}^{2\delta}(h) \hspace{1.3in}$$
$$= \zeta_k(hD) (1-\psi_k^{2\delta})(h) ) \tilde{\zeta}_k(hD) {\bf 1}_{\Gamma_k} N_{\beta}(h) \tilde{\psi}_{\epsilon_0} \tilde{\psi}^{2\delta}(h) + O(h^{\infty}).$$\

Then, by $L^2$-boundedness,
\begin{eqnarray} \label{wf2}
\| \zeta_k(hD) (1-\psi_k^{2\delta})(h) ) \tilde{\zeta}_k(hD) {\bf 1}_{\Gamma_k} N_{\beta}(h) \tilde{\psi}^{\epsilon_0} \tilde{\psi}^{2\delta}(h) u_h \|_{\Gamma_k} \nonumber \\
 = O(1)  \, \| \tilde{\zeta}_k(hD) {\bf 1}_{\Gamma_k} N_{\beta}(h) \tilde{\psi}_{\epsilon_0} \|_{L^2 \to L^2} \,\, \| \tilde{\psi}^{2\delta}(h) u_h \|_{\Gamma_k}. \end{eqnarray}\

Then, by $h$-wavefront calculus, 
\begin{eqnarray} \label{wf3}
WF_h' \Big(  \tilde{\zeta}_k(hD) {\bf 1}_{\Gamma_k} N_{\beta}(h) \tilde{\psi}_{\epsilon_0} \Big) \subset \Big\{ (y,'\xi';y,\xi) \in B^*\mathring{\partial \Omega} \times B^*\Gamma_k, \nonumber \\
\, |q(y)-c_k| \leq \epsilon_0, \, \sum_{\ell} |q(y') - c_{\ell}| \leq \epsilon_0,
 \, (y',\xi') = \beta (y,\xi), \, |\xi - \cos(\pi-\alpha_k)| \leq C' \epsilon_0 \Big\}. \hspace{0in} \end{eqnarray}

By continuity of the billiard map $\beta: B^*\partial \Omega \to  B^*\partial \Omega,$ under the admissibility assumption in Defintion \ref{admissible}, it follows that for $\epsilon_0>0$ sufficiently small in (\ref{wf3}), when $|q(y)-c_k| < \epsilon_0,$ one has that $\min_{\ell =1,..,M} |q(y')-c_{\ell}| > C_0>0$ for {\em all} corner indices $\ell$ in (\ref{wf3})  and where $C_0>0$ can be chosen {\em independent} of $\epsilon_0 >0.$ Consequently, 
\begin{equation} \label{wf4}
WF_h' \Big(  \tilde{\zeta}_k(y,hD) {\bf 1}_{\Gamma_k} N_{\beta}(h) \tilde{\psi}_{\epsilon_0} \Big) = \emptyset. \end{equation}

Thus, from (\ref{wf})-(\ref{wf4}) it follows that
\begin{equation} \label{q1upshot}
\| Q_2(h) u_h \|_{\partial \Omega} = O(h^{\infty}).\end{equation}

To summarize, in view of (\ref{q1upshot}), (\ref{q2upshot}) and (\ref{CRUCIAL}),  we have proved that\\

\begin{eqnarray} \label{CRUCIAL2}
\| \chi_j(h D)  (1 - \psi_{j}^{\delta}(h)) u_h  \|_{L^2(\Gamma_j)}  \lessapprox \sum_{k=j-1}^{j+1} \| N_{jk}^{\mathcal G}(h) \|_{L^2 \to L^2} \, \Big( 1 + O( \| \psi_k^{2\delta}(h) u_k \| ) \Big)\nonumber \\
+ \| N_j^{\mathcal D}(h) u_h \|_{L^2(\Gamma_j)} + O(h^\infty).\hspace{1in}
 \end{eqnarray}\

From the non-concentration result in Theorem \ref{T:non-con}, for any corner $c_k \in {\mathcal C},$

$$ \| \phi_h \|_{B(c_k, h^{2\delta})}  = O(h^{\delta})$$ 
and by an application of Sobolev restriction, it follows that  by setting $\delta = 1/2-0,$

$$  \| \psi_k^{2\delta}(h) u_k \|_{L^2(\Gamma_k)} = O(h^{\delta - 1/2}) = O(h^{-0}).$$

Then, from (\ref{CRUCIAL2}), 

\begin{eqnarray} \label{CRUCIAL3}
\| \chi_j(h D)  (1-\psi_{j}^{\delta}(h)) u_h  \|_{L^2(\Gamma_j)}    \hspace{2in}\nonumber \\ \lessapprox
h^{-0} \sum_{k=j-1}^{j+1} \| N_{jk}^{\mathcal G}(h) \|_{L^2 \to L^2} +   \| N_j^{\mathcal D}(h) u_h \|_{L^2(\Gamma_j)} + O(h^\infty).\end{eqnarray}\

Recall, that given  the transfer operator $N_{jk}(h)$ in (\ref{TRANSFERcut}), one can write  (see (\ref{TRANSFERrelate})),

$$N_{jk}^{\mathcal G}(h) = N_{jk}(h) \chi_{k}^{tr}(q',hD)  (1-\psi_{k}^{2\delta})(h), \,\,\,N_{jk}^{\mathcal D}(h) = N_{jk}(h)  \psi_k^{2\delta}(h).$$\

Since $N_{j}^{\mathcal D}(h)  = \sum_{k = j-1,j+1} N_{jk}^{\mathcal
  D}(h),$ in view of (\ref{CRUCIAL3}), one is reduced to bounding $\| N_{jk}(h) \|_{L^2 \to L^2}$ for the transfer operator $N_{jk}(h): C^{\infty}(\mathring{\Gamma_j}) \to C^{\infty}(\mathring{\Gamma_k})$ in (\ref{TRANSFER}).

\subsubsection{Estimating $\| N_{jk}(h) \|_{L^2 \to L^2}$ }\label{standard}

Before deriving upper bounds for  $\| N_{jk}(h) \|_{L^2 \to L^2},$ we review some background on $h$-Fourier integral operators with fold-type canonical relations.

\subsection{h-Fourier integral operators associated with one-sided folds} \label{transfer}

 We consider here the transfer operators $N_{jk}(h) : C^{0}(\Gamma_j) \to C^{0}(\Gamma_k); \,\, k=j-1,j+1$ with Schwartz kernel given in (\ref{TRANSFER}).

We briefly pause here to motivate the $O(h^{-1/4-0})$-bound in Theorem \ref{dirichlet} by making explicit the connection to the standard $L^2 \to L^2$ bounds for $h$-Fourier integral operators with canonical relations that are one-sided folds. The novelty here lies  in the extension of the estimates  up to corners. 

To begin, set  $S(s,t) := |q(s) - q(t)|$  and consider the singular set

\begin{equation} \label{sing}
 \Sigma:= \{ (q(s), q(t)) \in \Gamma_j \times \Gamma_k; \,\, s \in \text{supp} \, (1 - \psi_j^{\delta});  \,\,\, \partial_s \partial_t S(s,t) = 0 \}. \end{equation}
The set $\Sigma$ is the singular locus of the Lagrangian parametrization

\begin{equation} \label{parametrization}
 \iota(s,t) = (q(s), \partial_s S(s,t), q(t), \partial_t S(s,t)) \in \Lambda_{\beta} \cap \pi^{-1} (\text{supp} \psi_j^{\delta} ), \end{equation}
 in the sense that, by an application of the inverse function theorem,
 for $(s,t) \in \Sigma^c$, $\iota |_{\Sigma^c}$ is a canonical
 graph. Then,  for arbitrarily small (but fixed) $\epsilon >0$  we let
 $\chi_{\Sigma} \in C^{\infty}_0 (\Gamma_j \times \Gamma_k)$ be
 supported in an $2 \epsilon$-width tubular neighbourhood of $\Sigma$
 with $\chi_{\Sigma}  \equiv 1$ in an $\epsilon$-width tubular
 neighbourhood. Let $N_{jk}^{\Sigma}(h)$ (resp.  $N^{1 - \Sigma}_{jk}(h)$) be the operators with Schwartz kernels  $ \chi_{\Sigma} N_{jk}(h)$  (resp. $(1-\chi_{\Sigma}) N_{jk}(h).)$ Then,
by the $h$-Egorov theorem,
 $$  N_{jk}^{1-\Sigma} (h)^*  N_{jk}^{1-\Sigma}(h) \in Op_h( S^0_{\delta}(\mathring{\Gamma_k}) ).$$ 
 
 Thus, by $L^2$-boundedness,
 \begin{equation} \label{fiograph}
 \| N_{jk}^{1-\Sigma}(h) \|_{L^2 \to L^2} = O(1). \end{equation}
 
   A stationary phase argument as in subsection \ref{statphase} shows that  the transfer operator  $N_{jk}(h) : C^{\infty}(\Gamma_k) \to C^{\infty}(\Gamma_j)$ has a Schwartz kernel of the form 
  
 \begin{equation} \label{fiowkb}
 N_{jk}^{\Sigma}(h)(s,t) = (2\pi h)^{-1/2}  e^{i S(s,t)/h}  \, (1-\psi_j^{\delta})(s;h) \,c_{\chi_j}(s,t,h) \, \chi_{\Sigma}(s,t), \,\,\, (q(s),q(t)) \in \Gamma_j \times \Gamma_k. \end{equation}

Moreover, in (\ref{fiowkb}), the symbol $c_{\chi_j}(s,t,h)$ has the following properties:
\begin{description}
  \item[(i)]
\begin{align*}
 |q(s) - q(t) |^{1/2}  \cdot c_{\chi_j}(s,t,h) \in
 S^{0}_{\delta}(1) \end{align*}
and
\item[(ii)]
  \begin{align}  
    \text{supp} & \, c_{\chi_j}\subset \{(s,t); |q(t)-c_j| \leq
    C\epsilon_0 |q(s)-c_j|, \notag 
    \\
    & \langle d_t q(t), \rho(s,t) \rangle = \cos (\pi - \alpha_k) +
    O(\epsilon_0) < 1 \rangle \}, \label{fiowkb2}\end{align}
  \end{description}
where we continue to write $\rho(s,t)=  \frac{q(t)-q(s)}{|q(t)-q(s)|}.$ 
In view of (\ref{fiograph}), we have that
$$ \| N_{jk}(h) \|_{L^2 \to L^2} = \| N_{jk}^{\Sigma}(h) \|_{L^2  \to L^2} + O(1)$$
and one is consequently reduced to estimating  $\|  N_{jk}^{\Sigma}(h) \|_{L^2 \to L^2},$ with Schwartz kernel $N_{jk}^{\Sigma}(h)(s,t)$  in (\ref{fiowkb}).

\begin{lemma} \label{glancing}  Let $c_{j+1} = \Gamma_j \cap
  \Gamma_{j+1}$  be the corner adjacent to the  boundary edges $\Gamma_j$ and $\Gamma_{j+1}.$  Then, by choosing the glancing cutoff aperature $\epsilon_0>0$ sufficiently small, it follows that there exist constants  $C_1(\epsilon_0)>0$  and $C_2(\epsilon_0)>0$ such that for any $h \in (0,h_0(\epsilon_0)]$ sufficiently small and $(s,t)  \in \text{supp} \, \,( c_{\chi_j} \cdot \chi_{\Sigma}),$

$$C_1(\epsilon_0) \frac{ |q(t)-c_{j}| \, |q(s)-c_j|}{ |q(s)-q(t)|^{3}} \leq  | \partial_s \partial_t S(s,t) |  \leq C_2(\epsilon_0) \frac{ |q(t)-c_{j}| \, |q(s)-c_j|}{ |q(s)-q(t)|^{3}}, $$
\end{lemma}\

\begin{proof}  By an affine change of Euclidean coordinates, it suffices to assume that $c_{j} = (1,0)$ and
$$\Gamma_j = \{(s,0); \,\, s \in \text{ supp} \, (1-\psi_j^{\delta}) \} \subset \{ (s,0);   \,\,  0 \leq s \leq 1- C_1 h^{\delta}  \}.$$

 Then,
 
$$\Gamma_{j+1} = \{ (t, f(t));  t \in \text{supp} \, \psi_k \, \} = \{ (t, f(t)); 1 \leq t \leq 1 + \tilde{C_1} \epsilon_0 \},$$\
where with some $\alpha >0,$ the profile function
$$ f(t) = \alpha (t-1) + O( |t-1|^2), \quad 1 \leq t \leq 1 + \tilde{C_1} \epsilon_0.$$

In this case, the phase function is 
$$ S(s,t) = [  (s-t)^2 + f^2(t)   ]^{1/2}.$$

\begin{equation} \label{direct}
\partial_s \partial_t S(s,t) =     \frac{ 2 \alpha^2 (t-1) \, [ 1-s + O(1-t) ] }{ [  (s-t)^2 + \alpha^2 (1-t)^2 ]^{3/2} }. \end{equation}

Finally, by shrinking the glancing cutoff $\chi_j(hD)$ (i.e. taking $\epsilon_0 >0$ sufficiently small), one can assume that 
$$|t-1| \leq C \epsilon_0 |1-s|$$
and so,
 $$ 1-s + O(1-t) \gtrapprox 1-s$$  
 in (\ref{direct}). This completes the proof of the Lemma.

\end{proof}

It will be useful in the following to use the parametrizations of $\Gamma_j$ and $\Gamma_k$ given in the proof of Lemma \ref{glancing}.  For future reference, we note that since $|1-t| \leq C \epsilon_0 |1-s|$ for $(s,t) \in \text{supp} \, c_{\chi_j},$ it follows that by choosing $\epsilon_0 < 1$ small, in terms of the  parametrizing coordinates in Lemma \ref{glancing},
 \begin{eqnarray} \label{usefulbound1}
|q(s)-q(t)|  \approx | 1-s |, \quad (s,t) \in \text{supp}\, c_{\chi_j}, \end{eqnarray}
and so,  by  Lemma \ref{glancing} it follows that

\begin{equation} \label{usefulbound2}
|\partial_s \partial_t S(s,t)| \approx \frac{ | 1- t |}{ | 1- s |^2}, \quad (s,t) \in \text{supp} \,c_{\chi_j}.  \end{equation}\

In view of (\ref{usefulbound2}), since $|1-s| \gtrapprox h^{\delta}$ for $s \in (1-\psi^{\delta}_j),$ in the following, it will be useful to use the defining function
$$F(s,t):= t-1$$
to dyadically decompose $\text{supp} \chi_{\Sigma}$ in order to estimate $\| N_{jk}^{\Sigma}(h) \|_{L^2 \to L^2}.$

\begin{remark} From Lemma \ref{glancing} it is immediate that the singular manifold
$$ \Sigma = \{ (s,1),  s \in \text{supp} \, (1-\psi_j^{\delta} )  \} \cong S^*_+(\Gamma_j \cap \text{supp} \, (1- \psi_j^{\delta}))$$
which is just the (positive) glancing set along $\Gamma_j.$ 
Repeating the computation in Lemma \ref{glancing} with the corner  $c_{j}$  adjacent to sides $\Gamma_j$ and $\Gamma_{j-1},$  it follows  that the singular manifold in the latter case is $ \Sigma' = \{(s,-1); s \in \text{supp} \psi_j^{\delta} \} \cong S_-^* (\Gamma_j \cap \text{supp} \,  (1 - \psi_j^{\delta} ),$  so that the union $\Sigma \cup \Sigma'  \cong S^* ( \Gamma_j \cap \text{supp} \, \psi_j^{\delta} ),$ the entire glancing set along $\Gamma_j \cap \text{supp} \psi_j^{\delta}.$

It also follows by a direct computation using (\ref{direct}) that
$$ \partial_{s}^2 \partial_t S(s,t=1)  = 0, \quad \partial_s \partial_t^2 S(s,t=1) \neq 0, \quad s \in \text{supp} \, (1 - \psi_j^{\delta}).$$

 Thus,  the Lagrangian parametrization $\iota: \text{supp} \, (1-\psi_j^{\delta}) \, \times \, \text{supp} \, \tilde{\psi}_k  \to \Lambda_{\beta}$ in (\ref{parametrization}) has a one-sided fold singularity 
along the glancing set $\Sigma.$ The subtlety here is the presence of small-scale  cutoffs  in $h$ which will  create some additional terms resulting in $\log h$-loss in the usual bounds.

\end{remark}

To estimate $ \| N^{\Sigma}_{jk}(h)^* \cdot N^{\Sigma}_{jk}(h) \|_{L^2 \to L^2} $ we make the usual dyadic decomposition around the singular hypersurface $\Sigma$ (see \cite{Ph}), but since there are  small-scale (in $h$) non-standard symbols involved we will need to keep track of these terms in the estimates.  In view of (\ref{fiowkb}), the Schwartz kernel of $P(h):=N^{\Sigma}_{jk}(h)^* N^{\Sigma}_{jk}(h) : C^{0}(\Gamma_j) \to C^{0}(\Gamma_j)$ is of the form

\begin{eqnarray} \label{t22}
P(t,t',h):= (2\pi h)^{-1} \int_{\R} e^{i [ S(s,t) - S(s,t')]/h} \tilde{c}(s,t,t',h)  \, ds, \end{eqnarray}
where (see (\ref{TRANSFERcut})),

$$ \tilde{c}_{\chi_j}(s,t,t',h):= c_{\chi_j}(s,t,h) c_{\chi_j}(s,t',h) \chi_{\Sigma}(s,t) \chi_{\Sigma}(s,t') \, | (1- \psi_j^{\delta})(s,h) |^2,$$
and the symbols $c_{\chi_j}(s,t,h)$ and $c_{\chi_j}(s,t',h)$ satisfy the bounds in (\ref{fiowkb2})(i) with

\begin{equation} \label{ctilde}
 |q(s)-q(t)|^{1/2} \, |q(s)-q(t')|^{1/2} \tilde{c}_{\chi_j} (s,t,t',h) \in S^{0}_{\delta}(1). \end{equation}\

They also satisfy the support condition in (\ref{fiowkb2})(ii). Moreover, in view of these support conditions, by choosing $\epsilon_0>0$ small, it suffices to assume that $|t-t'| \ll 1$ in (\ref{t22}). 

We apply a standard Kuranishi argument combined with a dyadic
decomposition of the frequency variables in (\ref{t22}).  Let $
\chi^{\pm}_m \in C^{\infty}_0(\R;[0,1]), \, m=0,1,2,...$ be a sequence
of cutoffs with supp $\chi^{\pm}_m \subset [ \pm 2^{-m}, \pm
  2^{-m+1}]$ and $\sum_{m} \chi^{\pm}_m = 1.$  We make the decomposition 

$$ P(t,t',h)= (2\pi h)^{-1} \sum_{m} \int e^{i [ S(s,t) - S(s,t')]/h} \tilde{c}_{\chi_j}(s,t,t',h)  \, \chi_m(t-1) \, \chi_m(t'-1) \,ds + O(h^{\infty})$$

Setting
\begin{equation} \label{dyadicpsdo}
P_m(t,t',h):= (2\pi h)^{-1}  \int e^{i [ S(s,t) - S(s,t')]/h} \tilde{c}_{\chi_j}(s,t,t',h)  \, \chi_m(t-1) \, \chi_m(t'-1) \,ds, \end{equation}
it then follows that $ \| N_{jk}^{\Sigma}(h) \|_{L^2 \to L^2}^2 = \|P \|_{L^2 \to L^2} \leq \sum_{m} \|P_m \|_{L^2 \to L^2}$ and so, one is reduced to bounding the latter.

By Taylor expansion of the phase in (\ref{dyadicpsdo}),

 $$S(s,t) - S(s,t') = \partial_t S(s,t^*(t,t',s))  (t-t'), \quad t^*-1 \in \text{supp} \, \chi_m,$$
 we note that in view of the support of $\tilde{c}$ we have $ \max (|t-1|,|t'-1|) \leq C \epsilon_0 \ll 1$ and so, by the corresponding Kuranishi change of variable $s \mapsto \partial_t \partial_s S(s,t^*) = \xi,$

\begin{eqnarray} \label{t24}
P_m(t,t',h)= (2\pi h)^{-1}  \int_{\R} e^{i (t-t') \xi/h} \, \tilde{c}(s,t,t',h) \, \frac{ \chi_m (t-1) \, \chi_m(t'-1)}{ | \partial_s \partial_t S(s, t^* )|}  \, d\xi \nonumber \\
= (2\pi h)^{-1}  \int_{\R} e^{i (t-t') \xi/h} \, c_m(s,t,t',h)  \, d\xi, \hspace{1.2in} \nonumber \\ \nonumber \\
 c_m(s,t,t',h):= \tilde{c}(s,t,t',h) \cdot \frac{ \chi_m (t-1) \, \chi_m(t'-1)}{  \partial_s \partial_t S(s, t^* )},\hspace{1.5in} 
  \end{eqnarray}
where the last line  in (\ref{t24})  follows from (\ref{usefulbound1}) and (\ref{usefulbound2}).
In the integrand of (\ref{t24})  we abuse notation slightly and write $s = s(\xi,t,t')$ and $t^* = t^*(s(\xi ;t,t'),t,t').$

From (\ref{usefulbound1}), (\ref{usefulbound2}) and Lemma \ref{glancing},
,
$$ |c_m(s,t,t',h)| \lessapprox |q(s)-q(t)|^{-1/2} |q(s)-q(t')|^{-1/2} |q(s)-q(t^*)|^{3}$$
$$ \times \chi_m(t-1) \chi_m(t'-1) |q(t^*)-c_j|^{-1} \, |q(s)-c_j|^{-1},$$\

We note here that $q(t^*)$ lies  between $q(t)$ and $q(t')$ and so, $ t^*-1 \in \text{supp} \, \chi_m$. Then,  using (\ref{usefulbound1}) and (\ref{usefulbound2}),

\begin{equation} \label{symbolbounds1}
|c_m(s,t,t',h)| \lessapprox  |s-1|^{3} \, |s-1|^{-2} \, |t^*-1|^{-1} \, \chi_m(t^*-1) \lessapprox 2^{m} \, |s-1| \lessapprox 2^m. \end{equation}\

Similarily,  provided the dyadic scale $2^{m} \lessapprox h^{-\delta}, \,\, \delta = 1/2 -0,$ for the derivatives one gets that 

\begin{equation} \label{derivativebounds}
|\partial_{\xi}^{\alpha} \partial_{t,t'}^{\beta} c_{m}(s(\xi;t,t'),t,t',h))| \lessapprox C_{\alpha,\beta} 2^{m} h^{-\delta ( |\alpha| + |\beta|)}, \quad \text{when} \,\, 2^{m} \lessapprox h^{-\delta}. \end{equation}\

An application of $L^2$- boundedness  for the $h$-psdo $P_m$ (\cite{Zw} section 4.5.1) then gives

\begin{eqnarray} \label{CV}
\| P_m \|_{L^2 \to L^2} \lessapprox   \sum_{\alpha \leq C(n)} h^{|\alpha|/2}  \sup | \partial ^{\alpha} c_m | \nonumber \\
\lessapprox 2^{m}  \big( 1 + \sum_{1 \leq \gamma + \gamma' \leq C(n)}  h^{\gamma/2}  2^{m \gamma} \, h^{\gamma' (1/2-\delta)}  \big) \nonumber \\
\lessapprox 2^{m}  \big( 1 + \sum_{1 \leq \gamma \leq C(n)}  h^{\gamma/2}  2^{m \gamma} \big) 
\end{eqnarray} 
since $0 \leq \delta < 1/2.$

The first term on the RHS of (\ref{CV}) comes from differentiaion of
the dyadic cutoff term $\chi_m(t-1) \chi_m(t'-1)  (t^*-1)^{-1}$ whereas the second term arises from differentiaion of the symbol $ (t^*-1) \tilde{c}\in S^{0}_{\delta}(1).$  

Thus, from the last line of (\ref{CV}),

\begin{equation} \label{psdobound}
\| N_{jk}^{\Sigma}(h) \|_{L^2 \to L^2} = \| P_m(h) \|_{L^2 \to L^2}^{1/2}  \lessapprox 2^{m/2}; \quad \text{when} \, \, 2^{m} \lessapprox h^{-\delta}. \end{equation}\

On the other hand, one can bound $\| P_{m}(h) \|_{L^2 \to L^2}$ directly using (\ref{dyadicpsdo}) and (\ref{usefulbound1}). From (\ref{fiowkb2}) we have that
$$ | c_{\chi_j}(s,t) | = O( |s-1|^{-1/2}),$$
and so, by the Schur lemma,

\begin{eqnarray} \label{fiobound}
\| P_{m}(h) \|_{L^2 \to L^2} 
\lessapprox h^{-1}   \int \, \Big| \int_{ 0}^{1-h^{\delta}} c_{\chi_j}(s,t,h) c_{\chi_j}(s,t',h) \chi_m(t-1) \chi_m(t'-1) ds \, \Big| \, dt  \nonumber \\
\lessapprox h^{-1} \int_{\R} \, \Big( \, \int_{0}^{1-h^{\delta}} \frac{ds}{1-s} \, \Big) \, \chi_m(t-1) \chi_m(t'-1) dt \lessapprox h^{-1} 2^{-m} |\log h|.
\end{eqnarray}\

Using (\ref{psdobound}), (\ref{fiobound}) and taking square roots, we get that

\begin{equation} \label{transferbound2}
 \| N_{jk}^{\Sigma}(h) \|_{L^2 \to L^2} \leq \sum_{m} \min \, (  2^{m/2},  h^{-1/2} |\log h|^{1/2} 2^{-m/2} ).\end{equation}

We note that $2^{m/2} h^{1/2} |\log h|^{-1/2} = 2^{-m/2}$  is equivalent to $2^{m} = h^{-1/2} |\log h|^{1/2}$ so that $2^{m/2} =h^{-1/4} |\log h|^{1/4}.$ Thus, from (\ref{transferbound2}) it follows that

\begin{align}
\| N_{jk}(h) \|_{L^2 \to L^2} & =  \| N_{jk}^{\Sigma}(h) \|_{L^2 \to
  L^2} + O(1) \notag \\
& = O(h^{-1/4} |\log h|^{1/4})+ O(1)  = O(h^{-1/4} |\log h|^{1/4}). \label{TRANSFERBOUND} \end{align}
 
 To obtain (\ref{TRANSFERBOUND})  for the  dyadic scales $2^{m/2} \leq h^{-1/4} |\log h|^{1/4}$ we use the  $h$-psdo bound (\ref{psdobound}) in (\ref{transferbound2}), whereas for $2^{m/2} > h^{-1/4} |\log h|^{1/4},$ the volume bound in (\ref{fiobound}) is optimal.   Thus,  from  (\ref{TRANSFERBOUND}), it follows that

\begin{equation} \label{TRANSFERGEOMBOUND}
\| N_{jk}^{\mathcal G}(h) \|_{L^2 \to L^2}  = O(1) \, \| N_{jk}(h) \|_{L^2 \to L^2}  \end{equation}
and so, from (\ref{CRUCIAL3}), 

\begin{eqnarray} \label{UPSHOTA}
\| \chi_j(h D)  (1 -\psi_{j}^{\delta}(h)) u_h  \|_{L^2(\Gamma_j)} = O(h^{-1/4-0}) + \|N_{j}^{\mathcal D}(h) u_h \|_{L^2(\partial \Omega)}. \end{eqnarray}  \

We are left with bounding the diffractive term on the RHS of (\ref{UPSHOTA}).\\

\subsubsection{ Bounding the diffractive term $\|N_{j}^{\mathcal D}(h)u_h \|$} \label{diffraction}

Here, we simply use that from (\ref{TRANSFERrelate}) and (\ref{geo/diff defn}),

$$ \| N^{\mathcal D}_j(h) u_h \|_{L^2(\Gamma_j)}  \leq \sum_{k \neq j}  \| N_{jk}(h) \|_{L^2 \to L^2}  \, \| \psi_{k}^{2\delta} u_h \|_{L^2}.$$\

Then,  by non-concentration and Sobolev restriction we again get
 that with $\delta = 1/2-0,$ the mass  $ \| \psi_{k}^{2\delta} u_h \| = O(h^{-0})$ and so, in view of (\ref{TRANSFERBOUND}), it follows that
 
 \begin{equation} \label{DIFFBOUND}
 \| N_{j}^{\mathcal D}(h) u_h \|_{L^2} = O(h^{-1/4-0}). \end{equation}\

Consequently, from (\ref{DIFFBOUND}) and  (\ref{UPSHOTA}), the end result is that

\begin{equation} \label{UPSHOT2.0}
\| \chi_j(h D)  (1 -\psi_{j}^{\delta}(h)) u_h  \|_{L^2(\Gamma_j)} = O(h^{-1/4-0}). 
\end{equation}\

 On the other hand, from the small-scale  Rellich commutator result in Lemma \ref{rellich},

\begin{equation} \label{inteqn1.1}
\| [1-\chi_j (hD) ]\, (1-\psi_j^{\delta}(h)) u_h^{j} \|_{L^2(\Gamma_j)} = O(h^{-\delta/2}) = O(h^{-1/4 -0}). \end{equation}\

 So, from (\ref{UPSHOT2.0})  and (\ref{inteqn1.1}), it follows that 

\begin{eqnarray} \label{UPSHOT2.1}
 \|  ( 1 - \psi_j^{\delta}(h)) u_h^j \|_{L^2} = O(h^{-1/4-0}). \end{eqnarray}\
 
We are left with estimating mass near corners ; that is,$ \| \psi_j^{\delta}(h) u_h^{j} \|_{L^2}.$

 \subsection{Estimates near corners}
  
  Here, as in the diffractive case above, we use non-concentration in Theorem \ref{T:non-con} together with Sobolev restriction. Recall that from the interior estimates centered
at a corner $c_{j} \in \overline{\Omega}$ in Theorem \ref{T:non-con}, we have that

\begin{equation} \label{upshot3.5}
 \| \phi_h \|_{L^2( B(c_j, h^{\delta}) )}  = O(h^{\delta/2}) = O(h^{1/4-0}), \quad \delta = 1/2-0. \end{equation}\

 The bound in (\ref{upshot3.5}) combined with $h$-Sobolev estimates give
 
 \begin{equation} \label{upshot4}
\| u_h \|_{ L^2( \{ q \in \Gamma_j; |q-c_j| \leq h^{\delta} \} )}  \lessapprox h^{-1/2} \| \phi_h \|_{L^2(B(c_j, h^{\delta}) )} \lessapprox h^{-1/2 + 1/4 - 0}. \end{equation} \

Thus, it follows from (\ref{upshot4}) that 

\begin{equation} \label{UPSHOT2.2}
\| \psi_j^{\delta}(h) u_h \|_{L^2(\Gamma_j)} = O( h^{-1/4 - 0}). \end{equation}\

 Consequently,  in the obtuse case, Theorem \ref{dirichlet}  then follows from (\ref{UPSHOT2.1}) and (\ref{UPSHOT2.2}).

\end{proof}

\subsection{Proof of Theorem \ref{dirichlet}: the general case} \label{acute}
Assume now that the angle $\alpha_j \in  (0,
\pi/2].$
The analysis here is very similar to the obtuse case, so we only indicate here the relatively minor changes.  The diffractive term $\| N_{j}^{\mathcal D}(h) u_h \|$ is estimated in the same way as in section \ref{diffraction}.   The key difference here is that for the geometric term $ \| N_{j}^{\mathcal G}(h) u_h \|, $    one uses the iterated jumps equation $u_h = N(h)^{\,2} \, u_h$ instead of just $u_h = N(h) u_h,$ since near-glancing rays to the flat edge $\Gamma_j$ get reflected twice. First, they reflect in the adjacent edge $\Gamma_k,$  back to $\Gamma_j$  and then reflect once more along the initial flat edge $\Gamma_j$ (see Figure \ref{F:Fig-2}).      By a similar analysis to   that in the obtuse case in the previous section, one gets
\begin{align} 
 \chi_j(h D) &  (1-\psi_{j}^{\delta}(h)) {\bf 1}_{\Gamma_j} u_h
 \notag \\
& =  \sum_{k=j-1}^{j+1} N_{jk}(h)  \, \zeta_k(hD)  \,(1-\psi_k^{2\delta}(h)){\bf 1}_{\Gamma_k} \, N(h)^{ 2}  \, u_h
+ O( \| N_{jk}(h) \| \cdot  \| \psi_k^{2\delta}(h)) u_k \| )  \nonumber \\ 
& \quad + O(h^{\infty}). \label{acute1} \end{align}

   \begin{figure}
\hfill
\centerline{
\begingroup%
  \makeatletter%
  \providecommand\color[2][]{%
    \errmessage{(Inkscape) Color is used for the text in Inkscape, but the package 'color.sty' is not loaded}%
    \renewcommand\color[2][]{}%
  }%
  \providecommand\transparent[1]{%
    \errmessage{(Inkscape) Transparency is used (non-zero) for the text in Inkscape, but the package 'transparent.sty' is not loaded}%
    \renewcommand\transparent[1]{}%
  }%
  \providecommand\rotatebox[2]{#2}%
  \newcommand*\fsize{\dimexpr\f@size pt\relax}%
  \newcommand*\lineheight[1]{\fontsize{\fsize}{#1\fsize}\selectfont}%
  \ifx\svgwidth\undefined%
    \setlength{\unitlength}{267.13236237bp}%
    \ifx\svgscale\undefined%
      \relax%
    \else%
      \setlength{\unitlength}{\unitlength * \real{\svgscale}}%
    \fi%
  \else%
    \setlength{\unitlength}{\svgwidth}%
  \fi%
  \global\let\svgwidth\undefined%
  \global\let\svgscale\undefined%
  \makeatother%
  \begin{picture}(1,0.64069547)%
    \lineheight{1}%
    \setlength\tabcolsep{0pt}%
    \put(0,0){\includegraphics[width=\unitlength,page=1]{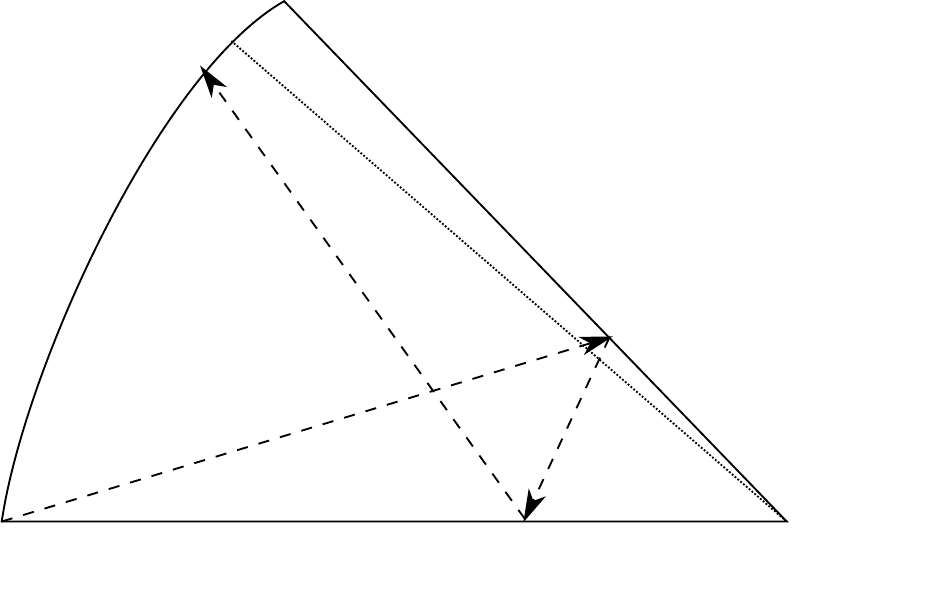}}%
    \put(-0.00221954,0.3948562){\makebox(0,0)[lt]{\lineheight{1.25}\smash{\begin{tabular}[t]{l}$\Gamma_\ell$\end{tabular}}}}%
    \put(0.64051351,0.42944484){\makebox(0,0)[lt]{\lineheight{1.25}\smash{\begin{tabular}[t]{l}$\Gamma_k$\end{tabular}}}}%
    \put(0.30509843,0.00630616){\makebox(0,0)[lt]{\lineheight{1.25}\smash{\begin{tabular}[t]{l}$\Gamma_j$\end{tabular}}}}%
    \put(0,0){\includegraphics[width=\unitlength,page=2]{Fig-2.pdf}}%
    \put(0.15015148,0.09993175){\makebox(0,0)[lt]{\lineheight{1.25}\smash{\begin{tabular}[t]{l}$\epsilon_0$\end{tabular}}}}%
    \put(0,0){\includegraphics[width=\unitlength,page=3]{Fig-2.pdf}}%
    \put(0.88506681,0.03794219){\makebox(0,0)[lt]{\lineheight{1.25}\smash{\begin{tabular}[t]{l}$\alpha$\end{tabular}}}}%
    \put(0.86542791,0.09604655){\makebox(0,0)[lt]{\lineheight{1.25}\smash{\begin{tabular}[t]{l}$c_j$\end{tabular}}}}%
    \put(0.38079554,0.41936158){\makebox(0,0)[lt]{\lineheight{1.25}\smash{\begin{tabular}[t]{l}$L_j$\end{tabular}}}}%
  \end{picture}%
\endgroup%
}
\caption{\label{F:Fig-2} The setup for acute angles.}
\hfill
\end{figure}

Just as in the previous section we consider corner cutoffs $\widetilde{\psi_k}^{2\delta}  \Subset \tilde{\psi_k}^{2\delta} \Subset \psi_k^{2\delta}$ with $\widetilde{\psi}^{2\delta}  = \sum_{k} \widetilde{\psi_k}^{2\delta} , \,\, \tilde{\psi}^{2\delta} = \sum_k \tilde{\psi_k}^{2\delta}$ and make the  decomposition 
 \begin{equation} \label{acute2}
  (1-\psi_k^{2\delta})N(h)^{2} =  (1-\psi_k^{2\delta}) N(h)  \, [  \tilde{\psi}^{2\delta}  + (1-\tilde{\psi}^{2\delta}) ]  \, N(h)  \, [ \widetilde{\psi}^{2\delta}  +    ( 1 - \widetilde{\psi}^{2\delta}  ) ]. \end{equation}
One substitutes (\ref{acute2}) in (\ref{acute1}) and each of the resulting 4 terms is bounded using the same analysis as in subsection \ref{obtuse}: near diagonal terms in the Schwartz kernels where $|q-q'| \lessapprox h^{2\delta}$  are bounded using Schur lemma.   Off-diagonal   terms with $|q-q'| \gtrapprox h^{2\delta}$ are estimated using the h-FIO representation in (\ref{inteqn2}) together with the stationary phase argument in subsection \ref{statphase}. Just as in the obtuse case, the result is that
$$ \|  \chi_j(h D)  (1-\psi_{j}^{\delta}(h)) {\bf 1}_{\Gamma_j} u_h \|_{L^2} = O(h^{-0}) \sum_{k=j-1}^{j+1} \| N_{jk}(h) \|_{L^2 \to L^2}.$$
The bounds $\| N_{jk}(h) \|_{L^2 \to L^2} = O(h^{-1/4} |\log h|^{1/4})$ follow in the same way in the previous section (they do not depend on properties of the angle $\alpha_j$) and that completes the proof of Theorem \ref{dirichlet} in the general case. \qed

\section{Estimates for transversal mass} \label{bdypsdo}


We begin with some preliminaries on $h$-pseudodifferential operators along boundary edges.

\subsection{ Pseudodifferential operators along boundary} \label{bdypsdos}
Before giving the proof, we begin with some background. Let $\Gamma_k \subset \partial \Omega$ be a boundary face (edge) with corner endpoints $c_k$ and $c_{k+1}.$ The face $\Gamma_k$ extends smoothly to an open edge $\Gamma_k' \Supset \Gamma_k$ and  we let $(x',x_n): U_k \to \R^2$ be Fermi coordinates in a tubular neighbourhood $U_k \supset \Gamma_k'$ with $\Gamma_k ' = \{ x_n = 0 \}.$  Let $\rho \in C^{\infty}_{0}(\Gamma_k')$ be a cutoff with $0 \leq \rho \leq 1$ and  $\rho(x') =1$ for $x' \in \Gamma_k.$

\begin{definition} We say that $P(h)\in \Psi_h^{m}(\Gamma_k)$ if $P(h): C^{\infty}_0(\Gamma_k') \to C^{\infty}_0(\Gamma_k')$ is a properly-supported $h$-psdo with Schwartz kernel of the form
$$P(x',y',h) = (2\pi h)^{-1} \int_{\R} e^{i(x'-y')\xi'/h} a(x'.\xi',h) \rho(x') \rho(y') \, d\xi',$$
where $a(x',\xi',h) \in S^{m,-\infty}(T^* \Gamma_k').$ Similarily, when $a \in S^{m,-\infty}_{\delta}$ we write $P(h) \in \Psi_{h,\delta}^{m}(\mathring{\Gamma_k}).$
\end{definition}

We denote the induced boundary Laplacian  on the edge $\Gamma_k$ by $\Delta_k$ where the latter extends to a differential operator $\Delta_k: C^{\infty}_{0}(\Gamma_k') \to C^{\infty}_0(\Gamma_k').$

 In view of the Neumann boundary condition, at a corner point $c_k$ we have $\partial_{\nu_k} \phi_h(c_k) = \partial_{\nu_{k-1}} \phi_h(c_k) =0.$ Since  $\partial_{\nu_k}$ and $\partial_{\nu_{k-1}}$ are linearly independent, it then follows that $c_k$ is critical for $\phi_h$ so that
\begin{equation} \label{extend}
\partial_{x'} u_h (c_k) = 0. 
\end{equation} 

Without loss of generality assume that $x'(c_k) =0$. Then, in view of
(\ref{extend}) it is clear that $u_h$ locally extends (independent of $h$)  to a function, $v_h,$   on $\Gamma_k'$  that is even with respect to the involution $x' \to - x'$ and an analogous statement holds at the other corner $c_{k+1}.$  Since in addition $v_h''(-x') = (-1)^2 v_h''(x') = v_h(x')$, it follows that   $v_h \in C^{2}_{loc}(\Gamma_k').$ Denoting  the corresponding extension by $v_h,$ we set
$$\tilde{u}_h := \rho \cdot v_h \in C^2_{0}(\Gamma_k') \cap L^2(\Gamma_k').$$

Since  the construction of $v_h$ involves two even involutions (one at each corner), by choosing $\Gamma_k'$ sufficiently small,  we can (and will)   assume that
$$ \| \tilde{u}_h \|_{L^2(\Gamma_k')} \leq 3 \| u_h \|_{L^2(\Gamma_k)}.$$

In the  proof of Theorem \ref{dirichlet} (see (\ref{inteqn1.1})), we have used transversal eigenfunction mass estimates $h^{\delta}$ close to corners. We collect the necessary results in the following:

\begin{lemma} \label{rellich} Let $\chi_{j}(hD) \in \Psi_h^{0}(\Gamma_j )$ be the h-psdo glancing cutoff defined in subsection \ref{defn} and $\psi_j^{\delta} \in C^{\infty}_0(\Gamma_j)$ be the spatial corner cutoff in (\ref{cornercutoff}). Then, for any $\delta \in [0,1/2),$ there exists constants $C_{\delta}(\Omega)>0$ and $h_0 >0$  such that for $h \in (0,h_0]$ and any $k=1,...,M,$

$$   \| (I - \chi_{j}(hD)) (1-\psi_j^{\delta})(x',h)  u_h \|_{L^2(\Gamma_j)} \leq C_{\delta}(\Omega) h^{-\delta/2}. $$\

\end{lemma}

\begin{proof}

Choose  Fermi coordinates $(x',x_n): \Omega_j \to   \R^2$ in a tubular
neighbourhood $U_j$  of $\Gamma_j' \supset \Gamma_j$ as above and for any $\delta \in [0,1/2),$ we consider the  test operator   $A_{\delta}(h): C^{\infty}_0(U_j) \to C^{\infty}_0(U_j)$ given by

$$A_{\delta}(h) :=\chi(h^{-\delta} x_n) \cdot  (1-\psi_j^{\delta}(x',h)) \, hD_n.$$

 Since 

$$ \big\{ \, (x_n,x') \in \text{supp} \, \chi(h^{-\delta} \cdot) \cdot (1-\psi_j^{\delta}(\cdot,h) \, \big \}  \cap (\partial \Omega \setminus \Gamma_j) = \emptyset,$$
by the Rellich identity \cite{CTZ}, \

\begin{eqnarray} \label{rellichformula}
\frac{i}{h} \langle [ -h^2 \Delta ,  A_{\delta }(h) ] \phi_h, \phi_h \rangle_{L^2(\Omega)} =\langle (1-\psi_k^{\delta}(x',h) \,  (I + h^2 \Delta_{\Gamma_k})  u_h^k,   u_h^k \rangle_{L^2(\Gamma_k)} \nonumber \\
+ O(h^{1-\delta}) \| u_h \|_{\Gamma_k}^2
 \end{eqnarray}
provided $\partial_{\nu} \phi_h |_{\partial \Omega} = 0.$
In Fermi coordinates,
$$ -h^2 \Delta = (hD_n)^2 + R(x',x_n,hD_{x'}), \,\, R(x',0,hD_{x'}) = -h^2 \Delta_{\Gamma_k}$$
and $R(x,hD_{x'})$ is an $h$-differential operator of order two acting tangentially to the boundary.

As a result,  one can write the commutator matrix elements on the LHS of (\ref{rellichformula}) as a sum:
\begin{eqnarray} \label{comm1}
\frac{i}{h} \langle [ (hD_n)^2 ,  A_{\delta }(h) ]  \phi_h, \phi_h \rangle 
+ \frac{i}{h} \langle [ R(x,hD_{x'}),  A_{\delta }(h) ] \phi_h, \phi_h \rangle \noindent \\
= \frac{i}{h} \langle [ (hD_n)^2 ,  \chi(h^{-\delta} x_n) ]   \cdot (1-\psi_k^{\delta}(x',h) ) hD_n\phi_h, \phi_h \rangle  \nonumber \\
+ \frac{i}{h} \langle [ R(x,hD_{x'}) , (1-\psi_k^{\delta}(x',h) ) hD_n ]  \chi(h^{-\delta} x_n) \cdot \phi_h, \phi_h \rangle \end{eqnarray}

Since $hD_n \chi(h^{-\delta} x_n) = h^{-\delta} \chi'(h^{-\delta})  \in h^{-\delta} S_{\delta}^{0},$ it follows that 
$$  \frac{i}{h} \big\langle [ (hD_n)^2 ,  \chi(h^{-\delta} x_n) ]   \cdot (1-\psi_j^{\delta}(x',h) ) hD_n\phi_h, \phi_h \big\rangle   = O(h^{-\delta}).$$
Moreover, we note that the symbol of  $[ (hD_n)^2 ,  \chi(h^{-\delta} x_n) ]   \cdot (1-\psi_j^{\delta}(x',h) )$ is supported in the $h^{\delta}$ strip  where $ |x_n| \lessapprox h^{\delta}$ and, in general, non-concentration is of no use in such strips (only $h^{\delta}$ balls). 
As for the second term in the last line of (\ref{comm1}), the non-standard terms arise from hitting the cutoff $(1-\psi_j^{\delta}(x',h))$ with $D_{x'}.$ Since
$$D_{x'} (1-\psi_j^{\delta}(x',h))  = - h^{-\delta} \partial \psi_j^{\delta \,  }  (x',h) \in h^{-\delta} S^{0}_{\delta}$$
and 
$$ \text{supp} \, \partial \psi_j^{\delta}(\cdot;h) \psi(h^{-\delta} \cdot) \subset \{ (x',x_n);  |x'| \lessapprox h^{\delta}, \,\, |x_n| \lessapprox h^{\delta} \}.$$\

So, by $L^2$-boundedness,

$$ \frac{i}{h} \langle [ R(x,hD_{x'}) , (1-\psi_j^{\delta}(x',h) ) hD_n ]  \chi(h^{-\delta} x_n) \cdot \phi_h, \phi_h \rangle  = O(h^{-\delta}) \| \phi_h \|^2_{ \{ |(x',x_n) | \lessapprox h^{\delta} \} } = O(1),$$
where the final estimate follows by non-concentration (note that estimates on $h^{\delta}$ balls  are equivalent to estimates on $h^{\delta}$-cubes). All other commutator terms are $O(1)$ by standard $L^2$ results.   So,  after absorbing the error term, it then follows from Theorem \ref{T:non-con}   that

\begin{equation} \label{upshot1}
 \langle  (1-\chi_j(h^{-\delta}x')) (I+h^2 \Delta_j)  h^2 \Delta_j  u_h, u_h \rangle_{L^2(\Gamma_j)} = O(h^{-\delta}).
\end{equation}\

Running the same argument as above with the test operator

$$\tilde{A}_{\delta}(h) = (1-\chi(h^{-\delta}x')) \chi_{\epsilon}(x_n) (-h^2 \Delta_j(x',hD)) hD_{x_n}$$
gives

\begin{equation} \label{upshot2}
 \langle  (1-\chi_j(h^{-\delta}x')) (I+h^2 \Delta_j)  h^2 \Delta_j  u_h, u_h \rangle_{L^2(\Gamma_j)} = O(h^{-\delta}). \end{equation}\
 
 Consequently, setting
 $$P(x',hD) = (1-\chi_j(h^{-\delta}x')) (I+h^2 \Delta_j)^2 :C^{\infty}(\mathring{\Gamma_j}) \to C^{\infty}(\mathring{\Gamma_j}),$$
  it follows by adding (\ref{upshot1}) and (\ref{upshot2}) that
 
\begin{equation} \label{UPSHOT1}
\langle P(x',hD) u_h, u_h \rangle_{L^2(\Gamma_j)} = O(h^{-\delta}). \end{equation}\

Let $\tilde{\chi} \in C^{\infty}_0(\Gamma_j')$ with
$\tilde{\chi}|_{\Gamma_j} =1.$  Abusing notation somewhat,  we extend $P(h)$ as an operator $P(h): C^{\infty}_0(\Gamma_j') \to C^{\infty}_0(\Gamma_j')$ so that $P(h) \in \Psi^0_{h,\delta}(\Gamma_j).$ Let $\chi_+ \in C^{\infty}(\R;[0,1])$ with $ \chi_+ \geq 0$ so that $\chi_+(u) = 0$ for $u \leq 1/4$ and $\chi_+(u)= 1$ for $u \geq 1/2.$ Consider the $h$-psdo $Q(h) \in \Psi_{h,\delta}^{0)}(\Gamma_j)$ given by

$$ Q(h):= \tilde{\chi} \,  P(h)  \chi_+ (P)^* \chi_+(P) \, \tilde{\chi}.$$\

We apply sharp Garding to the operator $Q(h)$ in two ways:  First,  note   that
$$\sigma(Q) = p \, |\chi_+(p)|^2 \geq 0, \quad p(x',\xi') = |\xi'|^2_{x'} - 1$$
and  since $p |\chi_+(p)|^2 \leq p,$ sharp Garding and (\ref{UPSHOT1}) gives

\begin{equation} \label{two}
\langle \tilde{\chi}  P(h)  \chi_+ (P)^* \chi_+(P)  \tilde{\chi} \tilde{u}, \tilde{u} \rangle_{\Gamma_j'} \lessapprox \langle P(h)  u, u \rangle_{\Gamma_j}   + O(h) \| u \|_{\Gamma_j}^2= O(h^{-\delta}).
\end{equation}\

 Next,  apply sharp Garding yet again using $p  \geq 1/4 $ on supp $\chi_+(p)$ to get

\begin{equation} \label{three}
\langle \tilde{\chi}  \chi_+ (P)^* \chi_+(P) \tilde{\chi}  \tilde{u}, \tilde{u} \rangle_{\Gamma_j'}  \lessapprox  \langle \tilde{\chi}  P(h)  \chi_+ (P)^* \chi_+(P) \tilde{u}, \tilde{u} \rangle_{\Gamma_j'}  + O(h) \|u \|^2_{\Gamma_j} = O(h^{-\delta}),\end{equation}
where the last bound in (\ref{three}) follows from (\ref{two}). Finally, note by non-negativity and the bound in (\ref{two}),

\begin{equation} \label{UPSHOT2}
 \langle \chi_+(P)^* \chi_+(P) u, u \rangle_{\Gamma_j} \leq \langle \tilde{\chi}  P(h)  \chi_+ (P)^* \chi_+(P) \tilde{u}, \tilde{u} \rangle_{\Gamma_j'}  = O(h^{-\delta}). \end{equation}

Consequently from (\ref{UPSHOT2}), by taking square roots,

\begin{equation} \label{guts2}
\| (1-\chi_j(h^{-\delta} x') \cdot (I-\chi_j^{\delta}(hD)) u_h \|_{L^2(\Gamma_j)} = O(h^{-\delta/2}) \end{equation}
and that finishes the proof of  Lemma \ref{rellich}. \end{proof}





\end{document}